\RequirePackage{fix-cm}
\documentclass{svjour3}                                    % onecolumn (standard format)
\smartqed  % flush right qed marks, e.g. at end of proof
\usepackage{latexsym}
\usepackage{graphicx}
\usepackage{amsmath, amsfonts, amssymb, amsbsy}
\usepackage{mathtools}
\usepackage{color}
\usepackage{lscape}
\usepackage{algorithm}               % format of the algorithm
\usepackage{algorithmic}            % format of the algorithm
\usepackage{longtable}
\usepackage{multirow}
\usepackage{tikz}   
\usetikzlibrary{shapes,shadows,arrows}
\usepackage{enumitem}
\usepackage{hyperref}
\usepackage{colortbl}
\usepackage[stable]{footmisc}

\newcommand{\Real}{\mathbb{R}}

\newcommand{\Basis}{\mathcal{B}}

\newcommand{\Ninf}{\mathcal{N}_{-\infty}(\gamma)}   			% large neighbourhood of central path
\newcommand{\Ninfp}{\mathcal{N}_{-\infty}(\gamma,\lambda)}   	% large neighbourhood of perturbed central path
\newcommand{\sfsp}{\mathcal{F}^{0}_{\lambda}}   				% strictly feasible set for perturbed problems

\newcommand{\sspp}{\Omega^{P}_{\lambda}}		% solution set of the primal problem of the perturbed problems 
\newcommand{\ssdp}{\Omega^{D}_{\lambda}}		% solution set of the dual problem of the perturbed problems
\newcommand{\sspo}{\Omega^{P}}				% solution set of the primal problem of the original problems 
\newcommand{\ssdo}{\Omega^{D}}				% solution set of the dual problem of the original problems 

\newcommand{\Actv}{\mathcal{A}}				% active set
\newcommand{\Iactv}{\mathcal{I}}				% inactive set
\newcommand{\Nactv}{\mathcal{Z}}                           % not decided set
\newcommand{\Sactv}{\Actv_{+}}				% strictly active set
\newcommand{\psas}{\bar{\Actv}_{+}}			% predicted strictly active set for LP
\newcommand{\pacs}{\bar{\Actv}}				% predicted active set for LP

\newcommand{\pntp}{(x,y,s)}					% point for the perturbed problem

\newcommand{\optSol}{\left( x^{*},y^{*},s^{*} \right)} 
\newcommand{\optSollambda}{\left( x^{*}_{\lambda},y^{*}_{\lambda},s^{*}_{\lambda} \right)}
\newcommand{\optSolp}{\left( \hat{p},\hat{y},\hat{q} \right)}
\newcommand{\mup}{\mu_{\lambda}}

\newcommand{\lp}{{\sc lp}}
\newcommand{\ipm}{\sc ipm}
\newcommand{\ipms}{{\sc ipm}s}
\newcommand{\hsd}{\sc hsd}

\newcommand{\lcp}{{\sc lcp}}
\newcommand{\kkt}{{\sc kkt}}
\newcommand{\iterKp}{(x^{k},y^{k},s^{k})}

\newcommand{\tspndeg}{\hyperref[test-random_set]{TS1}}
\newcommand{\tspddeg}{\hyperref[test-random_pd_degen_set]{TS2}}
\newcommand{\tsn}{\hyperref[test-netlib_set]{TS3}}

\newcommand{\vmin}[2]{%
\min\left\{ [ #1 \,\, #2 ] \right\}
}%
\newcommand{\cmin}[2]{%
\min\left\{#1,#2\right\}
}%
\newcommand{\cmax}[2]{%
\max\left\{#1,#2\right\}
}%

\newcommand{\Remark}{%
\paragraph{Remark.}
}

\newcounter{algo}[section]

\newcommand{\algo}[3]{\refstepcounter{algo}
\begin{center}\begin{figure}[htbf]
\framebox[\textwidth]{
\parbox{0.95\textwidth} {\vspace{1.5ex}
{\bf Algorithm \thealgo : #2}\label{#1}\\
\vspace*{-3ex} \mbox{ }\\
{#3} \vspace{1.5ex} }}
\end{figure}\end{center}}

\begin{document}
%%%%%%%%%%%%%%%%%%%%% title page %%%%%%%%%%%%%%%%%%%%%%%%%%%%%%

\title{Active-set prediction for interior point methods using controlled perturbations}

\titlerunning{Active-set prediction for {\sc ipm}s using controlled perturbations}        % if too long for running head

\author{Coralia Cartis \and Yiming~Yan}

\institute{Coralia Cartis \at
	Mathematical Institute, 
	University of Oxford, 
	Andrew Wiles Building, 
	Radcliffe Observatory Quarter, 
	Woodstock Road, 
	Oxford, OX2 6GG,
	United Kingdom.   \\
	\email{cartis@maths.ox.ac.uk}  
	\and
	Yiming Yan \at
        School of Mathematics, 
	University of Edinburgh,
        James Clerk Maxwell Building,
         The King's Buildings, 
         Mayfield Road,
	Edinburgh, EH9 3JZ, 
        United Kingdom.  \\
        \email{yiming.yan@ed.ac.uk} \\
        This author was supported by the Principal's Career Development Scholarship from the University of Edinburgh.
}

\date{Received: date / Accepted: date}

% \date{April 15, 2014; Revised March 25, 2015 and September 5, 2015}

\maketitle

\renewcommand{\thefootnote}{\arabic{footnote}}

%%%%%%%%%%%%%%%%%%%%% Abstract %%%%%%%%%%%%%%%%%%%%%%%%%%%%%%
\begin{abstract}
We propose the use of controlled perturbations to address the challenging question of optimal active-set prediction for interior point methods. Namely, in the context of 
 linear programming, 
 we consider perturbing the inequality constraints/bounds so as to enlarge the feasible set.
We show that if the perturbations are chosen appropriately, the solution of the original problem lies on or close to the central path of the perturbed problem. We also find that a primal-dual path-following algorithm applied to the perturbed problem is able to accurately 
predict the optimal active set of the original problem when the duality gap for the perturbed problem is not too small;
furthermore, depending on problem conditioning, 
this prediction can happen sooner than predicting the active set for
the perturbed problem or when the original one is solved. 
Encouraging preliminary numerical experience is reported when comparing activity prediction for the perturbed and unperturbed problem formulations.
\end{abstract} 

%%%%%%%%%%%%%%%%%%%%% Main part %%%%%%%%%%%%%%%%%%%%%%%%%%%%%%

%%%%%%%%%%% Section: Introduction %%%%%%%%%%%%%%
\setcounter{page}{1}
\section{Introduction}    %% Update at 10.40pm on 11 April 2014
\label{sec-intro}
Optimal active-set prediction --- namely, identifying the active inequality constraints at the solution of a constrained optimization problem --- plays an important role
in the optimization process by removing the difficult combinatorial aspect of the problem and reducing it to an equality-constrained one that is in general easier to solve.
Active-set prediction is also crucial for efficient warmstarting and re-optimization capabilities of algorithms when a suite of closely related problems needs to be solved. 
Despite being state-of-the-art tools for solving large-scale Linear Programming ({\lp}) problems \cite{wright},
Interior Point Methods ({\ipms}) are well-known to encounter difficulties with active-set prediction due essentially to their construction. 
They generate iterates that progress towards the solution set through the (relative) interior of the feasible set, and thus avoid visiting possibly-many feasible vertices. 
This however, may also prevent {\ipms} from getting accurate information about the optimal active set early enough during their running.
When this information is more readily predictable/available towards the end of a run, as the iterates approach the solution set, the algorithm  has to solve increasingly ill-conditioned
and hence difficult, subproblems.  Finding ways to improve (even just partial) active set prediction for {\ipms} could thus be beneficial as it would allow earlier termination
of an otherwise ill-conditioned and computationally expensive process by say, projecting onto the solution set (as in finite termination~\cite{ye}), 
help with reducing the problem size or with obtaining a vertex solution  at the cost of just a few additional (and less expensive) simplex method~iterations.

%% State-of-art methods 
%   Literature review for active-set prediction 
Various ways have been devised for {\ipms} to predict the optimal
active set during their run, with the simplest being {\it cut-off}
\cite{gill1986,Karmarkar1991,McShane} --- which splits the variables
into active or inactive based on whether they are less than a
user-defined small value --- and the most well-known being {\it
  indicators} \cite{bakry} which form functions of
 iterates and identify the optimal active-set based on whether the values of these functions are less than a threshold. 
Mehrotra~\cite{Metrotra1991tr} suggests determining the active set by
a simple comparison of the relative increments of primal and dual
iterates, and Mehrotra and Ye~\cite{Mehrotra_Ye_1993} propose a
strategy to identify the active set by comparing the  primal variables
with the dual slacks; see~\cite{williams} for a review of active-set
prediction techniques for {\ipms} for {\lp} and
also~\cite{Oberlin2006} for a more recent survey.

Here we propose the use of {\it controlled perturbations}~\cite{cartis} for active-set prediction for {\ipms}.\footnote{Note that
\cite{cartis} proposed the use of such perturbations for creating a
sequence of LPs with strict interior, converging to the original LP in
the limit, so as to find the affine dimension of the feasible set of
the original LP and 
well-centred points in  Phase I of {\ipms}; a different focus and
approach than here. A relaxation technique for Mathematical
Programs with Equilibrium Constraints (MPECs) was also proposed independently
in \cite{demiguelMPEC} that relaxes the bound constraints  similarly to~\cite{cartis}, 
but that also relaxes the complementarity constraint, in order to
create a sequence of
nonlinear programming relaxations with strict interior, which thus
satisfy a constraint qualification, and also converge to the original MPEC
in the limit.} Namely, 
we perturb the inequality constraints of the {\lp} problem (by a small amount) 
so as  to enlarge the feasible set of the problem, then solve the resulting perturbed problem(s) using a path-following {\ipm} while predicting
on the way the active set of the original {\lp} problem. As Figure~\ref{fig-actvPredictDemo} illustrates, provided the perturbations are chosen judiciously, 
the central path of the perturbed problem may pass close to the optimal solution of the
original {\lp} problem when the barrier parameter for the perturbed problem is `not too small'.  Thus we expect that  while still    `far' from optimality for the perturbed problem,
some IPM iterates for the perturbed problem would  nonetheless be close to optimality for the original {\lp} problem (such as the third and fourth iterate  in Figure~\ref{fig-actvPredictDemo})
and would provide
a good prediction of the original optimal active set.
As it may happen that the chosen perturbations are `too large' or not sufficiently effective for active-set
prediction, we allow them to shrink after each IPM iteration so that the resulting perturbed feasible set is smaller but still contains the feasible set of the original {\lp}.
 
Since we employ perturbed problems, albeit artificially, our proposal  may be remindful of warmstarting techniques for {\ipms} and the related active-set prediction techniques
that have been developed in that context; see for example, the surveys \cite{Engau2010,Skajaa2013}. Thus we briefly review relevant contributions here.
One of the main warmstarting strategies focuses on the `iterates', namely it manipulates the ({\ipm}-computed) near optimal or optimal iterates 
of the initial problem to obtain a primal-dual feasible and well-centred point for the  perturbed problems, 
see for example,~\cite{Gondzio1998,Gondzio1999,Yildirim00,Gondzio2008,Skajaa2013}. Another category of approaches works on the `problem formulation', namely modify the problem formulation by relaxing the nonnegativity constraints in the form of shifted logarithmic barrier variables, which has some similarity to our approach. 
Earlier works in this framework include Freund~\cite{Freund1991,Freund199119,Freund1996}, Mitchell~\cite{Mitchell1994} and Polyak~\cite{Polyak1992} with promising theoretical properties.
More relevant and closer in spirit to our approach here is~\cite{Benson2007}, where Benson and Shanno propose a primal-dual penalty strategy relaxing the nonnegativity constraints for both primal and dual decision variables and then penalising the relaxation variables in the objective; encouraging numerical results are also reported. Engau, Anjos and Vannelli~\cite{Engau2009,Engau2010} apply a simplified primal-dual slack approach: instead of shifting the bounds and penalising the relaxation variables, slack variables for nonnegative constraints are introduced and penalised in the objective.  
One of the main differences between the above techniques and our approach 
is that we consider perturbations as parameters, not variables that are updated in the run of the {\ipm}; furthermore, 
our focus  is different as we specifically aim to  predict the active set of the original {\lp} problem by using  `fake' perturbations.

Another set of techniques --- regularization for {\ipm}s~\cite{saunders1996solving,altman1999regularized,castro2010existence}
--- is also only loosely connected to our approach.
In order to improve the conditioning of the coefficient matrix arising
in calculating Newton directions in {\ipm} iterations, regularization
terms (of proximal type, weighted, and quadratic in the variables) are
added to the (primal and dual) objective function. These terms result
in a diagonal perturbation of the linear KKT  system of interest,
improving stability of factorization procedures. Note that the effect
of our perturbations on the Newton system is not the same in that no similar
diagonal perturbation is obtained. This is due to our formulations having no quadratic terms
in the variables in the primal-dual objective, only  a quadratic term in the
perturbations; and to our approach perturbing
the inequality constraints of the problem and allowing negative components
of the primal and dual slack variables. However, the two techniques
have similar aims in that they attempt to deal with the increasing
ill-conditioning that affects {\ipm}s by improved early active-set
prediction (hence earlier termination and better conditioning) for our
approach and by directly
improving the conditioning of the linear algebra through
regularization.

%% Our idea

% Literature review for warmstarting IPMs

%% Our idea
\begin{figure}[ht]
	\centering
\includegraphics[width=0.9\textwidth]{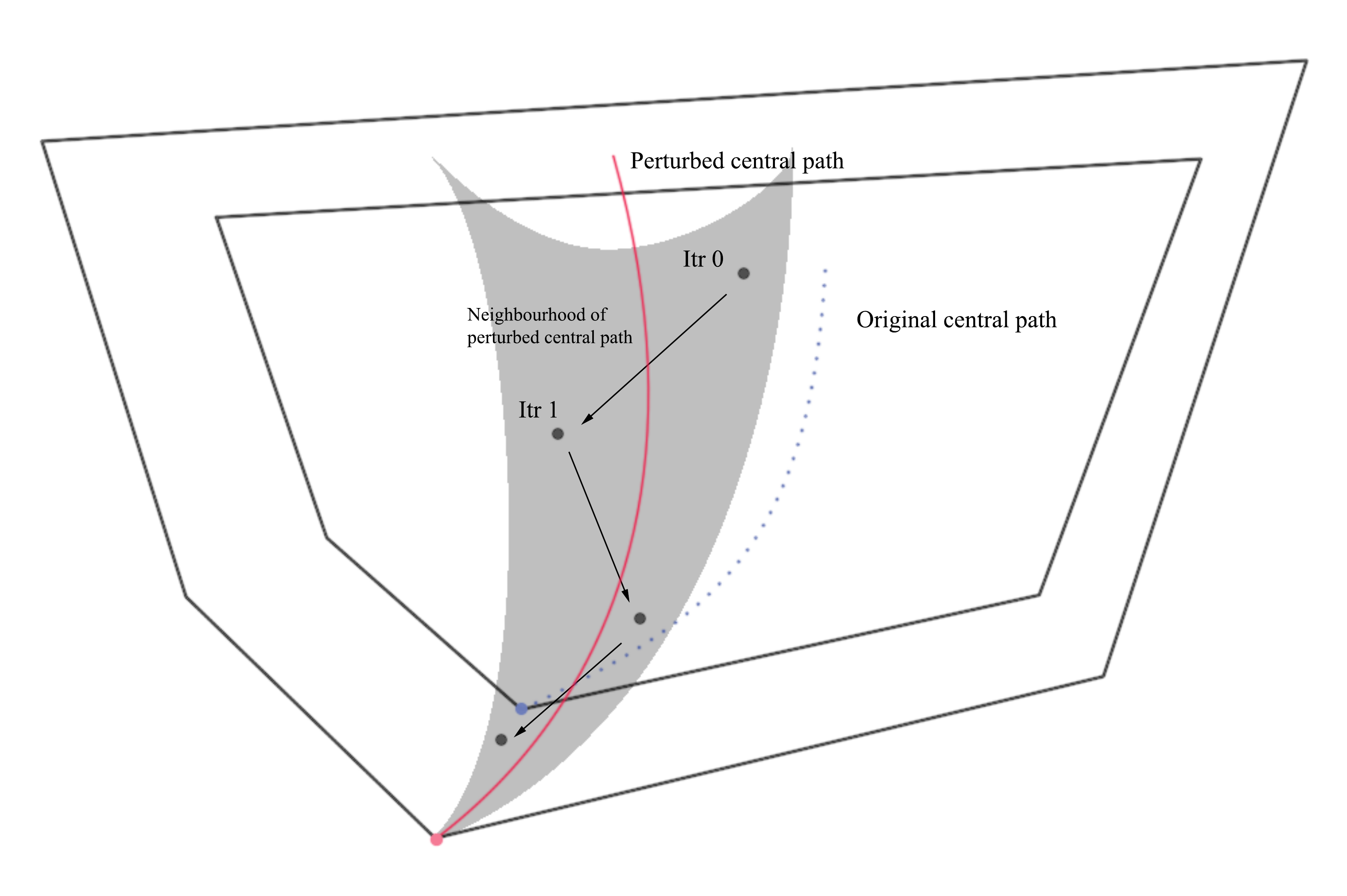}
	\caption{Enlarge the feasible set and predict the original active set}
	\label{fig-actvPredictDemo}
\end{figure}

To validate the use of controlled perturbations for active-set prediction for {\ipms}, after introducing them and the associated primal-dual perturbed {\lp} problems (Section 2),
we show that, for properly chosen perturbations, the solution of the original {\lp} problem lies on or 
close to the central path of the perturbed problems. 
Furthermore, in nondegenerate cases, the optimal active set of the perturbed problems remains the same as for the original problem (Section 3). 
We also prove that under certain conditions that do not necessarily require problem nondegeneracy, our predicted active sets provide  inner and outer approximations for the optimal active set of the original problem, and exactly
predict this set under a certain nondegeneracy assumption (but without requiring that the perturbed active set coincides with the original one).  We also find conditions on problem conditioning
that ensure that our prediction of the optimal
active set of the original {\lp} can happen sooner than the prediction of the optimal active set of the perturbed problems (so that our approach may not need to
solve the perturbed problems to high accuracy) (Section 5.1). Similarly, we characterise the situations when  our approach allows an earlier prediction of the original active
set as compared to the case when we solve and predict the original {\lp} directly (Section 5.2). In our preliminary numerical experiments (Section 6), we carry out two type of tests,
one comparing the accuracy of the predicted active sets and the other one exploring the case of crossover to simplex method.
% correction ratio
For verifying the accuracy of our active-set predictions, we apply an infeasible primal-dual path-following {\ipm} to perturbed and original randomly-generated {\lp} problems, terminate the algorithm at  various iterations and compare the accuracy of predictions using certain correction comparison ratios. We observe that when using perturbations, the precision of our predictions 
is generally higher --- namely, more than 4 times higher at certain iterations --- than that when we do not use perturbations. 
% crossover 
When crossing over to simplex method, we test the efficiency of our active-set predictions by comparing the number of simplex iterations needed 
to solve the original problem to optimality, after some initial {\ipm} iterations. We conduct this test on both randomly generated problems and a subset of Netlib problems. We find that when using perturbations for the {\ipm} iterations, we can save (on average) over $30\%$ simplex iterations compared to the case of not using any perturbations before cross-over to simplex.

%%%%%%%% Section: Controlled perturbations for linear programming problem %%%%%%%%%%%

\section{Controlled perturbations for linear programming}
\label{prel}
%\resetcounters
Consider the following pair of primal-dual linear programming ({\lp}) problems,
\begin{equation}
\label{eqv:originalPD}
\tag{PD}
    \begin{array}{lll}
        \mbox{  (Primal)} &\hspace*{.5cm} & \mbox{  (Dual)}\\
        \begin{array}{cl}
            \displaystyle  \min_{x \in \Real^{n}} & c^{T}x\\
            \displaystyle \mbox{s.t.}& Ax = b,\\
                       & x \geq 0,
                \end{array} &
                \hspace*{1cm} &
                \begin{array}{cl}
                    \displaystyle \max_{(y,s) \in \Real^{m} \times \Real^{n}} & b^{T}y\\
                    \displaystyle \mbox{s.t.} & A^{T}y + s = c,\\
                                & s \geq 0,
                \end{array}
        \end{array}
\end{equation}
where $A \in \Real^{m \times n}$, $b \in \Real^{m}$ and $c \in \Real^{n}$ with $m \leq n$ are problem data, and $\pntp \in \Real^{n} \times \Real^{m} \times \Real^{n}$.

We enlarge the feasible set of this~\eqref{eqv:originalPD} problem by using \textit{controlled perturbations}~\cite{cartis}, namely we relax the  nonnegativity constraints in~\eqref{eqv:originalPD} and consider the pair of perturbed problems,
\begin{equation}
\label{eqv:perPD}
\tag{PD$_{\lambda}$}
  \begin{array}{lll}
    \mbox{(Primal)} & \hspace*{.5cm} & \mbox{(Dual)}\\
    \begin{array}{cl}
        \displaystyle \min_{x \in \Real^{n}}    & {(c+\lambda)^{T}(x+\lambda)}\\
        \mbox{s.t.} & Ax = b,\\
                    & x \geq -\lambda,
    \end{array} &
    \hspace*{1.cm} &
    \begin{array}{cl}
       \displaystyle  \max_{(y,s)  \in \Real^{m} \times \Real^{n} } & {(b+A\lambda)^{T}y}\\
        \mbox{s.t.}  & A^{T}y + s = c,\\
                     & s \geq -\lambda,
    \end{array}
  \end{array}
\end{equation}
for some vector of perturbations $\lambda \geq 0$. (Note that different perturbations for $x$ and $s$ could be used, but for simplicity, we use the same vector of perturbations for both.)
It can be checked~\cite{cartis} that the two problems in~\eqref{eqv:perPD} are dual to each other. Note that if $\lambda \equiv 0$,~\eqref{eqv:perPD} coincides with~\eqref{eqv:originalPD}.
We denote the set of \textit{strictly feasible points} of~\eqref{eqv:perPD},
\begin{equation}
\label{eqv:strictlyFeasibleSet_per}
	\sfsp = \left\{\pntp \,\middle| \, Ax=b, \, A^{T}y+s = c, \, x+\lambda>0, \, s+\lambda>0 \,\right\}.
\end{equation}

Writing down the first order optimality conditions ({\kkt} conditions) for~\eqref{eqv:perPD}, according for example to~\cite[Theorem 12.1]{Nocedal2006}, we find that $\optSollambda$ is a \textit{(primal-dual) solution} for~\eqref{eqv:perPD} if and only if it satisfies the following system,
\begin{equation}
\label{eqv:kkt_perturbed}
	\begin{array}{rcl}
		Ax & =& b,\\
		A^{T}y+s & =& c, \\
		(X+\Lambda)(S+\Lambda)e & = & 0,\\
		(x+\lambda,s+\lambda) & \geq & 0,
	\end{array}
\end{equation}
where $\Lambda = \text{diag}(\lambda)$, $X = \text{diag}(x)$, $S = \text{diag}(s)$ and $e = [1 \,\, \ldots  \,\, 1]^{T}$. Again if $\lambda \equiv 0$ in~\eqref{eqv:kkt_perturbed}, we recover the optimality conditions for~\eqref{eqv:originalPD}. 

\paragraph{Equivalent formulation of~\eqref{eqv:perPD}.}

Letting $p = x + \lambda$ and $q = s + \lambda$, we can write~\eqref{eqv:perPD} in the equivalent form, 
\begin{equation}
\label{eqv:perPD_formPQ}
    \begin{array}{lll}
        \mbox{  (Primal)} &\hspace*{0.5cm} & \mbox{  (Dual)}\\
        \begin{array}{cl}
            \min_{p} & c_{\lambda}^{T}p\\
            \mbox{s.t.}& Ap = b_{\lambda},\\
                       & p \geq 0,
                \end{array} &
                \hspace*{1cm} &
                \begin{array}{cl}
                    \max_{(y,q)} & b_{\lambda}^{T}y\\
                    \mbox{s.t.} & A^{T}y + q = c_{\lambda},\\
                                & q \geq 0,
                \end{array}
        \end{array}
\end{equation}
where $c_{\lambda} = c+\lambda$, $b_{\lambda} = b+A\lambda$ and $\lambda \geq 0$.
{\kkt} conditions ensure that $(  p^{*}_{\lambda},y^{*}_{\lambda},q^{*}_{\lambda} )$ is the (primal-dual) solution of~\eqref{eqv:perPD_formPQ} if and only if it satisfies
\begin{equation}
\label{eqv:kkt_perturbed_form2}
    \begin{array}{rcl}
        Ap & = & b_{\lambda},\\
        A^{T}y + q & = & c_{\lambda},\\
        PQe & = & 0,\\
        (p,q) & \geq & 0,
    \end{array}
\end{equation}
where $P = \text{diag}(p)$ and $Q = \text{diag}(q)$.
It is easy to show that $\optSollambda$ is a~\eqref{eqv:perPD} solution if and only if $(  p^{*}_{\lambda},y^{*}_{\lambda},q^{*}_{\lambda} )$, where $p^{*}_{\lambda} = x^{*}_{\lambda} + \lambda$ and $q^{*}_{\lambda} = s^{*}_{\lambda} + \lambda$, is a solution of~\eqref{eqv:perPD_formPQ}.
Thus we can construct an optimal solution for~\eqref{eqv:perPD} from an optimal solution of~\eqref{eqv:perPD_formPQ} and vice versa.

\paragraph{The central path of~\eqref{eqv:perPD}.}
Following~\cite[Chapter 2]{wright}, we derive the central path equations for \eqref{eqv:perPD} to be
\begin{equation}
\label{eqv:perturbedCentralPath}
	\begin{array}{rl}
		Ax & =b,\\
		A^{T}y+s & =c, \\
		(X+\Lambda)(S+\Lambda)e & = \mu \, e,\\
		(x+\lambda,s+\lambda) & > 0,
	\end{array}
\end{equation}
where $\mu>0$ is the barrier parameter for the perturbed problem~\eqref{eqv:perPD}.
The central path of~\eqref{eqv:perPD} is well defined under mild assumptions, including
\begin{equation}
\label{asm-full_row_rank}
%\tag{Assumption}
\text{\bf Assumption: } \hspace*{15ex}	A  \text{ has full row rank } m. \hspace*{15ex} 
\end{equation}

\begin{lemma}[{\cite[Lemma 5.1]{cartis}}]
\label{prop-perturbedCentralPathWellDefined}
Let~\eqref{asm-full_row_rank} hold and $\lambda \geq 0$. Then the central path of the perturbed problem \eqref{eqv:perPD} is well defined, namely, the system~\eqref{eqv:perturbedCentralPath} has a unique solution for each $\mu>0$, provided $\sfsp$ in \eqref{eqv:strictlyFeasibleSet_per} is nonempty. In particular, if $\lambda > 0$,  $\sfsp$ is nonempty whenever~\eqref{eqv:originalPD} has a nonempty primal-dual feasible~set.
\end{lemma}

Note that if $\lambda > 0$, the condition required for the existence of the perturbed central path is weaker than that for the central path of \eqref{eqv:originalPD}. The latter requires~\eqref{eqv:originalPD} to have a nonempty \textit{strictly} feasible~set, namely, for there to be \eqref{eqv:originalPD} feasible points that strictly satisfy all problem inequality constraints.

%%%%%%%%%%%%%%% Section:  perturbed problem and their properties %%%%%%%%%%%%%%%%

\section{Perturbed problems and their properties}
\label{sec-properties}
%\resetcounters
\subsection{Perfect and relaxed perturbations}

Geometrically, the original optimal solution $\optSol$ of~\eqref{eqv:originalPD} may lie on  or near the central path of the perturbed problem~\eqref{eqv:perPD} for carefully chosen perturbations; see Figures~\ref{fig-perfectPerturbations} and~\ref{fig-relaxedPerturbations}. Algebraically, this happens if $\optSol$ satisfies the third relation in~\eqref{eqv:perturbedCentralPath} exactly or approximately.
We make these considerations precise in the next two theorems. 
\begin{figure}[h]
\begin{minipage}[b]{0.48\linewidth}
\centering
\includegraphics[width=\textwidth]{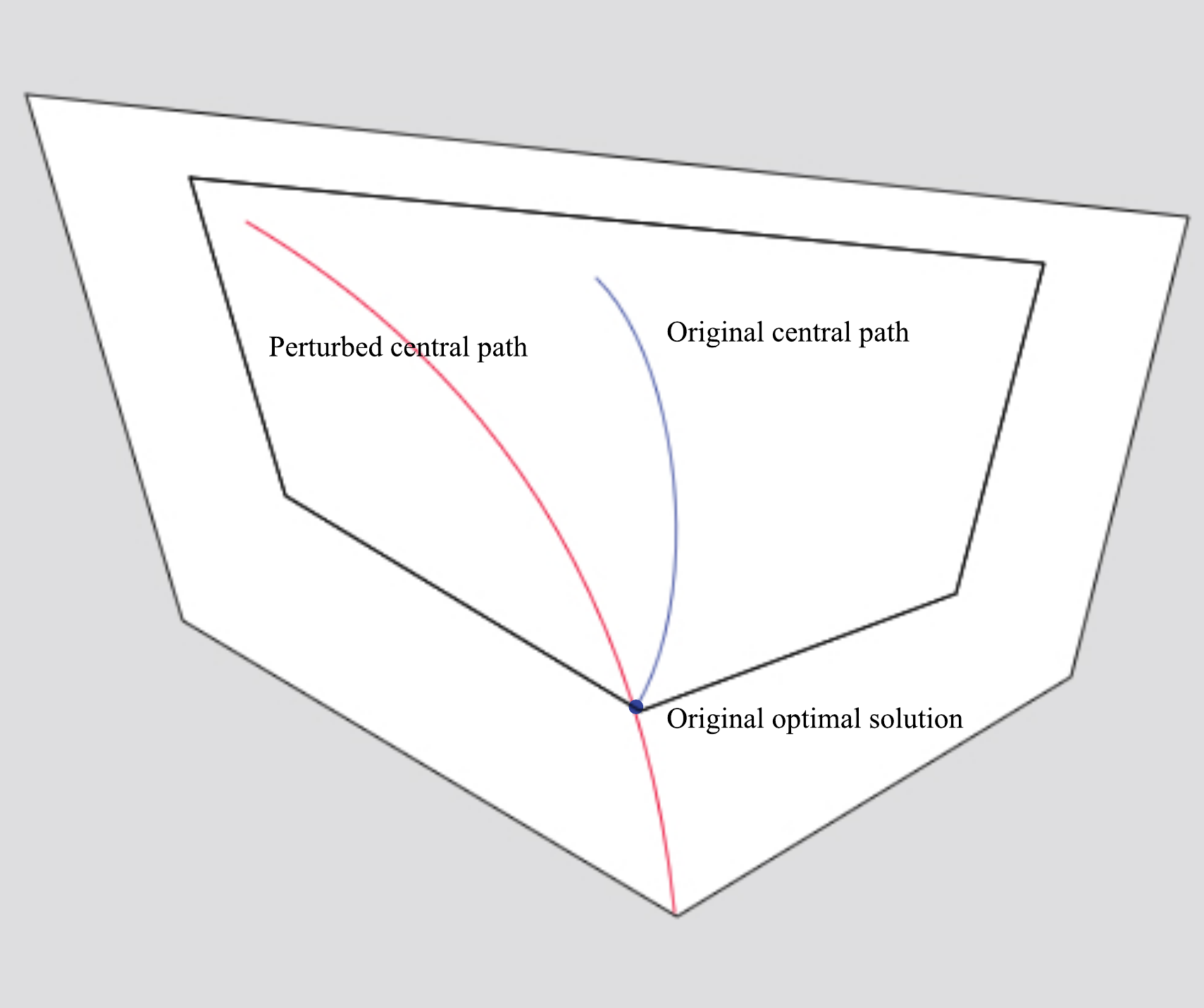}
\caption{Perfect perturbations.}
\label{fig-perfectPerturbations}
\end{minipage}
\hspace*{0.03\linewidth}
\begin{minipage}[b]{0.48\linewidth}
\centering
\includegraphics[width=0.99\textwidth]{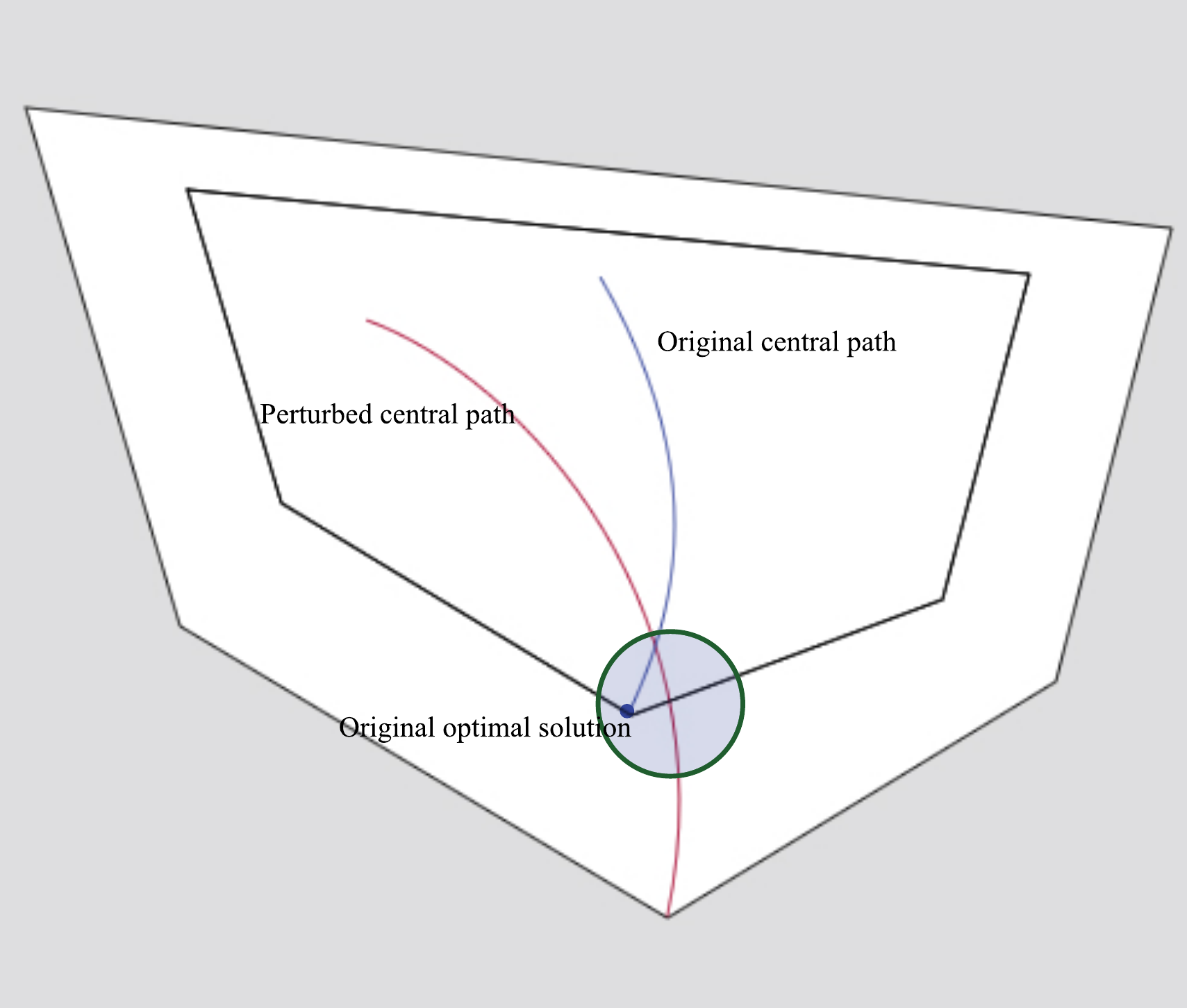}
\caption{Relaxed perturbations.}
\label{fig-relaxedPerturbations}
\end{minipage}
\end{figure}

\begin{theorem}[Existence of `perfect' perturbations]
\label{thm-optsolOnperturbedCentralpath}
Assume~\eqref{asm-full_row_rank} holds and $\optSol$ is a solution of \eqref{eqv:originalPD}. Let $\hat{\mu}>0$. Then there exists a vector of perturbations
$$\hat{\lambda} = \hat{\lambda}(x^{*},s^{*},\hat{\mu})>0,$$
such that the perturbed central path~\eqref{eqv:perturbedCentralPath} with $\lambda = \hat{\lambda}$ passes through $\optSol$ exactly when $\mu = \hat{\mu}$.
\end{theorem}

\begin{proof}
Since $\optSol$ is an optimal solution of~\eqref{eqv:originalPD}, it is also primal-dual feasible, and so $\optSol \in \sfsp$ for any $\lambda > 0$.
Thus, according to Lemma~\ref{prop-perturbedCentralPathWellDefined}, the perturbed central path is well defined. Furthermore, if there exists a $\hat{\lambda}>0$ such that 
\begin{equation}
\label{eqv:onCentralPathProof-thirdEq}
\left(X^{*}+\hat{\Lambda}\right)\left(S^{*}+\hat{\Lambda}\right)e = \hat{\mu}e,
\end{equation}
then $\optSol$ is the unique solution of the perturbed central path equations~\eqref{eqv:perturbedCentralPath} with $\lambda = \hat{\lambda}$ and $\mu = \hat{\mu}$, which implies the central path of perturbed problems passes through $\optSol$. 
It remains to solve~\eqref{eqv:onCentralPathProof-thirdEq} for $\hat{\lambda} = [\hat{\lambda}_{1} \,\, \ldots \,\, \hat{\lambda}_{n}]^{T}$. Since $x^{*}_{i}s^{*}_{i} = 0$, $i=1,\ldots,n$,
we have that~\eqref{eqv:onCentralPathProof-thirdEq} is equivalent to
\[
 \hat{\lambda}_{i}^{2}  + \left(x^{*}_{i}+s^{*}_{i}\right)\hat{\lambda}_{i} - \hat{\mu}  = 0,\quad i=1,\ldots,n,
\]
whose positive root for each $i$ gives the corresponding component of the required~$\hat{\lambda}$.
\qed \end{proof}

It is a stringent and impractical requirement to force the optimal solution of the original problem to be exactly on the central path of the perturbed problems. Thus we relax this requirement to allow for the original solution to belong to a small neighbourhood of this path.

\begin{theorem}[Existence of relaxed perturbations]
\label{thm-optsol-in-the-neigh-of-perCP_general}
Assume~\eqref{asm-full_row_rank} holds and  $(x^{*},y^{*},s^{*})$ is a~\eqref{eqv:originalPD} solution and let $\hat{\mu}>0$ and $\xi \in (0,1)$. Then there exist vectors $\hat{\lambda}_{L} = \hat{\lambda}_{L}(x^{*},s^{*},\hat{\mu},\xi)>0$ and $\hat{\lambda}_{U} = \hat{\lambda}_{U}(x^{*},s^{*},\hat{\mu},\xi)>0$
such that for $\hat{\lambda}_{L} \leq \lambda \leq \hat{\lambda}_{U}$, $\optSol$ is strictly feasible for~\eqref{eqv:perPD} and satisfies
\begin{equation}
\label{def-symmetric_neighborhood}
\xi \hat{\mu} e \leq (X^{*}+\Lambda)(S^{*}+\Lambda)e\leq \frac{1}{\xi} \hat{\mu} e. 
\end{equation}
\end{theorem}
\begin{proof}
Clearly, $\optSol$ satisfies~\eqref{eqv:strictlyFeasibleSet_per} and so $\optSol \in \sfsp$ for any $\lambda = [ \lambda_{1}\,\, \ldots\,\,\lambda_{n} ]^{T} > 0$.  The inequalities~\eqref{def-symmetric_neighborhood} are equivalent to
\begin{equation}
\label{symmetric_neighborhood_eqn}
\left\{
    \begin{array}{rcl}
        \lambda_{i}^{2}+(x^{*}_{i}+s^{*}_{i})\lambda_{i} - \xi \hat{\mu} & \geq & 0\\
        \lambda_{i}^{2}+(x^{*}_{i}+s^{*}_{i})\lambda_{i} - \frac{1}{\xi} \hat{\mu} & \leq & 0,
    \end{array}
\right.
\end{equation}
for all $i \in \{1,\ldots,n\}$ and $\xi \in (0,1)$.
Solving~\eqref{symmetric_neighborhood_eqn} for $\lambda_{i}$, we obtain
\begin{equation}
\label{eqv:intervalForPer}
\left\{
    \begin{array}{rcl}
        \lambda_{i} \geq & \frac{-(x^{*}_{i}+s^{*}_{i})+\sqrt{(x^{*}_{i}+s^{*}_{i})^{2}+4\xi\hat{\mu} }}{2} = & \frac{2\xi\hat{\mu}}{ x^{*}_{i}+s^{*}_{i} + \sqrt{(x^{*}_{i}+s^{*}_{i})^{2}+4\xi \hat{\mu}}} = (\hat{\lambda}_{L})_{i},\\
        0< \lambda_{i} \leq & \frac{-(x^{*}_{i}+s^{*}_{i})+\sqrt{(x^{*}_{i}+s^{*}_{i})^{2}+\frac{4\hat{\mu}}{\xi}}}{2} = & \frac{\frac{2\hat{\mu}}{\xi}}{x^{*}_{i}+s^{*}_{i}+\sqrt{(x^{*}_{i}+s^{*}_{i})^{2}+\frac{4\hat{\mu}}{\xi }}} = (\hat{\lambda}_{U})_{i},
    \end{array}
\right.
\end{equation}
for all $i \in \{1,\ldots,n\}$.
For any $\xi \in (0,1)$, it is easy to see that~\eqref{eqv:intervalForPer} yields a well-defined interval for $\lambda_{i}$, $i \in \{1,\ldots,n\}$.
\qed \end{proof}

From the above theorem, we see that by choosing the perturbations judiciously, we can bring any solution of the original problem into a `neighbourhood' of the perturbed central path.

\subsection{Preserving the optimal active set}
Since we are interested in predicting the optimal active set of the original problem, this section addresses the relation between the active set of the perturbed problem and that of the original {\lp}. We find that for sufficiently small perturbations, these two active sets remain the same provided the original problem is nondegenerate.

\begin{theorem}
\label{thm-preserving_actv}
Assume~\eqref{asm-full_row_rank} holds and the original pair of~\eqref{eqv:originalPD} problems has a unique and nondegenerate primal solution $x^{*}$.
Then there exists a positive scalar $\hat{\lambda}=\hat{\lambda}(A,b,c,x^{*})$ such that the pair of perturbed problems~\eqref{eqv:perPD} with $0 \leq \|\lambda\| < \hat{\lambda}$ has a strictly complementary solution $ (x^{*}_{\lambda},y^{*}_{\lambda},s^{*}_{\lambda})$ with the same active and inactive sets as $x^{*}$, where $\|\cdot\|$ denotes the Euclidean norm.
\end{theorem}
\begin{proof}
Since~\eqref{eqv:originalPD} has a unique and nondegenerate primal solution, it must have a unique primal-dual~nondegenerate solution $\optSol$~\cite[Theorem 4.5 (b)]{sierksma2001}, which must be strictly complementary and so $x^{*} + s^{*} > 0$. 
Thus, letting
\begin{equation}
\label{eqv:scp}
\Actv = \left\{ i \in \{ 1,\ldots,n \} \,\middle|\, x^{*}_{i} = 0 \right\} 
\quad \text{and} \quad 
\Iactv = \left\{ i \in \{ 1,\ldots,n \} \,\middle|\, s^{*}_{i} = 0 \right\},
\end{equation}
the {\kkt} conditions for~\eqref{eqv:originalPD} at $\optSol$---namely,~\eqref{eqv:kkt_perturbed} with $\lambda = 0$---become
\begin{subequations}
\label{eqv:origin_opt_AxbAtys_partiton}
\begin{align}
    x^{*}_{\Actv} = 0, \quad x^{*}_{\Iactv} > 0 \quad \text{and} \quad s^{*}_{\Iactv} = 0, \quad s^{*}_{\Actv} > 0,  \\
    A_{\Iactv} x^{*}_{\Iactv} =  b, \quad  A^{T}_{\Iactv}y^{*}   =  c_{\Iactv}, \quad A^{T}_{\Actv}y^{*} + s_{\Actv}^{*}   = c_{\Actv},
\end{align}
\end{subequations}
where $A = \left[ A_{\Iactv}   \,\,  A_{\Actv} \right]$, $ (x^{*})^{T} = [ (x^{*}_{\Actv})^{T} \,\, (x^{*}_{\Iactv})^{T}]$ and $(s^{*})^{T} = [ \, (s^{*}_{\Actv})^{T} \,\, (s^{*}_\Iactv)^{T} \,]$.
As the~\eqref{eqv:originalPD} solution is also nondegenerate, we must have $|\Iactv| = m$ and $rank(A_{\Iactv}) = m$, namely, $A_{\Iactv}$ is nonsingular.
We work with the equivalent form~\eqref{eqv:perPD_formPQ} of problems~\eqref{eqv:perPD}, and construct a solution $\optSolp$ of~\eqref{eqv:perPD_formPQ} such that 
$\hat{p} + \hat{q} > 0$,  $\hat{p}_{\Actv} = 0$ and $\hat{q}_{\Iactv} = 0$, namely,
\begin{subequations}
\label{eqv:uniqueSolpyq}
\begin{align}
  \hat{p}_{\Actv}  =  0,  \quad \hat{p}_{\Iactv}  = x^{*}_{\Iactv} + \lambda_{\Iactv} + A^{-1}_{\Iactv}A_{\Actv}\lambda_{\Actv}, \label{eqv:p_uniq_nondegen_primal}\\
  \hat{y}  = y^{*}+ (A^{T}_{\Iactv})^{-1}\lambda_{\Iactv}, \quad \hat{q}_{\Iactv} = 0,  \quad \hat{q}_{\Actv}  =  s^{*}_{\Actv} + \lambda _{\Actv} - (A_{\Iactv}^{-1}A_{\Actv})^{T}\lambda_{\Iactv}. \label{eqv:uniq_nondegen_dual}
\end{align}
\end{subequations}
 
Using~\eqref{eqv:origin_opt_AxbAtys_partiton}, it is straightforward to show that $\optSolp$ in~\eqref{eqv:uniqueSolpyq} satisfies all linear and nonlinear equality constraints in the {\kkt} conditions~\eqref{eqv:kkt_perturbed_form2}. It remains to prove that $\hat{p}_{\Iactv} > 0$ and $\hat{q}_{\Actv} > 0 $.
Let $\sigma_{\max}$ be the largest singular value of $A^{-1}_{\Iactv}A_{\Actv}$, and define a positive scalar $\hat{\lambda}$ as
\[
\hat{\lambda} =
\frac{\vmin{x^{*}_{\Iactv}}{s^{*}_{\Actv}}}{\sigma_{\max}},
\]
where $\vmin{x^{*}_{\Iactv}}{s^{*}_{\Actv}}$ is a scalar that denotes the smallest element of $x^{*}_{\Iactv}$ and $s^{*}_{\Actv}$.
From $\lambda \geq 0$ and from norm properties, we have that
\[
\hat{p}_{\Iactv} 
\geq x^{*}_{\Iactv} - \| A^{-1}_{\Iactv}A_{\Actv}\lambda_{\Actv} \| e_{\Iactv}
\geq x^{*}_{\Iactv} - \| A^{-1}_{\Iactv}A_{\Actv}\| \cdot \|\lambda_{\Actv} \| e_{\Iactv}
\geq x^{*}_{\Iactv} - \| A^{-1}_{\Iactv}A_{\Actv}\| \cdot \|\lambda \| e_{\Iactv}
\]
and
\[
\displaystyle\begin{array}{lcl}
\hat{q}_{\Actv} 
\geq s^{*}_{\Actv} - \| \left( A^{-1}_{\Iactv} A_{\Actv}\right)^{T} \lambda_{\Iactv} \| e_{\Actv}
&\geq& s^{*}_{\Actv} - \| \left( A^{-1}_{\Iactv} A_{\Actv}\right)^{T}\| \cdot \| \lambda_{\Iactv} \| e_{\Actv}\\[1ex]
&\geq& s^{*}_{\Actv} - \| \left( A^{-1}_{\Iactv} A_{\Actv}\right)^{T}\| \cdot \| \lambda \| e_{\Actv}.
\end{array}
\]
Using matrix norm properties, we obtain that $\left\| A^{-1}_{\Iactv}A_{\Actv} \right\| = \left\| (A^{-1}_{\Iactv}A_{\Actv})^{T} \right\| = \sigma_{\max}$.
This and $0 < \|\lambda\| < \hat{\lambda}$ now imply
\[
\hat{p}_{\Iactv} 
> x^{*}_{\Iactv} - \sigma_{\max}\hat{\lambda}  e_{\Iactv}
\geq x^{*}_{\Iactv} - \vmin{x^{*}_{\Iactv}}{s^{*}_{\Actv}} e_{\Iactv} \geq 0,
\]
and
\[
\hat{q}_{\Actv} 
> s^{*}_{\Actv} - \sigma_{\max} \hat{\lambda} e_{\Actv}
\geq s^{*}_{\Actv} - \vmin{x^{*}_{\Iactv}}{s^{*}_{\Actv}} e_{\Actv}
\geq 0,
\]
where we also use the definition of $\hat{\lambda}$.
\qed \end{proof}

\paragraph*{Remarks on the assumptions and proof of Theorem~\ref{thm-preserving_actv}.} \hfill

$\bullet$ An equivalent non-degeneracy assumption that would be sufficient in this theorem is to require that all~\eqref{eqv:originalPD} solutions are primal-dual nondegenerate
\cite[Section 5]{guler1993}.

$\bullet$ We have assumed in this theorem that~\eqref{eqv:originalPD} is primal-dual nondegenerate and has a unique solution, which guarantees $A_{\Iactv}$ is nonsingular.  
Considering the general case when $\optSol$ is a possibly non-unique strictly complementary solution, to construct the desired solution $\optSolp$ of~\eqref{eqv:perPD_formPQ} with the same active set and strictly complementary partition, one needs to satisfy exactly primal-dual feasibility requirements such as  
\begin{equation}
\label{eqv:tmp1}
A_{\Iactv} \hat{p}_{\Iactv} = b+A\lambda = b+A_{\Actv}\lambda_{\Actv} + A_{\Iactv}\lambda_{\Iactv}.
\end{equation}
Clearly, one can only guarantee~\eqref{eqv:tmp1} to be consistent for $\lambda > 0$ if $A_{\Actv}\lambda_{\Actv}$ belongs to the range space of $ A_{\Iactv}$. 
Alternatively, one could consider satisfying~\eqref{eqv:tmp1} only approximately and look for a solution $\hat{p}$ of the form 
\begin{equation}
\label{eqv:tmp2}
\hat{p}_{\Actv} = 0 \quad \text{and} \quad \hat{p}_{\Iactv} = x_{\Iactv}^{*} + \lambda_{\Iactv} + \hat{u},
\end{equation} 
where $\hat{u}$ is the least-squares/minimal norm solution of $A_{\Iactv} u  = A_{\Actv} \lambda_{\Actv} $. For instance in the case when $| \Iactv | \leq m$, we have $\|  A_{\Iactv} \hat{u}  - A_{\Actv} \lambda_{\Actv}  \| \leq \| A_{\Actv} \lambda_{\Actv} \|$. The right-hand side of the latter inequality goes to zero as $\lambda  \to  0$ and so primal feasibility can be approximately achieved. It can also be shown that $\hat{p}_{\Iactv}$ in \eqref{eqv:tmp2} stays positive. \hfill$\Box$

Note that the nondegeneracy assumption in Theorem~\ref{thm-preserving_actv} is not required in the results of  the next section or in our implementations and numerical experiments. 
Thus this theorem and its assumptions do not restrict our algorithmic or even main theoretical approach of predicting the optimal active set of the~\eqref{eqv:originalPD} problem by solving a perturbed~\eqref{eqv:perPD} problem. 
%%%%%%%%%%% Section: Using perturbations to predict the optimal active set %%%%%%%%%%%%%%

\section{Using perturbations to predict the original optimal active set}
\label{sec-predict}
%\resetcounters

Recalling our main aim, we now present results for predicting the optimal active set of~\eqref{eqv:originalPD}. The idea is to solve the perturbed problem instead of the original one using {\ipms}, but attempt to predict the active set for the original problem during the run of the algorithm. Without assuming that the original and perturbed problems have the same optimal active set, we prove that under certain conditions and given proper perturbations, when the duality gap of~\eqref{eqv:perPD} is sufficiently small, the predicted (strictly) active set for~\eqref{eqv:originalPD} coincides with the actual optimal (strictly) active set of~\eqref{eqv:originalPD}~(Theorems~\ref{thm-predicted_A_equals_Actual_A},~\ref{thm-predicted_A_plus_equals_Actual_A_plus}).

\subsection{Some useful results}
We first derive a bound on the distance between the original optimal solution set and strictly feasible points of the perturbed problems.

\begin{lemma}[An error bound for~\eqref{eqv:originalPD}]
\label{lem-bound_x_s_strict_fea}
Let $\pntp \in \sfsp$, where $\sfsp$ is defined in~\eqref{eqv:strictlyFeasibleSet_per}, and $\lambda \geq 0$.
Then there exists a~\eqref{eqv:originalPD} solution $\optSol$ such that
\begin{equation}
\label{eqv:iter_error_bound}
%\tag{ER}
	\| x - x^{*}\| \leq \tau_{p} \left( r( x, s ) + w( x, s )  \right) \quad \text{and} \quad \| s - s^{*}\| \leq \tau_{d} \left( r( x, s ) + w( x, s )  \right),
\end{equation}
where  $\tau_{p} > 0$ and $\tau_d > 0 $ are problem-dependent constants independent of $\pntp$ and $\optSol$, and 
\begin{equation}
\label{eqv:error_bounds_feasible_x_s}
	r( x, s )= \|\cmin{ x }{ s }\|
\quad \text{and} \quad
w( x, s ) = \|(-x, -s, x^{T} s )_{+}\|,
\end{equation}
and where $\cmin{x}{s} = \left(\, \min(x_{i},s_{i}) \,\right)_{ i = 1,\ldots, n }$ and $(x)_{+} = \left(\, \max(x_{i},0) \,\right)_{ i = 1,\ldots, n }$.
\end{lemma}
See Appendix~\ref{apd-proof_lemma_4_1} for a proof of this lemma.

\begin{lemma}{\normalfont\textbf{\cite[Lemma 5.13]{wright}}}
\label{lem-boundXandSAccording2Partition}
For any $\pntp \in \sfsp$, where $\sfsp$ is defined in~\eqref{eqv:strictlyFeasibleSet_per}, we have
\begin{equation}
\label{eqv:xi_si_sep}
0<x_{i}+\lambda_{i} \leq \frac{\mup}{C_{1}} \quad (i \in \Actv_{\lambda})
\quad \text{and} \quad
0<s_{i}+\lambda_{i} \leq \frac{\mup}{C_{1}} \quad (i \in \Iactv_{\lambda}), 
\end{equation}
where 
\begin{equation}
\label{eqv:mu}
\mup = \frac{(x+\lambda)^{T}(s+\lambda)}{n}
\end{equation}
and
\begin{equation}
\label{eqv_C1}
	C_{1} = \frac{\epsilon(A,b_{\lambda},c_{\lambda})}{n}
\end{equation}
with
\begin{equation}
\label{eqv:eps_hat}
\begin{split}
   & \epsilon(A,b_{\lambda},c_{\lambda}) \\ 
= & \min \left( \,\, \min_{i\in \Iactv_{\lambda}} \sup_{x^{*}_{\lambda} \in \sspp} \left\{ \, (x^{*}_{\lambda})_{i}  + \lambda_{i}\, \right\}, \,  \min_{i\in \Actv_{\lambda}}\sup_{(y^{*}_{\lambda},s^{*}_{\lambda}) \in \ssdp} \left\{ \,(s^{*}_{\lambda})_{i} +\lambda_{i} \,\right\}  \,\,\right) > 0,
\end{split}
\end{equation}
and
$\sspp$ and $\ssdp$ are the primal and dual solution sets of~\eqref{eqv:perPD} respectively,
and where $(\Actv_{\lambda}, \Iactv_{\lambda})$ is the strictly complementary active and inactive partition of the solution set of~\eqref{eqv:perPD}.
\end{lemma}
\begin{proof} Firstly, \eqref{eqv:eps_hat} is well-defined: when the feasible set of~\eqref{eqv:originalPD} is  nonempty, that of~\eqref{eqv:perPD} is also nonempty, and so $\epsilon(A,b_{\lambda},c_{\lambda}) > 0$. To prove the Lemma,
apply~\cite[Lemma 5.13]{wright} to~\eqref{eqv:perPD_formPQ} and recall $x = p-\lambda$ and $s = q - \lambda$.
Note that Lemma 5.13 is a more complex result that also assumes loose proximity to the problem central path, but only strict feasibility is required to prove the required inequalities in~\eqref{eqv:xi_si_sep}. 
\qed \end{proof}

\begin{lemma}
\label{lem-boundXandXstar}
Let $\pntp \in \sfsp$, where $\sfsp$ is defined in~\eqref{eqv:strictlyFeasibleSet_per} for some $\lambda \geq 0$. Then there exists a~\eqref{eqv:originalPD} solution $\optSol$ and problem-dependent constants $\tau_{p}$ and $\tau_{d}$ that are independent of $\pntp$ and $\optSol$, such that
%\begin{equation}
%\label{eqv:bounds_x_xstar_on_mu_lambda}
%\resizebox{.9\hsize}{!}{$
%	\|x-x^{*}\| < \tau_{p} \left( C_{2} \mup + 4 \|\lambda\| \max\left(\|\lambda\|,1\right) \right)
%	\text{ and }
%	\|s-s^{*}\| < \tau_{d} \left( C_{2} \mup + 4 \|\lambda\| \max\left(\|\lambda\|,1\right) \right)$},
%\end{equation}
\begin{equation}
\label{eqv:bounds_x_xstar_on_mu_lambda}
\displaystyle\begin{array}{c}
	\|x-x^{*}\| < \tau_{p} \left( C_{2} \mup + 4 \|\lambda\| \max\left(\|\lambda\|,1\right) \right)\\[1ex]
	\text{ and }\\[1ex]
	\|s-s^{*}\| < \tau_{d} \left( C_{2} \mup + 4 \|\lambda\|
          \max\left(\|\lambda\|,1\right) \right),
\end{array}
\end{equation}
where  
\begin{equation}
\label{eqv:C2}
	C_{2}  = \frac{n\sqrt{n}}{\epsilon(A,b_{\lambda},c_{\lambda})} + n,
\end{equation} 
$\epsilon(A,b_{\lambda},c_{\lambda})$ is defined in~\eqref{eqv:eps_hat} and $\mup$ in~\eqref{eqv:mu}.

\end{lemma}
\begin{proof}
Since $x+\lambda > 0$ and $s+\lambda > 0$, we have $ -x < \lambda $ and $-s < \lambda$, 
which implies
\begin{equation}
\label{eqv:bounds_x_plus_s_plus}
0 \leq (-x)_{+} < \lambda \quad \text{and} \quad 0 \leq (-s)_{+} < \lambda.
\end{equation}
Using~\eqref{eqv:mu}, $\lambda \geq 0$ and $(x+\lambda,s+\lambda) \geq 0$, we have
\begin{equation}
\label{eqv:bounds_xts_plus}
x^{T}s =  n \mup + \lambda^{T}\lambda - \lambda^{T} (x + \lambda) - \lambda^{T}  (s + \lambda) \leq  n \mup + \| \lambda \|^{2}. 
\end{equation}

From~\eqref{eqv:error_bounds_feasible_x_s},~\eqref{eqv:bounds_x_plus_s_plus} and~\eqref{eqv:bounds_xts_plus}, we obtain
\begin{equation}
\label{eqv:bound_w_mu_lambda}
	w(x,s) \leq \|(-x)_{+}\| + \|(-s)_{+}\| + ( x^{T}s)_{+} \leq n \mup + 2\|\lambda\| + \|\lambda\|^{2}.
\end{equation}
It remains to find an upper bound for $r(x,s)$ in~\eqref{eqv:error_bounds_feasible_x_s}.
If $i \in \Actv_{\lambda}$, from~\eqref{eqv:xi_si_sep} we have 
$
	\min \left( x_{i} + \lambda_{i} , s_{i} + \lambda_{i} \right) \leq x_{i}+\lambda_{i} \leq  \frac{\mup}{C_{1}}.
$
Similarly, we also have $\min \left( x_{i} + \lambda_{i} , s_{i} + \lambda_{i} \right) \leq \frac{\mup}{C_{1}}$ for $i \in \Iactv_{\lambda}$.
Thus
$
0 < \cmin{x+\lambda}{ s + \lambda}                     
\leq \frac{\mup}{C_{1}} e,
$
and so from~\eqref{eqv:error_bounds_feasible_x_s},
%\begin{equation}
%\label{eqv:bound_r_mu_lambda}
%\resizebox{.895\hsize}{!}{$
\[ 
r(x,s)    = \| \cmin{x+\lambda}{s+\lambda} -\lambda \| 
	   \leq \| \cmin{x+\lambda}{s+\lambda} \| + \| \lambda \|
	     \leq \frac{\mup}{C_{1}}\sqrt{n} +  \|\lambda \|. 
\]
%$}
%\end{equation}
This, \eqref{eqv:iter_error_bound} and \eqref{eqv:bound_w_mu_lambda} now
provide the bound~\eqref{eqv:bounds_x_xstar_on_mu_lambda}.
% now follows from~\eqref{eqv:iter_error_bound}, ~\eqref{eqv:bound_w_mu_lambda} and~\eqref{eqv:bound_r_mu_lambda}.
\qed \end{proof}

\subsection{Predicting the original optimal active set  using perturbations}
%During the iterative process of solving the perturbed problem using an interior point framework, we try to predict the optimal active set for the original problem.
Assume $\optSol$ is a~\eqref{eqv:originalPD} solution. We denote by
$\Actv(x^{*})$ the optimal active set at $x^{*}$ and by $\Sactv (s^{*})$, the `strictly' active set at $s^{*}$, namely, 
\begin{equation}
\label{eqv:actv_strictly_actv_at_a_sol}
%\resizebox{.89\hsize}{!}{$
	\Actv(x^{*}) = \left\{   \, i \in \{1,\ldots,n\}\,|\, x^{*}_{i}  = 0 \right\}
	\quad\text{and}\quad
	\Sactv(s^{*}) = \left\{   \, i \in \{1,\ldots,n\}\,|\,
          s^{*}_{i} > 0 \right\}.
%$}
\end{equation}
Let
\begin{equation}
\label{eqv:predictedSets}
%\resizebox{.89\hsize}{!}{$
	\pacs (x)  = \left\{ i \in \{1,\ldots, n\} \, | \, x_{i}  <  C \right\}
	\quad \text{and} \quad 
	\psas  (s)= \left\{ i \in \{1,\ldots, n\} \, | \, s_{i} \geq C
        \right\},
%$}
\end{equation}
where $C>0$ is some constant threshold.  $\pacs(x)$ is considered as the predicted active set and $\psas(s)$, the  predicted `strictly'  active set at a primal-dual pair $\pntp$ for~\eqref{eqv:perPD}.
\begin{theorem}
\label{thm-bound_active_set}
Let $C > 0$ and fix the perturbation $\lambda$ such that
\begin{equation}
\label{eqv:thm_bound_lambda}
	 0 < \|\lambda\| < \min \left( 1, \frac{C}{8 \max(\tau_{p},\tau_{d}) }  \right),
\end{equation}
where $\tau_{p}$ and $\tau_{d}$ are the problem-dependent constants in~\eqref{eqv:bounds_x_xstar_on_mu_lambda}.
Let $\pntp \in \sfsp$ with $\mup$ sufficiently small, namely,
\begin{equation}
\label{eqv:mu_bounds_doubside}
	\mup < \frac{C}{2C_{2}\max(\tau_{p},\tau_{d})},
\end{equation}
where $\sfsp$ is defined in~\eqref{eqv:strictlyFeasibleSet_per}, $\mup$ in~\eqref{eqv:mu} and $C_{2} > 0$ in~\eqref{eqv:C2} is a problem-dependent constant when $\lambda$ is fixed. Then there exists a~\eqref{eqv:originalPD} solution $\optSol$ such that
\[
	\psas(s) \subseteq \Sactv(s^{*}) \subseteq \Actv(x^{*}) \subseteq \pacs(x).
\]
\end{theorem}
\begin{proof}
From $\|\lambda\| < 1$ and~\eqref{eqv:bounds_x_xstar_on_mu_lambda}, we have
$ \|x - x^{*}\| \leq \tau_{p} \left( C_{2} \mup + 4 \|\lambda\| \right)  $ and $\|s - s^{*}\| \leq \tau_{d} \left( C_{2} \mup + 4 \|\lambda\| \right)$,
which imply
\begin{equation}
\label{eqv_proof_bounds_primal}
	x^{*}_{i} - \tau_{p} \left( C_{2} \mup + 4 \|\lambda\| \right) \leq x_{i} \leq x^{*}_{i} + \tau_{p} \left( C_{2} \mup + 4 \|\lambda\| \right)
\end{equation}
and
\begin{equation}
\label{eqv:proof_bounds_dual}
	s^{*}_{i} - \tau_{d} \left( C_{2} \mup + 4 \|\lambda\| \right) \leq s_{i} \leq s^{*}_{i} + \tau_{d} \left( C_{2} \mup + 4 \|\lambda\| \right),
\end{equation}
for all $i \in \{1,\ldots,n\}$.
If $i \in \Actv(x^{*})$, from~\eqref{eqv:thm_bound_lambda},~\eqref{eqv:mu_bounds_doubside} and~\eqref{eqv_proof_bounds_primal}, we have $x_{i} < C$, namely $i \in \pacs(x)$. So $\Actv(x^{*}) \subseteq \pacs(x)$. 
If $i \notin \Sactv(s^{*})$, $s^{*}_{i} = 0$. Then from~\eqref{eqv:thm_bound_lambda},~\eqref{eqv:mu_bounds_doubside} and~\eqref{eqv:proof_bounds_dual}, we have $s_{i} < C$, namely, $i \notin \psas(s)$. Thus $\psas(s) \subseteq \Sactv(s^{*})$. From $x^{*}_{i}s^{*}_{i} = 0$ for all $i \in \{1,\ldots,n\}$, we have $\Sactv(s^{*}) \subseteq \Actv(x^{*})$.
\qed \end{proof}

Theorem~\ref{thm-bound_active_set} shows that $ \pacs (x) $ and $ \psas(s) $ serve as a pair of approximations that bound $\Actv(x^{*})$.  Next we go a step further and show that $\pacs (x)$ is equivalent to $\Actv (x^{*})$ under certain conditions.

\begin{theorem}
\label{thm-predicted_A_equals_Actual_A}
Let 
\begin{equation}
\label{eqv:psi_p}
\psi_{p}  = \inf_{x^{*} \in \sspo}\min_{ i \notin \Actv(x^{*})}  (x^{*}_{i})
\end{equation}
where ${\sspo}$ is the solution set of the primal problem in~\eqref{eqv:originalPD} and $\Actv(x^{*})$ is defined in~\eqref{eqv:actv_strictly_actv_at_a_sol}. Assume $\psi_{p} > 0$.
Fix $\lambda$ and $C$ such that  
\begin{equation} 
\label{eqv:thm_bound_lambda_value_C_primal}
	0 < \|\lambda\| < \min\left( 1, \frac{\psi_{p}}{16\max (\tau_{p}, \tau_{d})}\right)
	\quad \text{and} \quad
	C = \frac{\psi_{p}}{2},
\end{equation}
where $\tau_{p}$ and $\tau_{d}$ are the problem-dependent constants defined in~\eqref{eqv:bounds_x_xstar_on_mu_lambda}.
Let $\pntp \in \sfsp$ with $\mup$ sufficiently small, namely,
\begin{equation}
\label{eqv:mu_max_actv_primal}
	\mup < \frac{\psi_{p}}{ 4 C_{2} \max (\tau_{p}, \tau_{d}) },
\end{equation}
where $\sfsp$ is defined in~\eqref{eqv:strictlyFeasibleSet_per}, $\mup$ in~\eqref{eqv:mu} and $C_{2} > 0$ in~\eqref{eqv:C2}. Then there exists a~\eqref{eqv:originalPD} solution $\optSol$ such that 
\[
	\pacs(x) = \Actv(x^{*}),
\]
where $\pacs(x)$ is defined in~\eqref{eqv:predictedSets}.
\end{theorem}
\begin{proof}
From Theorem~\ref{thm-bound_active_set} we have $\Actv(x^{*}) \subseteq \pacs(x)$. It remains to prove $\pacs(x) \subseteq \Actv( x^{*} )$.
If $i \notin \Actv(x^{*})$, from the left inequality in~\eqref{eqv_proof_bounds_primal},~\eqref{eqv:thm_bound_lambda_value_C_primal} and~\eqref{eqv:mu_max_actv_primal}, we have
\[
	x_{i} > x^{*}_{i} - \frac{\psi_{p}}{2} \cdot \frac{\tau_{p}}{\max( \tau_{p}, \tau_{d} )}
	\geq \inf_{x^{*} \in \sspo}\min_{ i \notin \Actv(x^{*})}  (x^{*}_{i}) - \frac{\psi_{p}}{2}
	=\psi_{p} - \frac{\psi_{p}}{2}= C.
\]
Thus $i \notin \pacs(x)$, which implies $\pacs(x) \subseteq \Actv( x^{*} )$.
\qed \end{proof}

Next, we show that $\psas(s)$, the predicted strictly active set at a strictly feasible point $\pntp$ of~\eqref{eqv:perPD}, is the same as $\Sactv(s^{*})$ at some~\eqref{eqv:originalPD} solution $\optSol$.

\begin{theorem}
\label{thm-predicted_A_plus_equals_Actual_A_plus}
Let 
\[
	\psi_{d}  = \inf_{(y^{*}, s^{*}) \in \ssdo}\min_{ i \in \Sactv(s^{*})}  (s^{*}_{i})
\] 
where $\ssdo$ is the solution set of the dual problem in~\eqref{eqv:originalPD} and $\Sactv(s^{*})$ is defined in~\eqref{eqv:actv_strictly_actv_at_a_sol}. 
Assume $\psi_{d} > 0$.
Fix $\lambda$ and $C$ such that
\begin{equation}
\label{eqv:thm_bound_lambda_value_C_dual}
	0 < \|\lambda\| < \min\left( 1, \frac{\psi_{d}}{16 \max (\tau_{p}, \tau_{d}) }\right)
	\quad\text{and}\quad 
	C = \frac{\psi_{d}}{2},
\end{equation}
where $\tau_{p}$ and $\tau_{d}$ are the problem-dependent constants in~\eqref{eqv:bounds_x_xstar_on_mu_lambda}.
Let $\pntp \in \sfsp$ with  $\mup$ sufficiently small, namely
\begin{equation}
\label{eqv:mu_max_actv_dual}
	\mup < \frac{\psi_{d}}{ 4 C_{2} \max (\tau_{p}, \tau_{d}) },
\end{equation}
where $\sfsp$ is defined in~\eqref{eqv:strictlyFeasibleSet_per}, $\mup$ in~\eqref{eqv:mu} and $C_{2} > 0$ in~\eqref{eqv:C2}.
Then there exists a~\eqref{eqv:originalPD} solution $\optSol$ such that 
\[
	\psas(s) = \Sactv(s^{*}),
\]
where $\psas(s)$ is defined in~\eqref{eqv:predictedSets}.
\end{theorem}
\begin{proof}
From Theorem~\ref{thm-bound_active_set}, we have $ \psas(s) \subseteq \Sactv(s^{*})$.
If $i \in \Sactv(s^{*})$, $s^{*}_{i} > 0$. This,~\eqref{eqv:proof_bounds_dual},~\eqref{eqv:thm_bound_lambda_value_C_dual} and~\eqref{eqv:mu_max_actv_dual} give us
\[
s_{i} > s^{*}_{i} - \frac{\psi_{d}}{2} \cdot \frac{\tau_{d}}{\max( \tau_{p}, \tau_{d} )}
	\geq \inf_{(y^{*},s^{*} )\in \ssdo}\min_{ i \in \Sactv(s^{*})}  (s^{*}_{i}) - \frac{\psi_{d}}{2}
	=\psi_{d} - \frac{\psi_{d}}{2}= C,
\]
namely $ \Sactv(s^{*}) \subseteq \psas(s)$.
\qed \end{proof}

\paragraph*{Remarks on Theorems~\ref{thm-bound_active_set}--\ref{thm-predicted_A_plus_equals_Actual_A_plus}.}\hfill

$\bullet$ We require $\mup$, the mean value of the complementary products,  to be sufficiently small in Theorems~\ref{thm-bound_active_set}--\ref{thm-predicted_A_plus_equals_Actual_A_plus}. This choice is possible since we have $\mup = 0$ at any optimal solution of~\eqref{eqv:perPD} and $\mup$  can be decreased to zero (such as in an {\ipm} framework).

$\bullet$  In Theorems~\ref{thm-predicted_A_equals_Actual_A} and~\ref{thm-predicted_A_plus_equals_Actual_A_plus}, we do not require that the optimal active set of~\eqref{eqv:perPD} is the same as that of~\eqref{eqv:originalPD} in order to be able to predict the original optimal active set of~\eqref{eqv:originalPD}.

$\bullet$  $\psi_{p}$ in~\eqref{eqv:psi_p} is positive if the primal problem in~\eqref{eqv:originalPD} has a unique (degenerate or nondegenerate) solution, but we expect that it may often be zero in the case of multiple solutions. (Clearly, in our implementations,
we do not choose the cut-off value based on the theoretical quantity $\psi_{p}$.)
Similarly to $\psi_{p}$, if the dual problem in~\eqref{eqv:originalPD} has a unique (degenerate or nondegenerate) solution, we have $\psi_{d} > 0$. 

$\bullet$  Fix $\lambda$ sufficiently small and let $\iterKp$ be iterates of a primal-dual path-following {\ipm} applied to~\eqref{eqv:perPD}. 
Then assuming these iterates belong to some good neighbourhood of the central path of~\eqref{eqv:perPD} and that the barrier parameter is decreased appropriately, we have $\mup^{k} \to 0$ as $k \to \infty$~\cite[Theorem 5.11]{wright}. So, by applying Theorem~\ref{thm-predicted_A_equals_Actual_A},  for each $k$ sufficiently large, there exists a~\eqref{eqv:originalPD} solution $\optSol$ such that $\pacs (x^{k}) = \Actv (x^{*})$ (see also Lemma~\ref{lem-wright_lemma_to_prediction}~below). \hfill$\Box$

\section{Comparing perturbed and unperturbed active-set predictions}
\label{sec-copmare_p_up}
%\resetcounters

\subsection{Comparing with active-set prediction for~\eqref{eqv:perPD}}
Consider the `large' neighbourhood of the perturbed central path
\begin{equation}
\label{eqv:NInf}
\Ninfp = \{ \, (x,y,s) \in \sfsp \,\,|\,\, (x_{i} + \lambda_{i})(s_{i} + \lambda_{i}) \geq \gamma \mup, \, i=1,\ldots,n \,\},
\end{equation}
where $\sfsp$ is defined in~\eqref{eqv:strictlyFeasibleSet_per} and $\mup$ is defined in~\eqref{eqv:mu}; see~\cite[(1.16)]{wright} for the definition~\eqref{eqv:NInf} in the case of $\lambda \equiv 0$.

Next we rephrase Lemma 5.13 in~\cite{wright} as an active-set prediction result for~\eqref{eqv:perPD}.

\begin{lemma}
\label{lem-wright_lemma_to_prediction}
Let $\pntp$ in $\Ninfp$ and $\mup$ defined in~\eqref{eqv:mu}. Assume $C$ in~\eqref{eqv:predictedSets} is set to $C = \frac{\epsilon(A,b_{\lambda},c_{\lambda})\gamma}{n} $, where $\epsilon(A,b_{\lambda},c_{\lambda})$ is defined in~\eqref{eqv:eps_hat}. Then when $\mup < \bar{\mu}^{\max}_{\lambda}$, where
\begin{equation}
\label{eqv:mu_bar_max_p}
	\bar{\mu}^{\max}_{\lambda} = \frac{\epsilon^{2}(A,b_{\lambda},c_{\lambda})\gamma}{n^{2}}, 
\end{equation}
for any strictly complementary solution $\optSollambda$ of~\eqref{eqv:perPD} we have
\[
	\pacs (x+\lambda) = \Actv(x^{*}_{\lambda} + \lambda),
\]
where $\pacs (x+\lambda)$ is defined in~\eqref{eqv:predictedSets} with $x$ replaced by $x+\lambda$ and $\Actv(x^{*}_{\lambda} + \lambda)$ is defined in~\eqref{eqv:actv_strictly_actv_at_a_sol} with $x^{*}$ replaced by $x^{*}_{\lambda} + \lambda$.
\end{lemma}
\begin{proof}
We work with the equivalent form~\eqref{eqv:perPD_formPQ} of~\eqref{eqv:perPD}.
Given~\eqref{eqv:mu_bar_max_p}, apply~\cite[Lemma 5.13]{wright} to~\eqref{eqv:perPD_formPQ}, recalling that $x = p-\lambda$ and $s = q-\lambda$, and then we have
\begin{equation}
\label{eqv:perturbed_probs_A_I_sep}
\begin{array}{rl}
 i \in \Actv_{\lambda}  : &  0 < x_{i} + \lambda_{i} \leq \frac{\mup}{C_{1}} < C_{1}\gamma \leq s_{i} + \lambda_{i}, \\
 i \in \Iactv_{\lambda} : &  0 < s_{i} + \lambda_{i} \leq \frac{\mup}{C_{1}} < C_{1}\gamma \leq x_{i} + \lambda_{i},
\end{array}
\end{equation}
where $(\Actv_{\lambda}, \Iactv_{\lambda})$ is the strictly complementary active and inactive partition of the solution set of~\eqref{eqv:perPD_formPQ}. For any strictly complementary solution $\optSollambda$ of~\eqref{eqv:perPD}, $(x^{*}_{\lambda}+\lambda,y^{*}_{\lambda}, s^{*}_{\lambda} + \lambda)$ is a strictly complementary solution of~\eqref{eqv:perPD_formPQ}. This and the definition of $\Actv (x^{*}_{\lambda} + \lambda)$ give us that $\Actv (x^{*}_{\lambda} + \lambda) = \Actv_{\lambda}$. 
From~\eqref{eqv:perturbed_probs_A_I_sep} and the definition of $\pacs (x+\lambda)$, we also have $\pacs (x+\lambda) = \Actv_{\lambda}$. 
\qed \end{proof}

Substituting~\eqref{eqv:C2} into~\eqref{eqv:mu_max_actv_primal}, we obtain the following 
threshold value
\begin{equation}
\label{eqv:mupmax}
	\mu^{\max}_{\lambda}  := \frac{\psi_{p}\epsilon(A,b_{\lambda}, c_{\lambda})}{4 n \max (\tau_{p}, \tau_{d})  \left( \sqrt{n} + \epsilon(A,b_{\lambda}, c_{\lambda}) \right)} ,
\end{equation}
where $\epsilon(A,b_{\lambda}, c_{\lambda})$ is defined in~\eqref{eqv:eps_hat}, $\psi_{p}$ in~\eqref{eqv:psi_p}, and $\tau_{p}$ and $\tau_{d}$ are the positive constants in the bounds~\eqref{eqv:bounds_x_xstar_on_mu_lambda}. 
Theorem~\ref{thm-predicted_A_equals_Actual_A} provides that when $\psi_{p}>0$ and $\lambda$ is sufficiently small and fixed, if $\mup < \mu^{\max}_{\lambda}$, we can predict the optimal active set of~\eqref{eqv:originalPD}. 
Lemma~\ref{lem-wright_lemma_to_prediction} shows that when $\mup < \bar{\mu}^{\max}_{\lambda}$, where $\bar{\mu}^{\max}_{\lambda}$ is defined in~\eqref{eqv:mu_bar_max_p}, we can provide the strictly complementary partition of the solution set of~\eqref{eqv:perPD} from any primal-dual pair in the neighbourhood $\Ninfp$ of the perturbed central path.
To verify if our approach can predict the optimal active set of~\eqref{eqv:originalPD} before the strictly complementary partition of~\eqref{eqv:perPD}, we determine conditions under which $\mu^{\max}_{\lambda} > \bar{\mu}^{\max}_{\lambda}$.  
\begin{theorem}
\label{thm-mu_max_lambda_g_hat_mu_max_lambda}
In the conditions of Theorem~\ref{thm-predicted_A_equals_Actual_A}, let
\begin{equation}
\label{eqv:rho}
	\rho = \frac{\psi_{p}}{  \max (\tau_{p}, \tau_{d}) }.
\end{equation}
If
\begin{equation}
\label{eqv:hat_mu_max_g_mu_max}
		\epsilon(A,b_{\lambda}, c_{\lambda})  \leq  \mathcal{O} \left( \sqrt{n\rho} \min\left(  \sqrt{\rho}, 1  \right)  \right),
\end{equation}
then 
\[
\mu^{\max}_{\lambda} > \bar{\mu}^{\max}_{\lambda},
\] 
where $\epsilon(A,b_{\lambda},c_{\lambda})$ is defined in~\eqref{eqv:eps_hat}, $\mu^{\max}_{\lambda} $ in~\eqref{eqv:mupmax} and $ \bar{\mu}^{\max}_{\lambda}$ in~\eqref{eqv:mu_bar_max_p}.
\end{theorem}
\begin{proof}
Note that $\mu^{\max}_{\lambda} > \bar{\mu}^{\max}_{\lambda}$ is equivalent to 
%\resizebox{.45\textwidth}{!}{
$$ \epsilon^{2}(A,b_{\lambda}, c_{\lambda}) + \sqrt{n} \epsilon(A,b_{\lambda}, c_{\lambda})  - \frac{\rho}{4 \gamma} < 0, $$
which is satisfied if
\begin{equation}
\label{eqv:bound_hat_eps_full}
	0 < \epsilon(A,b_{\lambda}, c_{\lambda})  \leq  
	\frac{\sqrt{n}}{2\sqrt{\gamma}} \cdot \frac{ \rho }{  \sqrt{ \gamma + \rho } + \sqrt{\gamma} }.
\end{equation}
Since $\gamma \in (0,1)$ and $\sqrt{a+b} \leq \sqrt{a} + \sqrt{b}$ for any $a$ and $b$ nonnegative scalars, we have
\[
\frac{  \rho }{  \sqrt{ \gamma +  \rho } + \sqrt{\gamma} } 
\geq 
\frac{ \rho }{  \sqrt{\rho}  + 2\sqrt{\gamma} } 
\geq
\frac{1}{ 3\sqrt{\gamma} } \frac{   \rho }{   \max\left(\sqrt{ \rho }, 1\right)  }
\geq
\frac{\sqrt{\rho}}{3 \sqrt{\gamma}} \min \left( \sqrt{\rho}, 1 \right).
\]
The result follows from~\eqref{eqv:bound_hat_eps_full} and the above inequalities.
\qed \end{proof}

Theorem~\ref{thm-mu_max_lambda_g_hat_mu_max_lambda} implies that when solving the perturbed problems~\eqref{eqv:perPD}, if $\epsilon(A,b_{\lambda}, c_{\lambda})$ is sufficiently small, we can predict the optimal active set of~\eqref{eqv:originalPD} before $\mup$ gets so small that we can even obtain the strictly complementary partition of~\eqref{eqv:perPD}. 
To see an example when~\eqref{eqv:hat_mu_max_g_mu_max} is satisfied, see our remarks after Theorem~\ref{thm-mu_max_lambda_g_mu_max}.

\Remark
In Theorem~\ref{thm-mu_max_lambda_g_hat_mu_max_lambda}, we do not require the optimal active set of~\eqref{eqv:perPD} to be the same as the optimal active set of~\eqref{eqv:originalPD}. In fact, we will show that, in the numerical tests for the randomly generated problems (degenerate or nondegenerate), the optimal active sets of most perturbed problems are different from those of the original problems, but we can still predict sooner/better for~\eqref{eqv:originalPD}. In particular, the numerical experiments show that we are not solving~\eqref{eqv:perPD} to high accuracy and  there are iterations where we can predict the active set for~\eqref{eqv:originalPD} but we are not close to the solution set of~\eqref{eqv:perPD} or able to predict the active set of~\eqref{eqv:perPD}; see page~\pageref{exp-predict_sooner_than_scp_per}. \hfill $\Box$

\subsection{Comparing with active-set prediction for~\eqref{eqv:originalPD}}
Similarly to Lemma~\ref{lem-wright_lemma_to_prediction}, when we solve the original~\eqref{eqv:originalPD} problems we can predict the optimal~\eqref{eqv:originalPD} active set when the~\eqref{eqv:originalPD} duality gap is smaller than some threshold.  In this section, we intend to compare this threshold with the threshold value of $\mup$ when we are able to predict the optimal active set of~\eqref{eqv:originalPD} by solving~\eqref{eqv:perPD} and show that the latter could be greater than the former under certain conditions (Theorem~\ref{thm-mu_max_lambda_g_mu_max}). 

Lemma 5.13 in~\cite{wright} yields an active-set prediction result for~\eqref{eqv:originalPD}. In fact this result can be obtained by setting 
$\lambda = 0$ in Lemma~\ref{lem-wright_lemma_to_prediction}, but for clarity, we restate it here.

\begin{lemma}
\label{lem-wright_lemma_to_prediction_original}
{\normalfont \textbf{\cite[Lemma 5.13]{wright}}}
Let $(x,y,s)$ in $\Ninf$, where $\Ninf$ is the neighbourhood $\Ninfp$ in~\eqref{eqv:NInf} with $\lambda = 0$, and let $\mu$ as in~\eqref{eqv:mu} with $\lambda = 0$. Let the cut-off value $C$ in~\eqref{eqv:predictedSets} be set to $C = \frac{\epsilon(A,b,c)\gamma}{n} $, where 
\begin{equation}
\label{eqv:eps}
	\epsilon(A,b,c) = \min\left(\,\,\min_{i\in \Iactv}\sup_{x^{*} \in \sspo} x^{*}_{i}, \,\,
\min_{i\in \Actv}\sup_{(y^{*},s^{*}) \in \ssdo} \,s^{*}_{i}   \right) > 0,
%%\footnote{$\epsilon(A,b,c) > 0$ provided that the feasible set of~\eqref{eqv:originalPD} is nonempty.}
\end{equation}
$\sspo$ and $\ssdo$ are the primal and dual solution sets of~\eqref{eqv:originalPD} respectively, and $(\Actv, \Iactv)$ is the strictly complementary active and inactive partition of the solution set of~\eqref{eqv:originalPD}.
When $\mu < \mu^{\max}$, where 
\begin{equation}
\label{eqv:mu_max}
	\mu^{\max} = \frac{\epsilon^{2}(A,b,c)}{n^{2}} \gamma, 
\end{equation}
then for any strictly complementary solution $\optSol$ of~\eqref{eqv:originalPD} we have
\[
	\pacs (x) = \Actv(x^{*}),
\]
where $\pacs (x)$ is defined in~\eqref{eqv:predictedSets} and $\Actv(x^*)$ is defined in~\eqref{eqv:actv_strictly_actv_at_a_sol}.
\end{lemma}

Before we deduce a relationship between $\mu^{\max}_{\lambda}$ in~\eqref{eqv:mupmax} and $\mu^{\max}$ in~\eqref{eqv:mu_max}, we first relate two other important quantities, $\epsilon(A,b_{\lambda},c_{\lambda})$ and $\epsilon(A,b,c)$. 

\begin{lemma}
\label{lem-hat_epsilon_g_epsilon}
Assume~\eqref{asm-full_row_rank} holds  and~\eqref{eqv:originalPD} has a unique and nondegenerate solution $\optSol$. 
Then
there exists a sufficiently small 
$ \bar{\lambda}(A,b,c,x^{*},s^{*})>0 $
such that
\begin{equation}
\label{eqv:epsilon_epsilon_hat}
	\epsilon(A,b_{\lambda},c_{\lambda}) > \epsilon(A,b,c)
\end{equation}
for all $\lambda$ such that $0 \leq \lambda = \alpha \bar{\lambda}<\bar{\lambda}$, where $\alpha \in (0,1)$, and 
where $\epsilon(A,b_{\lambda}, c_{\lambda}) $ is defined in~\eqref{eqv:eps_hat} and $ \epsilon(A,b,c)$ in~\eqref{eqv:eps}.
\end{lemma}
The proof of this lemma is given in Appendix~\ref{sec-compare_eps}.

\begin{theorem}
\label{thm-mu_max_lambda_g_mu_max}
In the conditions of Theorem~\ref{thm-predicted_A_equals_Actual_A}, assume~\eqref{asm-full_row_rank} holds and~\eqref{eqv:originalPD} has a unique and nondegenerate solution $\optSol$. 
Provided
\begin{equation}
\label{eqv:epsilon_min_psi_p_tau_p}
		 \epsilon(A,b, c)  \leq  \mathcal{O} \left( \sqrt{n\rho} \min  \left(  \sqrt{ \rho }, 1 \right)  \right),
\end{equation}
where $\rho$ is defined in~\eqref{eqv:rho}, 
there exists a sufficiently small $ \bar{\lambda}(A,b,c,x^{*},s^{*})>0$ 
such that
\[
\mu^{\max}_{\lambda} > \mu^{\max},
\] 
for all $0<\lambda = \alpha \bar{\lambda}<\bar{\lambda}$, where $\alpha \in (0,1)$ and where
$\mu^{\max}_{\lambda}$ is defined in~\eqref{eqv:mupmax} and $\mu^{\max}$ in~\eqref{eqv:mu_max}.
\end{theorem}
\begin{proof}
Applying Theorem~\ref{thm-mu_max_lambda_g_hat_mu_max_lambda} with $\lambda = 0$ and so replacing $\epsilon(A,b_{\lambda}, c_{\lambda}) $ with $ \epsilon(A,b,c)$, we deduce
\[
	\mu^{\max} < \frac{\rho\epsilon(A,b, c)}{4 n \left( \sqrt{n} + \epsilon(A,b, c) \right)}.
\] 
From Lemma~\ref{lem-hat_epsilon_g_epsilon}, we have $\epsilon(A,b_{\lambda},c_{\lambda}) > \epsilon(A,b,c)$. This and the definition of $\mu^{\max}_{\lambda}$ in~\eqref{eqv:mupmax}~give
\[
 \frac{\rho\epsilon(A,b, c)}{4 n \left( \sqrt{n} + \epsilon(A,b, c) \right)} < \mu^{\max}_{\lambda}.  
\]
\qed \end{proof}

Theorem~\ref{thm-mu_max_lambda_g_mu_max} implies that if $\epsilon(A,b, c)$ is sufficiently small, we may  find the optimal active set of~\eqref{eqv:originalPD} `sooner' if we solve~\eqref{eqv:perPD} using a primal-dual path-following {\ipm} than if we solve~\eqref{eqv:originalPD}.

\Remark
When~\eqref{eqv:originalPD} has a unique solution $\optSol$, we have 
\[
	\epsilon(A,b,c) = \min\left(\,\,\min_{i\in \Iactv} x^{*}_{i}, \,\,
\min_{i\in \Actv} s^{*}_{i}   \right) \leq \min_{i\in \Iactv} x^{*}_{i} = \psi_{p}.
\]
Note that according to~\cite{Mangasarian}  $\tau_{p}, \tau_{d} = \mathcal{O}(1)$ numerically. Thus provided $\psi_p > 1$ or $n$ is sufficiently large,~\eqref{eqv:epsilon_min_psi_p_tau_p} is satisfied.
We illustrate this in an example next.
 
\paragraph*{A simple example of predicting the optimal~\eqref{eqv:originalPD} active set using perturbations.}To~illustrate our results in this section, consider the following simple example
\begin{equation}
\label{eqv:example}
	\min \,\,x_{1} + 2x_{2} \quad \text{subject to} \quad x_{1} + x_{2} = 1,\,\, x_{1} \geq 0,\,\, x_{2} \geq 0,
\end{equation}
with the optimal solution $ x^{*} = (1,0)$ and $y^{*} = 1 $, $s^{*} = (0,1)$.\hspace*{2.35mm}Thus~\eqref{eqv:example} has a unique and primal-dual nondegenerate solution with optimal active set $\Actv(x^{*}) = \{2\}$, and so $\psi_{p} = \epsilon(A,b,c) = 1$.  
Let the vector of perturbations be $\lambda = \alpha  (1, 5)$ where $\alpha  = 10^{-2}$. The perturbed problems~\eqref{eqv:perPD} also have a unique solution $x^{*}_{\lambda} = (1+5\alpha, -5\alpha)$, $y^{*}_{\lambda} = 1+\alpha$ and $s^{*}_{\lambda} = (-\alpha, 1-\alpha)$. 
So $\epsilon(A,b_{\lambda},c_{\lambda})$ $ = \min\left( 1+6\alpha, 1+4\alpha \right) = 1 + 4\alpha = 1.04 $.

First we verify the conditions in Theorem~\ref{thm-predicted_A_equals_Actual_A}, which are needed in both Theorems~\ref{thm-mu_max_lambda_g_hat_mu_max_lambda} and~\ref{thm-mu_max_lambda_g_mu_max}. Since it is not clear how to deduce the value of  $\tau_p$ and $\tau_d$,  we estimate them numerically\footnote{%
We estimate $\tau_p$ and $\tau_d$ from their definition in~\eqref{eqv:iter_error_bound}, namely, we solve the following optimisation problem in {\sc matlab},
$ \max \|x-x^{*}\| \slash (r(x,s)+w(x,s))$ subject to $(x,y,s) \in \sfsp$,
where $r(x,s)$ and $w(x,s)$ are defined in~\eqref{eqv:error_bounds_feasible_x_s} and $\sfsp$ in~\eqref{eqv:strictlyFeasibleSet_per}; 
similarly for $\tau_d$. 
} and it turns out that $\tau_{p} \approx \tau_{d} \approx 0.8$. We set the cut-off constant $C$ that separates the active and inactive constraints to be $C = \frac{ \psi_{p}}{2}  = 0.5$ and verify that $\|\lambda\| = \sqrt{26}\alpha <  \frac{\psi_{p}}{16\max(\tau_{p},\tau_{d})}< 1$. Thus the conditions in~\eqref{eqv:thm_bound_lambda_value_C_primal} are satisfied. Based on Theorem~\ref{thm-predicted_A_equals_Actual_A}, we can predict the original optimal active set when $\mup$ is less than $\mup^{\max} 
\approx 0.0662 $.  

Next we verify Theorems~\ref{thm-mu_max_lambda_g_hat_mu_max_lambda} and~\ref{thm-mu_max_lambda_g_mu_max}.
From~\eqref{thm-mu_max_lambda_g_hat_mu_max_lambda}, we get $\rho \approx 1.25$, and so $\sqrt{n\rho} \min  \left(  \sqrt{\rho}, 1  \right)  \approx 1.58$. Thus  $0<\epsilon(A,b,c) < \epsilon(A,b_{\lambda},c_{\lambda}) < \sqrt{n\rho} \min  \left(  \sqrt{\rho}, 1  \right)$, which implies that conditions~\eqref{eqv:hat_mu_max_g_mu_max} and~\eqref{eqv:epsilon_min_psi_p_tau_p} are satisfied.
For the constant $\gamma$, it is common to choose a small value to have a large neighbourhood of the central path; set $\gamma = 0.01$. Then from~\eqref{eqv:mu_bar_max_p} and~\eqref{eqv:mu_max}, we have $\bar{\mu}^{\max}_{\lambda}  \approx 0.0027 < \mup^{\max} $ and $\mu^{max} = 0.0025 < \mup^{\max} $. This implies that when we use perturbations, we can predict the original optimal active set sooner than the perturbed active set or the original active set without perturbations. Furthermore, the threshold values (constant $C$) needed to separate the active constraints from the inactive ones for predicting the perturbed active set and the original active set without perturbations are 0.0052 and 0.005 respectively, both of which are much smaller than the cut-off $C = \frac{ \psi_{p}}{2}  = 0.5$ for predicting the original optimal active set using~perturbations.

%%%%%%%%%%% Section: Numerical results %%%%%%%%%%%%%%

\section{Numerical results}
\label{numeric}
%\resetcounters

\subsection{The perturbed algorithm and its implementation}
\label{Implementation}
All numerical experiments in this section employ an infeasible primal-dual path-following interior point method structure~\cite[Chapter 6]{wright} whether applied to~\eqref{eqv:perPD} or~\eqref{eqv:originalPD}. 
The perturbed algorithm is summarised in Algorithm~\ref{alg:perAlg}.

\algo{alg:perAlg}{%
Perturbed Algorithm Framework.}{%

\textbf{Given} perturbations $(\lambda^{0}, \phi^{0})>0$ and a starting point $(x^0,y^0,s^0)$ with $(x^0 + \lambda^{0}, s^0 + \phi^{0}) > 0$,
\textbf{for} $k = 0,1,2,\ldots$
\begin{itemize}\itemsep2pt \parskip1pt \parsep0pt
	\item[] \textbf{solve} the perturbed system~\eqref{eqv:perturbedCentralPath} using Newton's method, namely
	\begin{equation}
	\label{eqv:perNewton_practical}
	\begin{bmatrix}
	A & 0       &  0  \\
	0 & A^{T} & I   \\
	S^{k} + \Phi^{k} & 0 & X^{k} + \Lambda^{k} 
	\end{bmatrix}
	\begin{bmatrix}
	\Delta x^{k} \\
	\Delta y^{k} \\
	\Delta s^{k}
	\end{bmatrix}
	= - 
	\begin{bmatrix}
	\,\, Ax^{k}-b\,\,  \\
	\,\, A^{T}y^{k} +s^{k} - c\,\,    \\
	\,\,  \left( X^{k}+\Lambda^{k} \right) \left( S^{k} + \Phi^{k} \right)e -  \sigma^{k}\mup^{k}e \,\,
	\end{bmatrix},
	\end{equation}
	\hspace*{8ex}where $\sigma^{k} \in [ 0 ,1 ]$ and 
	\begin{equation}
	\label{eqv:mupk}
		\mup^{k} = \frac{(x^{k}+\lambda^{k})^{T}(s^{k}+\phi^{k})}{n};
	\end{equation}
	\item[] \textbf{set} $ x^{k+1} = x^{k} + \alpha^{k}_{p} \,\Delta x^{k}$ and $(y^{k+1},s^{k+1}) = (y^{k} ,s^{k}) + \alpha^{k}_{d} \,( \Delta y^{k}, \Delta s^{k})$,  where  
	\item[] \hspace*{8ex}$( \alpha_{p}^{k}, \alpha_{d}^{k} )$ is chosen such that $(\, x^{k+1} + \lambda^{k},\,s^{k+1} + \lambda^{k}\,) > 0$;
	\item[] \textbf{predict} the optimal active set of~\eqref{eqv:originalPD} and denote by $\Actv^{k}$;
	\item[] \textbf{terminate} if some termination criterion is satisfied;
	\item[] \textbf{calculate} $(\lambda^{k+1}, \phi^{k+1})$ possibly by shrinking $(\lambda^{k}, \phi^{k})$ so that 
	\item[] \hspace*{8ex}$(x^{k+1} + \lambda^{k+1}, s^{k+1} + \phi^{k+1}) >0$;
\end{itemize}
\textbf{end (for)}.
}

\paragraph{Algorithm without perturbations.}
For comparison purposes, we refer to the algorithm with no perturbations (Algorithm~\ref{alg:perAlg} with $\lambda = \phi = 0$) as \textbf{{\refstepcounter{algo}\label{alg:unperAlg}}Algorithm {\thealgo}}.
We denote the duality gap for Algorithm~\ref{alg:unperAlg} as $\mu^{k}$, which is equivalent to $\mup^{k}$ in~\eqref{eqv:mupk} with $\lambda^{k} = \phi^{k} = 0$.

\paragraph{Starting point.}
We use the starting point proposed by Mehrotra for~\eqref{eqv:originalPD}~\cite[Section 7]{mehrotra} as  starting point for both Algorithms~\ref{alg:perAlg} and~\ref{alg:unperAlg}.
%\footnote{%
(We have also tested the case when Algorithm~\ref{alg:perAlg} is initialised from Mehrotra's starting point for~\eqref{eqv:perPD}. This change did not affect our results in any significant way, suggesting some level of robustness.) 
%}

\paragraph{Solving the Newton system~\eqref{eqv:perNewton_practical}.}
We follow~\cite[Chapter 11]{wright} and solve the augmented system form of~\eqref{eqv:perNewton_practical}. Also we set $\sigma^{k} = \min(0.1, 100\mup^{k})$.

\paragraph{Choice of perturbations.}
In our theory, we used the same vector of perturbations for both primal and dual variables. For better numerical efficiency, we have different perturbations $\lambda$ and $\phi$ for primal and dual variables respectively. We set the initial perturbations to be $\lambda^{0} = \phi^{0} = 10^{-2} e$, where $e$ is a vector of ones. (We have done experiments to explore the sensitivity of our algorithm to the value of the initial perturbations. 
For example, choosing  $\lambda^{0} = \phi^{0} = 10^{-1}e$  yields a high false-prediction ratio (proportion of mistakes). Perturbations of order $10^{-2}$ and $10^{-3}$ yield quickly a good approximation of the original~\eqref{eqv:originalPD} active set. For $\lambda^{0} = \phi^{0} = 10^{-4}e$, the perturbed algorithm starts to behave similarly to the unperturbed one simply because the perturbations are too small.)

\paragraph{Choice of stepsize.}
We choose a fixed, close to 1, fraction of the stepsize to the nearest constraints' boundary in the 
primal and dual spaces, respectively.

\paragraph{Shrinking the perturbations.}
One possible reason for getting a poor prediction of the active set is that the current perturbations are too large.   
So after we get the new iterate $(x^{k+1},y^{k+1},s^{k+1})$, we shrink the perturbations accordingly.
Assume $t^{k+1} = \min(x^{k+1})$ and $v^{k+1} = \min(s^{k+1})$. We update the perturbations as follows,
\begin{equation}
\label{eqv:updatePersAlg4_x}
    \quad \lambda^{k+1} =
        \begin{cases}
            \,\, \eta \lambda^{k},   &  \text{if $t^{k+1} >0$}\\
            \,\, (1-\zeta)\lambda^{k} +\zeta(-t^{k+1} )e, &  \text{if $t^{k+1} \leq0$}\\
            \end{cases}
            ,\quad
\nonumber
\end{equation}
and
\begin{equation}
\label{eqv:updatePersAlg4_s}
    \quad \phi^{k+1} =
        \begin{cases}
            \,\, \eta \phi^{k},   &  \text{if $v^{k+1} > 0$}\\
            \,\, (1-\zeta)\phi^{k} + \zeta(-v^{k+1} ) e , &  \text{if $v^{k+1} \leq0$}\\
            \end{cases}
            ,\quad
\nonumber
\end{equation}
where $\eta \in (0,1]$ and  $\zeta \in (0,1)$. It follows that  $ x^{k+1} + \lambda^{k+1} > 0 $ and $ s^{k+1} +\phi^{k+1} > 0. $ We observed in our numerical experiments that when solving nondegenerate problems, it is better to shrink faster, roughly keeping the perturbations to be $\mathcal{O} (\mup)$. When solving degenerate problems however, it is better to shrink slower, at a rate of $\mathcal{O} (\sqrt{\mup})$.
It is difficult and often impossible to distinguish a priori between degenerate and nondegenerate cases. After several numerical trials, we chose to set $\eta = 1$ and $\zeta = 0.5$. 

\paragraph{Active-set prediction.}
In our theory, we considered that all variables less than a threshold are active at the solution. In practice, we apply a more complex strategy, inspired by~\cite[Step 3 in Procedure 8.1]{bakry}.   
We partition the index set $\{ 1,2,\ldots,n  \}$ into three sets, $\Actv^{k}$ as the predicted active set, $\Iactv^{k}$ as the predicted inactive set and $\Nactv^{k} = \{ 1,2,\ldots,n  \} \backslash \left( {\Actv}^{k} \cup {\Iactv}^{k}\right)$ which includes all undetermined indices, and
during the running of the algorithm, we move indices between these sets according to the following criteria,
\begin{equation}
\label{eqv:actvPredctImplmtnCriteria}
    x^{k}_{i} < C
    \quad\text{and}\quad
    s^{k}_{i} > C,
\end{equation}
where $C$ is a constant user-defined threshold.
Theorem~\ref{thm-bound_active_set} guarantees the above criteria~\eqref{eqv:actvPredctImplmtnCriteria} are promising, as we are predicting the original optimal active set by estimating the intersection of $\psas(s^{k}_{i})$ and $\pacs(x^{k}_{i})$.
Initialise $\Actv^{0} = \Iactv^{0} = \emptyset$ and $\Nactv^{0} = \{  1,2,\dots,n \}$.
An index is moving from $\Nactv^{k}$ to $\Actv^{k}$ if \eqref{eqv:actvPredctImplmtnCriteria} is satisfied for two consecutive  iterations, otherwise from $\Nactv^{k}$ to $\Iactv^{k}$.
We move an index from $\Actv^{k}$ to $\Nactv^{k}$ if \eqref{eqv:actvPredctImplmtnCriteria} is not satisfied at the current iteration.
An index is moving from $\Iactv^{k}$ to $\Nactv^{k}$ if \eqref{eqv:actvPredctImplmtnCriteria} is satisfied at the current iteration.
In our implementation, we choose $C = 10^{-5}$.
Procedure~\ref{alg:actvPrediction} in
Appendix~\ref{apd:ActivesetPredictionProcedure} contains a pseudocode
of our active-set prediction technique.
Our strategy enables us to make use of both primal and dual information which may be beneficial given Theorems~\ref{thm-bound_active_set}--\ref{thm-predicted_A_plus_equals_Actual_A_plus}. 

\paragraph{Termination.} Termination criteria will be defined for each set of tests.

\subsection{Numerical results}
\label{subsec-num_results}
\subsubsection{Test problems}
\label{subsub-test_probs}
\paragraph{Randomly generated test problems (TS1).}\label{test-random_set}
We first randomly generate the number of constraints $m \in (10,200) $, the number of variables $n \in (20,500) $ and density of nonzero entries in $ A $ within $ (0.4,0.8)$, where $m < n$, $2m < n < 7m$. Then randomly generate a matrix $A \in \Real^{m\times n}$ of given density and a point $\pntp \in \Real^{n}\times\Real^{m}\times\Real^{n}$ with $x \geq 0$, $s \geq 0$ and density about 0.5. Finally we generate $b$ and $c$ by letting $b = Ax$ and $c = A^{\top}y+s$. Thus $(x,y,s)$ serves as a feasible point. Problems generated this way are generally well-conditioned and primal nondegenerate.
%\footnote{%
%We checked the degeneracy of 100 test problems generated this way, by looking at the vertex solutions obtained by the {\sc matlab} simplex solver.
%Majority of these problems are dual degenerate.} 
This test set is inspired by the random problem generation approach in~\cite[Section 8.3.4]{lpmatlab}.Whenever we use this test set, (the same) 100 problems are generated.

\paragraph{Randomly generated primal-dual degenerate test problems (TS2).}\label{test-random_pd_degen_set} 
Instead of generating a feasible point as for {\tspndeg}, we generate $(x,y,s)$ with $x\geq0$, $s\geq0$, $x_{i}s_{i} = 0$ for all $i \in \{1,\ldots,n\}$ so that the number of nonzeros of $x$ is strictly less than $ m$ and that of $s$ is strictly less than $n-m$. Then get $A$, $b$, $c$ as for  {\tspndeg}. Thus $(x,y,s)$ serves as a primal-dual degenerate solution. 100 problems are also generated for this test set.

\paragraph{Netlib problems (TS3).}\label{test-netlib_set}
Most Netlib
%\footnote{%
%Netlib is a collection of standard {\lp} test problems. The original test problems can be found at \url{http://www.netlib.org/lp/data/}. We use a {\sc matlab} version of the same test set, which is obtained from \url{http://www.math.ntu.edu.tw/~wwang/cola_lab/test_problems/netlib_lp/}.
%} 
test problems are not in the standard form. We reformulate them into the standard form by introducing slacks.
Since our implementation is basic, in {\sc matlab}, and mainly for illustration, we choose a subset of problems in Netlib with the number of primal variables less than $5000$ (including the slack variables). See Table~\ref{tab-sumCrossover} for the list of the 37 Netlib problems selected.

\subsubsection{On the accuracy of active-set predictions using prediction ratios}
Assume $\Actv^{k}$ is the predicted active set at iteration $k$ and $\Actv$ is the actual optimal active set. To compare the accuracy of the predictions, we introduce the following three prediction ratios.
\begin{itemize}%[itemsep=-1mm]
\item False-prediction ratio    $ = \frac{|\Actv^{k} \,\setminus\, (\Actv^{k} \,\cap\, \Actv)|}{|\Actv^{k} \,\cup\, \Actv|}$. 
\item Missed-prediction ratio  $ = \frac{|\Actv \,\setminus\, (\Actv^{k} \,\cap\, \Actv)|}{|\Actv^{k} \,\cup\, \Actv|}$.
\item Correction ratio              $ = \frac{| \Actv^{k} \,\cap\, \Actv |}{|\Actv^{k} \,\cup\, \Actv |}$.
\end{itemize}
False-prediction ratio  measures the degree of incorrectly identified active constraints, missed-prediction ratio measures the degree of incorrectly rejected active constraints and correction ratio shows the accuracy of the prediction. 
All three ratios range from 0 to 1. If the predicted set is the same as the actual optimal active set, correction ratio is 1. The main task for this test is to compare the three measures for Algorithms~\ref{alg:perAlg} and~\ref{alg:unperAlg}. 

When an {\lp} problem has multiple solutions, the active set of a vertex solution is different from that of the strictly complementary solutions (about $17\%$ difference on average for {\tspndeg} and $21\%$ for {\tspddeg}). To understand which active set do the (perturbed) Algorithm~\ref{alg:perAlg} and the (unperturbed)~\ref{alg:unperAlg} predict, we terminate both algorithms at the same iteration and compare the predicted active sets with the actual optimal active sets obtained from an interior point solver and a simplex solver\footnote{\label{fnt-actualActv}
We obtain the `actual optimal active set' by solving the problem using {\sc matlab}'s solver {\sc linprog} with the 'algorithm' option set to interior point or simplex and considering all variables less than $10^{-5}$ as active.}.

\paragraph{Prediction ratios for test sets {\tspndeg} and {\tspddeg}. }
In Figures~\ref{fig-cr-nondegen} and~\ref{fig-cr-degen}, we present the results for {\tspndeg} (left) and {\tspddeg} (right). The x-axis shows the number of interior point iterations at which we terminate the algorithms.
In each figure, the first three plots (from left to right, top to bottom) show the average value of the three measures mentioned above for the test problems in question. 
The last plot at the bottom right corner presents the corresponding $\log_{10}$ scaled relative {\kkt} residuals. 
We measure the relative residual~by
\begin{equation}
\label{eqv:rel_resid}
\mbox{relRes}^{k} = \frac{|| \left( Ax^{k} - b, A^{T}y^{k} + s^{k}-c, \left( X^{k}+\Lambda^{k} \right) \left( S^{k} + \Phi^{k} \right)e - \mup^{k}e \right)  ||}{1+\max \left(  || b ||, || c ||  \right)}.
\end{equation}
There are four lines in each plot, representing the prediction ratios by comparing the active set from Algorithm~\ref{alg:perAlg} with that from {\sc matlab}'s simplex solver (solid red line with circle) and from {\sc matlab}'s {\ipm} (solid black line with square sign), and Algorithm~\ref{alg:unperAlg} with simplex (dashed green line with diamond sign) and with {\ipm} (dashed blue line with star) respectively.

\begin{figure}[h]
\fboxsep= 3.5pt%padding thickness
\fboxrule=0.5pt%border thickness
\begin{minipage}[b]{0.45\linewidth}
\centering\fbox{
\includegraphics[width=\textwidth]
{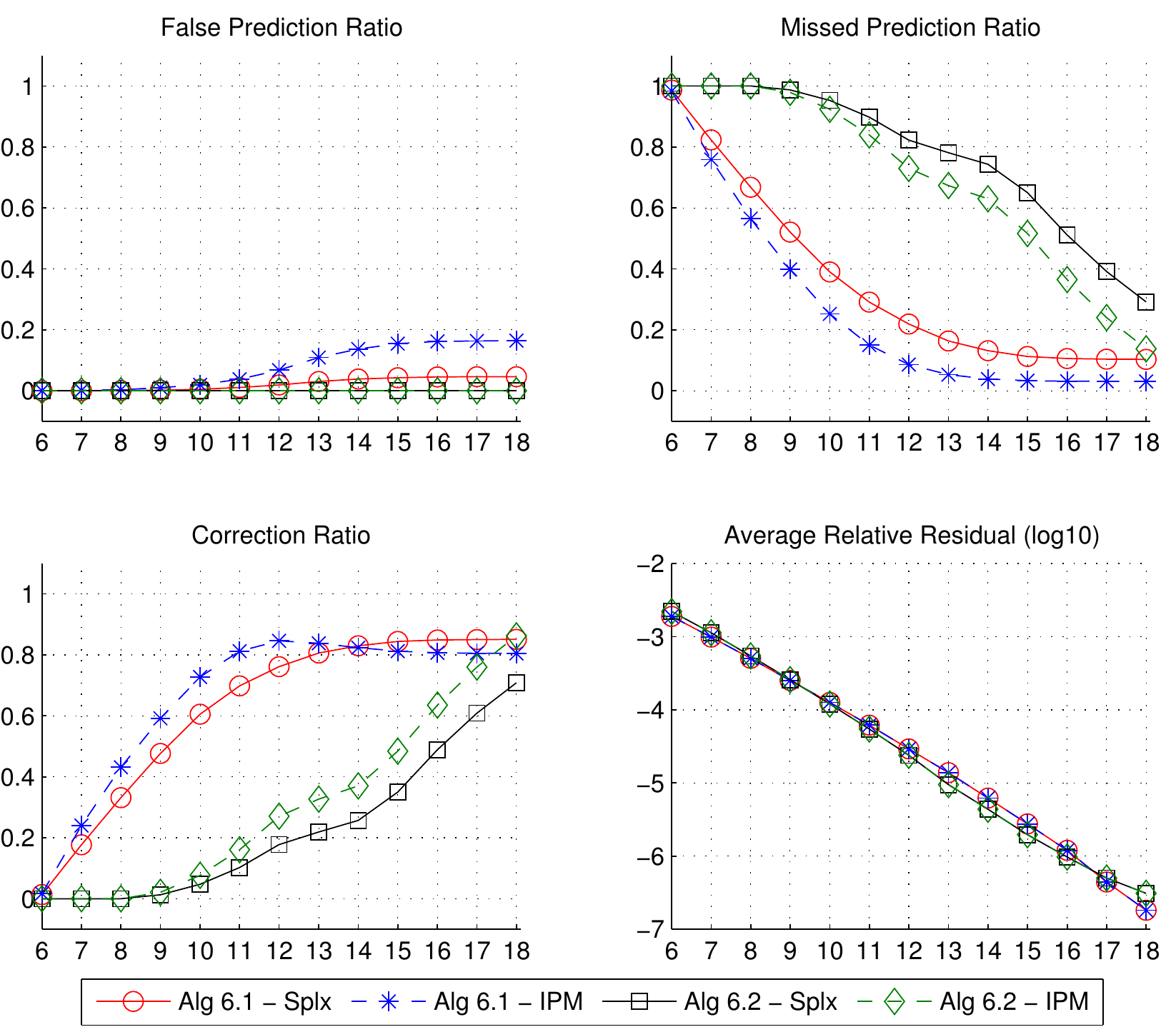} }
\caption{\small Prediction ratios for randomly generated problems}
\label{fig-cr-nondegen}
\end{minipage}
\hspace{0.4cm}
\begin{minipage}[b]{0.45\linewidth}
\centering\fbox{
\includegraphics[width=\textwidth]
{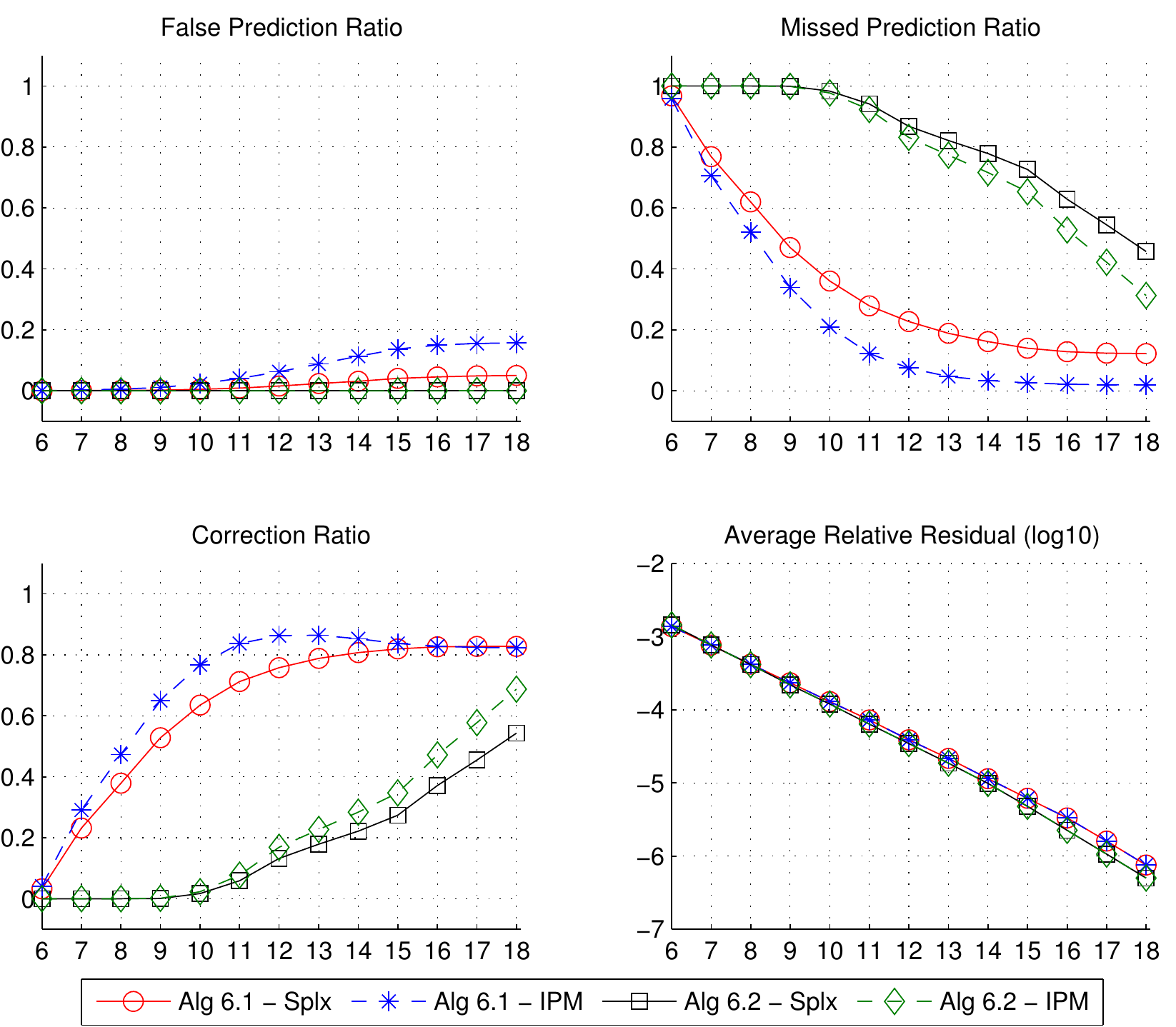} }
\caption{\small Prediction ratios for randomly generated primal-dual degenerate}
\label{fig-cr-degen}
\end{minipage}
\end{figure}

\begin{itemize}\itemsep1.5pt \parskip1pt \parsep1pt
\item Figures~\ref{fig-cr-nondegen} and~\ref{fig-cr-degen} show that the average correction ratios for Algorithm~\ref{alg:perAlg} are at least as good and generally better than those for  Algorithm~\ref{alg:unperAlg}. Thus it seems that using perturbations can only improve the active-set prediction capabilities of {\ipms}.
\item Algorithm~\ref{alg:unperAlg} is in fact an interior point solver applied to~\eqref{eqv:originalPD} which approaches a strictly complementary~\eqref{eqv:originalPD} solution. This is confirmed by having better correction ratio when comparing Algorithm~\ref{alg:unperAlg} with the {\ipm} than when comparing it with the simplex.
\item Due to the fact that the active set from the {\ipm} (the strictly complementary partition) contains less elements than that from the simplex (vertex solution), the correction ratio of Algorithm~\ref{alg:perAlg} compared with the {\ipm} is higher than that compared with the simplex at the early stage. However the false-prediction ratios of the former climb up to about $0.16$ at the end for both test cases. Thus the corresponding correction ratios go down. The false-prediction ratios of comparing Algorithm~\ref{alg:perAlg} with simplex are much less, about $0.05$ for both cases. The behaviours of the false-prediction ratios seem to imply that Algorithm~\ref{alg:perAlg} predicts the active set of a vertex solution (that may not be the same vertex as obtained by the simplex solver).
\item  After 18 iterations, the correction ratios do not reach 1. This is due to ill-conditioning which prevents us from solving any further. For this 18\textsuperscript{th} iteration,  the perturbations are not zero, they are about $\mathcal{O} (10^{-2})$ for problems in {\tspndeg} and $\mathcal{O} (10^{-3})$ for the degenerate problems in {\tspddeg}, and on average the relative residual~\eqref{eqv:rel_resid} is lower than~$10^{-6}$.
\end{itemize}

\paragraph{Can Algorithm~\ref{alg:perAlg} predict the optimal active set of~\eqref{eqv:originalPD} sooner than it obtains the strictly complementary partition of~\eqref{eqv:perPD}?}\label{exp-predict_sooner_than_scp_per}
In Figures~\ref{fig-cr2-nondegen} and~\ref{fig-cr2-degen}, besides comparing the predicted active set of~\eqref{eqv:originalPD} with the actual active set of~\eqref{eqv:originalPD}, we also compared the predicted active set of~\eqref{eqv:perPD}\footnote{\label{fnt-val-of-pert}
Here, for each of the test problems, we set $\lambda$ in~\eqref{eqv:perPD}   to be the value of the perturbations when
terminating Algorithm~\ref{alg:perAlg} at the $18^{\text{th}}$
iteration. We then apply Algorithm~\ref{alg:unperAlg} to the
equivalent form~\eqref{eqv:perPD_formPQ} of~\eqref{eqv:perPD}, which
means we solve the perturbed problem using an {\ipm} method and predict the active set of the perturbed problem on the way.} with the actual active set of~\eqref{eqv:perPD} obtained from a simplex solver (solid purple line with downward-pointing triangle) and an {\ipm} solver (dashed brown line with upward-pointing triangle), respectively; see Footnote~\ref{fnt-actualActv} on the choice of solvers. We again use the test sets {\tspndeg} and {\tspddeg}. 

We can see that on average Algorithm~\ref{alg:perAlg} can predict a better active set for~\eqref{eqv:originalPD} than when applying Algorithm~\ref{alg:unperAlg} to predict the active set of~\eqref{eqv:perPD}. Furthermore, for test case {\tspndeg}, before iteration 12, Algorithm~\ref{alg:unperAlg} cannot predict much concerning the active set of~\eqref{eqv:perPD} while Algorithm~\ref{alg:perAlg} already has an increasingly accurate prediction for the active set of~\eqref{eqv:originalPD} (approximately $80\%$ of the active set of~\eqref{eqv:originalPD} at iterations 12). We can draw similar conclusions for {\tspddeg}. 

\begin{figure}[h]
\fboxsep= 3.5pt%padding thickness
\fboxrule=0.5pt%border thickness
\begin{minipage}[b]{0.455\linewidth}
\centering\fbox{
\includegraphics[width=\textwidth]
{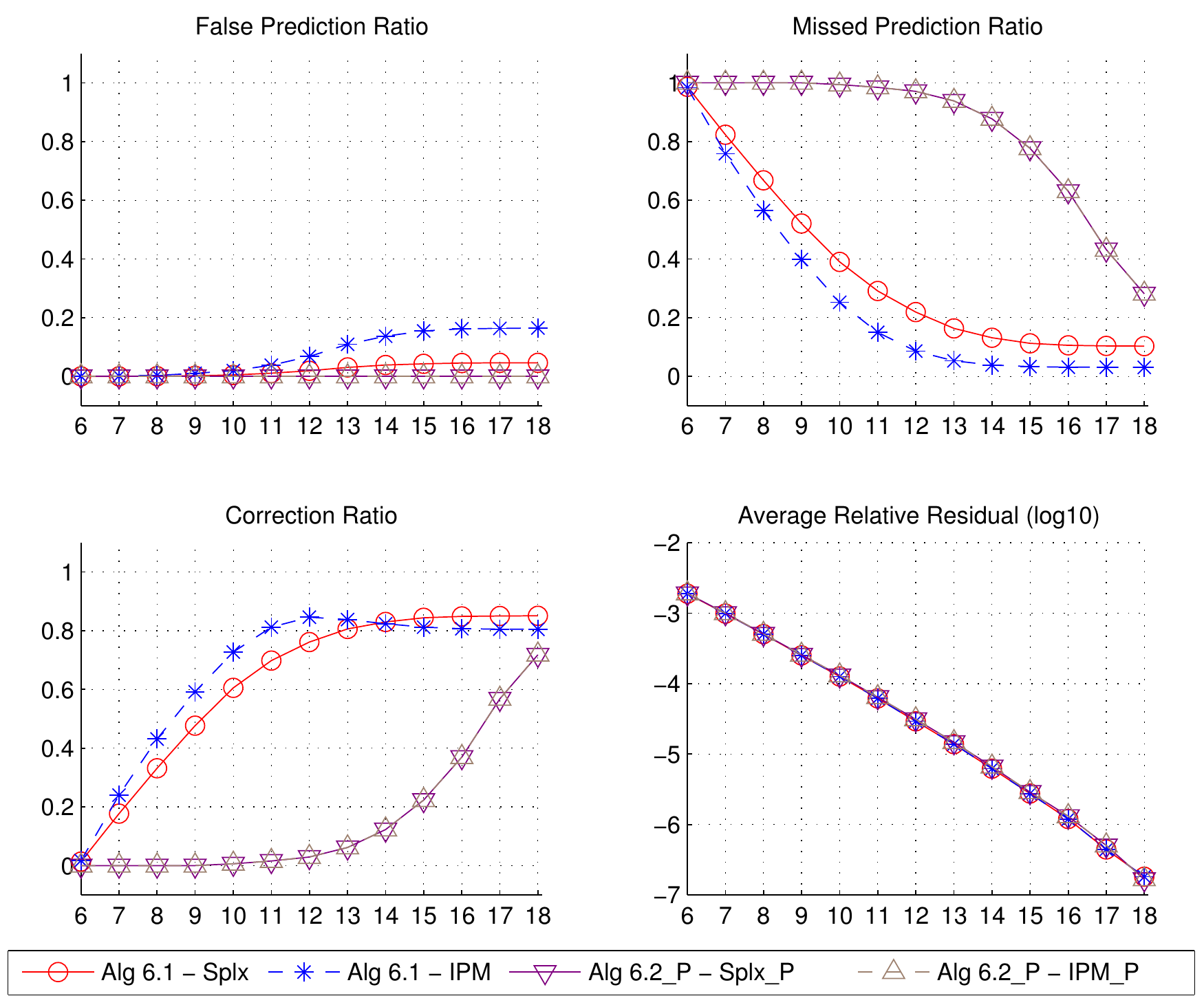} }
\caption{\small Comparing perturbed active-set predictions for (\tspndeg)}
\label{fig-cr2-nondegen}
\end{minipage}
\hspace{0.4cm}
\begin{minipage}[b]{0.455\linewidth}
\centering\fbox{
\includegraphics[width=\textwidth]
{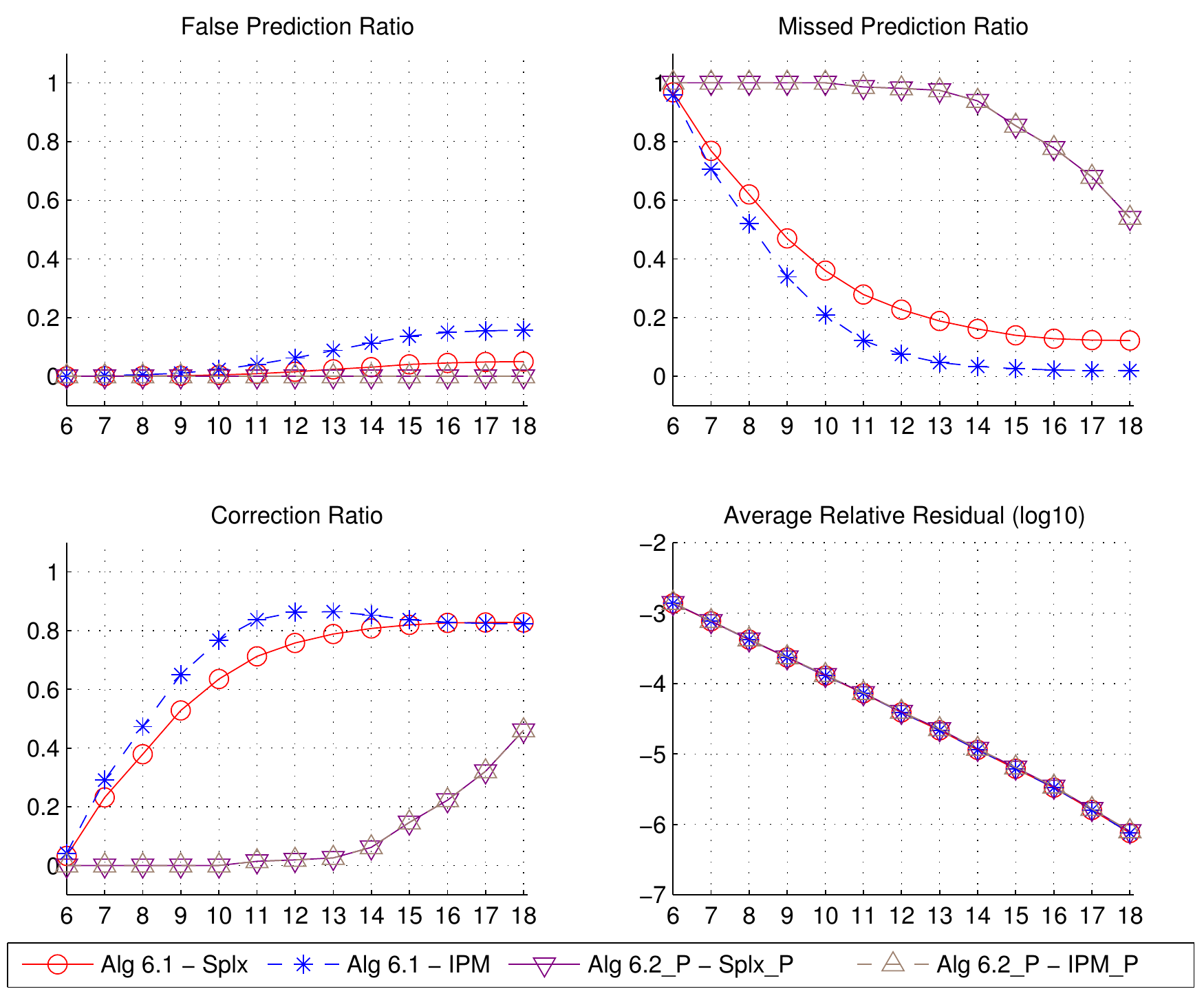} }
\caption{\small Comparing perturbed active-set predictions for (\tspddeg)}
\label{fig-cr2-degen}
\end{minipage}
\end{figure}

\paragraph{On the difference between the optimal active set of~\eqref{eqv:perPD} and that of~\eqref{eqv:originalPD}.}\label{exp-pacyv_diff_oactv}
Note~that, to yield good performance, we do not need to force the
active set of~\eqref{eqv:perPD} (as defined in Footnote \ref{fnt-val-of-pert}) to be the same as the (original) active set of~\eqref{eqv:originalPD}. In fact, for most test problems in both {\tspndeg} and {\tspddeg}, this does not hold. 
When perturbations are not so small, namely $\mathcal{O}( 10^{-2} )$ or $\mathcal{O}( 10^{-3} )$, which is the case even in the last {\ipm} iterations in Figures~\ref{fig-cr2-nondegen} and~\ref{fig-cr2-degen}, the perturbed optimal active set is different from the original optimal active set for $98\%$ of the test problems in {\tspndeg} and all test problems in {\tspddeg}. Furthermore, for problems in {\tspndeg}, the average difference between the strictly complementary partition of~\eqref{eqv:perPD} and that of~\eqref{eqv:originalPD} is as high as $33\%$ and the difference between the active set at a vertex solution of~\eqref{eqv:perPD} and that of~\eqref{eqv:originalPD} is about $15\%$ on average; for {\tspddeg}, the average difference between the strictly complementary partitions of~\eqref{eqv:perPD} and~\eqref{eqv:originalPD} is about $29\%$ and the difference between active sets at vertex solutions is $17\%$ on average. Another interesting observation is that, for both {\tspndeg} and {\tspddeg}, over $90\%$ of the perturbed problems have a unique and nondegenerate solution, regardless of the uniqueness or degeneracy of the original test problems. This is the reason why the predictions of the perturbed active set when comparing with simplex and {\ipm} are identical in Figures~\ref{fig-cr2-nondegen} and~\ref{fig-cr2-degen}.

\paragraph{Prediction ratios for test set {\tsn} (Netlib test
  problems).}

Figure~\ref{fig-correction-ratios-netlib} gives the prediction ratios
for the Netlib test problems in {\tsn}. In contrast to the randomly
generated problems in {\tspndeg} and {\tspddeg}, the number of
iterations required by Netlib test problems to reach sufficient accuracy to allow
meanginful predictions varies significantly from problem to problem
(and so we cannot in general compare the prediction ratios at some
fixed, predefined iterations). Thus
 to test the prediction ratios on  {\tsn} problems, we follow a
 slightly different procedure, inspired by \cite{bakry}. 
%\footnote{A similar methodology was also used in \cite{bakry} to
%compare the performance of different indicators.}
For each  {\tsn} test problem, we first
 solve it to optimality using Algorithm~\ref{alg:unperAlg}, requiring the
relative residual in~\eqref{eqv:rel_resid} (with $\lambda^k=\phi^k=0$) to be less than $10^{-8}$
and we record the total number of iterations needed to reach this accuracy, say $M$.  
Then we  calculate the prediction ratios on this same test problem
for Algorithms~\ref{alg:perAlg} and~\ref{alg:unperAlg} (with {\sc
  matlab}'s simplex and {\ipm} output) over the last $10$ iterations
preceding (and including) the $M^{\text{th}}$ iteration. We then average the
prediction ratios for each algorithm on all {\tsn} test problems at
each of the $M-i$ iterations for $i\in \{0,\ldots,9\}$.
Again, in Figure~\ref{fig-correction-ratios-netlib}, there are four lines in each plot, representing the
prediction ratios by comparing the active set from
Algorithm~\ref{alg:perAlg} with that from {\sc matlab}'s simplex
solver (solid red line with circle) and from {\sc matlab}'s {\ipm}
(solid black line with square sign), and Algorithm~\ref{alg:unperAlg}
with simplex (dashed green line with diamond sign) and with {\ipm}
(dashed blue line with star), respectively. The bottom right figure
plots the  corresponding average relative
residual~\eqref{eqv:rel_resid} on a $\log_{10}$ scale.
\begin{figure}[htb]
\fboxsep= 10.5pt%padding thickness
\fboxrule=0.5pt%border thickness
\centering\fbox{
\includegraphics[width=0.50\textwidth]
{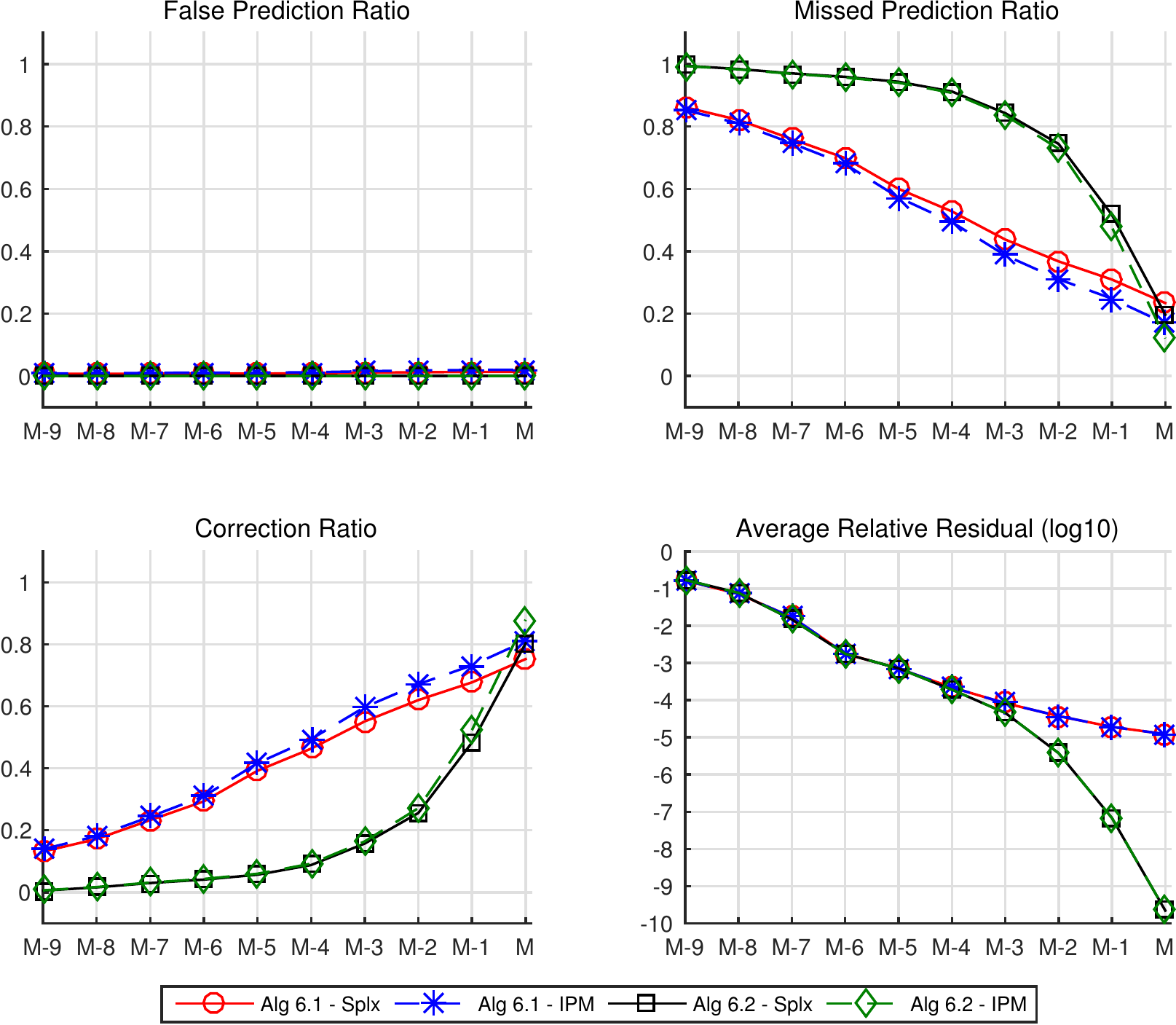} }
\caption{Comparing prediction ratios for the Netlib problems in
  {\tsn}. $M$ denotes the (variable) total number of iterations
  required to solve each test problem to a given accuracy (thus $M$ is
  generally different for each test problem).} 
\label{fig-correction-ratios-netlib}
\end{figure}
\begin{itemize}\itemsep1pt \parskip0.5pt \parsep0.5pt
\item
Figure~\ref{fig-correction-ratios-netlib} shows that using
perturbations can only improve the active-set prediction capabilities
of {\ipm}s on the tested Netlib problems,
especially in the earlier stages of the runs, when the relative
residuals are not too small. For example, the average correction
ratios when using Algorithm~\ref{alg:perAlg} are about three times
better than those of using Algorithm~\ref{alg:unperAlg} at iteration $M-5$, when the average relative residual is just slightly less than $10^{-3}$.  
\item The correction ratios for both algorithms are slightly worse when
compared with the vertex solution from {\sc matlab}'s simplex solver
than with the strictly complementary solution from {\sc matlab}'s {\sc
ipm} solver; thus it is unclear in this case whether
Algorithm~\ref{alg:perAlg} gets us closer to a vertex  or an
interior solution of the original problem (the cross-over to simplex
results for  {\tsn} in the next section seem to indicate the former is
still the case). 
\item The average relative residual  in the bottom
right plot is still quite large over the last few iterations for
Algorithm~\ref{alg:perAlg} indicating that we have not solved the
perturbed problems to high accuracy while still being able to predict well
the optimal active set of the original problem, as desired.
\end{itemize}

\subsubsection{Crossover to simplex}
In this section, we test the efficiency of our active-set predictions using perturbations when crossing over to a simplex method after some {\ipm} iterations.
We choose {\sc lp\_solve}~\cite{lpsolve} as our simplex solver (as its {\sc matlab} interface allows us to set the initial basis).
%\footnote{
%Although there are many different simplex implementations, many of which are probably more efficient and powerful than {\sc lp\_solve}, we chose {\sc lp\_solve} because its {\sc matlab} %interface has the functionality that allows us to set the initial basis. To the best of our knowledge, this is the only such open source simplex solver.}

\paragraph{Initial basis for the simplex method.}
Assume we terminate the perturbed algorithm Algorithm~\ref{alg:perAlg} at the $k^{th}$ iteration, with the predicted active set $\Actv^{k}$. To generate an initial basis $\Basis$ from $\Actv^{k}$, we first obtain all independent columns in $A_{\Iactv^{k}}$. If this submatrix is not of rank $m$, we choose a column from $A_{\Actv^{k}}$ and append it to the submatrix provided it is independent of existing columns in the submatrix.
The order in which columns are added back in is decided by dual information, namely we keep trying a series of columns $\left\{ A_{i_{t}}\right\}$, where $i_{t} \in \Actv^{k}$ and $s^{k}_{i_{1}} \leq s^{k}_{i_{2}} \leq \ldots \leq s^{k}_{i_{| \Actv^{k}|}}$, until a full rank square matrix is obtained.
Since $A$ is full row rank\footnote{%
%A matrix A can always be reduced to a full row rank matrix~\cite[Page 31-32]{wright}. 
In our tests, we apply the preprocessing code from \href{http://www.caam.rice.edu/~zhang/lipsol/}{{\sc lipsol}}~\cite{Zhang96} to ensure that $A$ is full row rank.},
this procedure is finite. A similar approach has been used in~\cite[Section 7]{Tone} to form a basis of $A$.

To conduct the tests, we first choose a threshold $\mup^{\text{cap}}$, run Algorithm~\ref{alg:perAlg}, terminate the algorithm when $\mup^{k} < \mup^{\text{cap}}$, record the number of interior point iterations, say $K$, generate an initial basis $\Basis$ by the above procedure and finally start the simplex solver {\sc lp\_solve} from the initial basis $\Basis$. For comparison purposes we perform exactly $K$ iterations of Algorithm~\ref{alg:unperAlg},
and generate a new basis for~\eqref{eqv:originalPD} by the same procedure, without constraining the value of $\mu^{k}$. All tests in this part are run with $\mup^{\text{cap}} = 10^{-3}$.

We compare the number of simplex iterations used to get an optimal solution after crossover from Algorithms~\ref{alg:perAlg} and~\ref{alg:unperAlg}, visualising the results via a relative performance profile~\cite{Morales}. Namely, we consider the following relative iteration count,
\begin{equation}
\label{eqv:rel_performance}
	\mbox{rl}_{i} = -\log_{2} \frac{\mbox{Iter}^{p}_{i}}{\mbox{Iter}^{0}_{i}},
\end{equation}
where $i$ stands for the $i^{th}$ problem, the numerator stands for the number of simplex iterations performed after Algorithm~\ref{alg:perAlg} and the denominator measures the same but after Algorithm~\ref{alg:unperAlg}. If, for problem $i$, Algorithm~\ref{alg:perAlg} uses fewer simplex iterations, we get a positive valued bar with $\mbox{height} = \mbox{rl}_{i} $. If Algorithm~\ref{alg:unperAlg} wins, we obtain a negative valued bar with height defined as $-\mbox{rl}_{i} $. The value of the bar will be 0 if these two yield the same simplex iterations or {\sc lp\_solve} fails for both algorithms. If {\sc lp\_solve} fails to solve problem $i$ for Algorithm~\ref{alg:perAlg}, we have a negative valued bar with height of $\max_{i} \left( | \mbox{rl}_{i} | \right)$, otherwise a positive valued bar with the same height. It is clear that the winner outperforms the loser by $2^{ | \mbox{rl}_{i} | }$ times and one algorithm outperforms the other by having more bars (or larger area of bars) in its direction. 

\paragraph{Crossover to simplex for randomly generated test problems ({\tspndeg} and {\tspddeg}).}
In Figures~\ref{fig-crossover-nondegen} and~\ref{fig-crossover-degen}, we show the profiles for {\tspndeg} (left) and {\tspddeg} (right), with bars sorted from largest to smallest in height.
We can see that, counting the number of simplex iterations after each algorithm, the performance of Algorithm~\ref{alg:perAlg} dominates that of Algorithm~\ref{alg:unperAlg} in both cases. 

\begin{figure}[htd]
\begin{minipage}[t]{0.48\linewidth}
\centering
\includegraphics[width=\textwidth]{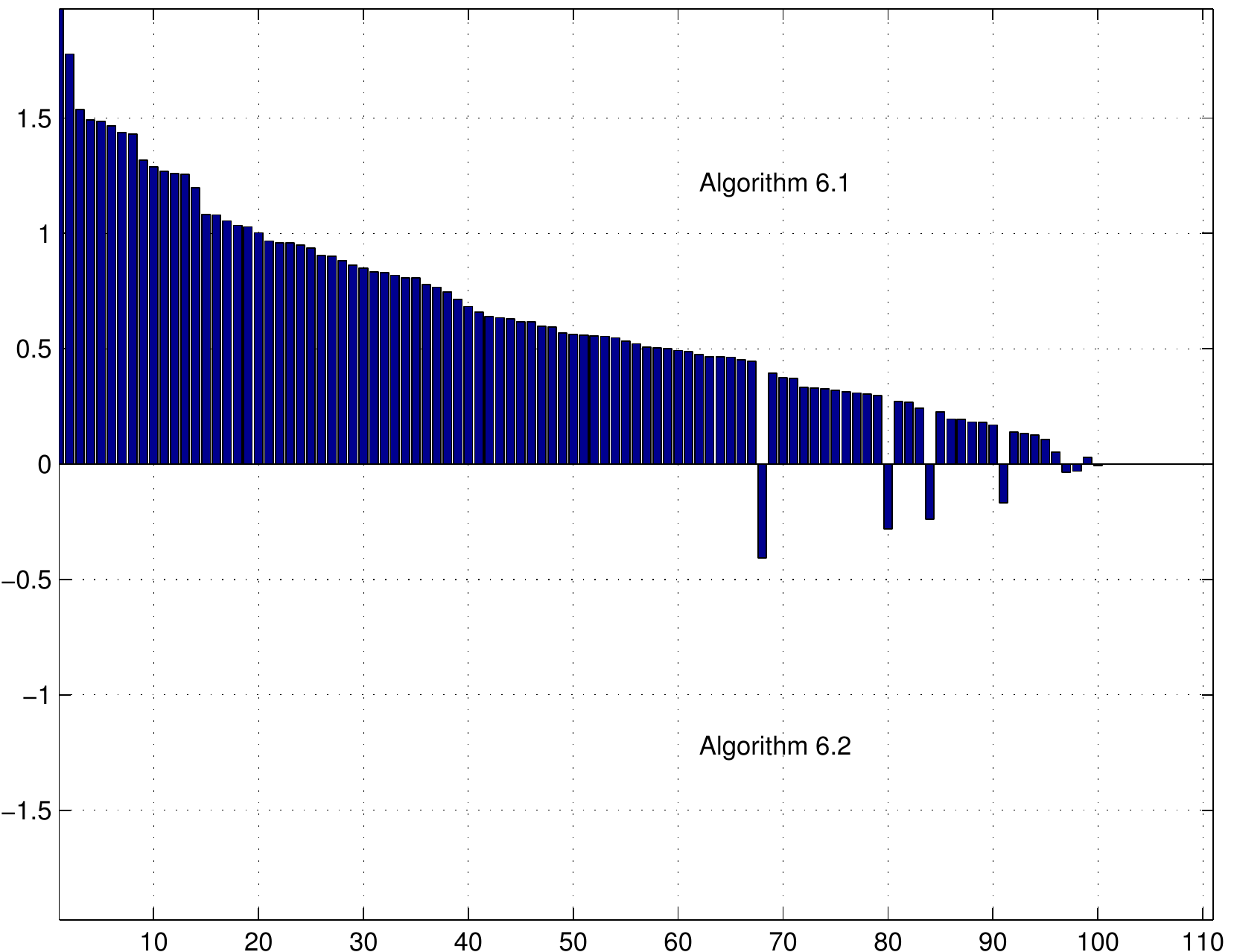}
\caption{\small Simplex iteration count for randomly generated problems}
\label{fig-crossover-nondegen}
\end{minipage}
\hspace{0.2cm}
\begin{minipage}[t]{0.48\linewidth}
\centering
\includegraphics[width=\textwidth]{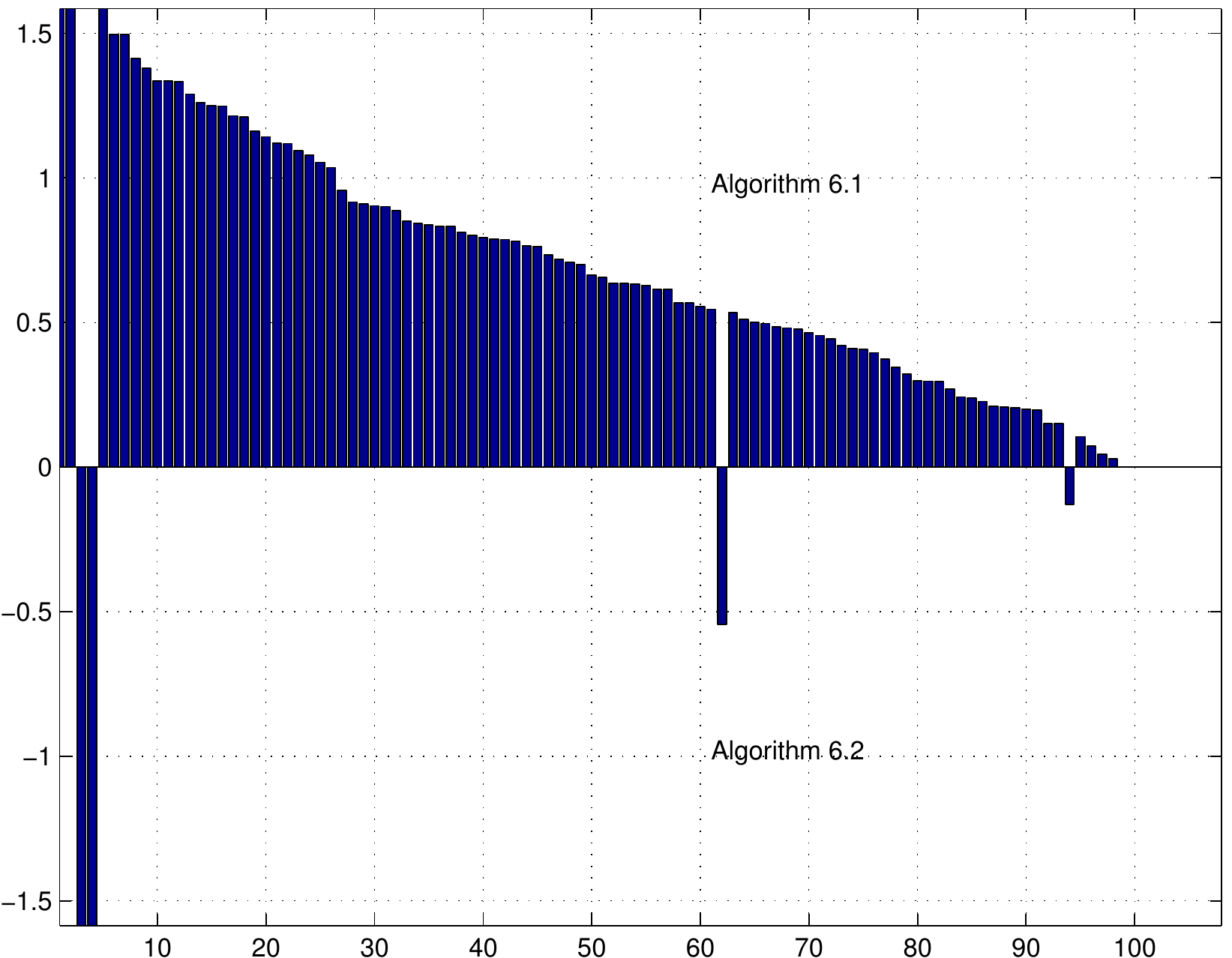}
\caption{\small Simplex iteration count for randomly generated primal-dual degenerate problems}
\label{fig-crossover-degen}
\end{minipage}
\end{figure}

In Table~\ref{tab-crossover-random}, we show the average number of simplex iterations, the average {\ipm} iterations and the average $\mup^{k}$ and $\mu^{k}$ when we terminate Algorithms~\ref{alg:perAlg} and~\ref{alg:unperAlg} for both test sets ({\tspndeg} and {\tspddeg}). On average, using perturbations saves about $34\%$ simplex iterations for the test case~{\tspndeg} and about $37\%$ for~{\tspddeg}. Due to our experimental setup, the number of {\ipm} iterations are the same for Algorithms~\ref{alg:perAlg} and~\ref{alg:unperAlg}, and the average final $\mup^{k}$ and $\mu^{k}$ before crossover are of order $10^{-4}$.\footnote{%
The definition of $\mup^{k}$ and $\mu^{k}$ in Algorithms~\ref{alg:perAlg} and~\ref{alg:unperAlg}, respectively, as well as the choice of $(x^{0},s^{0})$ to be identical for~\eqref{eqv:perPD} and~\eqref{eqv:originalPD}, imply that $\mup^{0} > \mu^{0}$, with the difference being essentially dictated by the level of perturbations $(\lambda^{0}, \phi^{0})$. Thus we are not making it any easier for Algorithm~\ref{alg:perAlg} compared to Algorithm~\ref{alg:unperAlg} in the choice of starting point. }
\begin{table}[htb]
\centering
\caption{Crossover to simplex when $\mup^{k}< 10^{-3}$ for random problems.}
\label{tab-crossover-random}
\scalebox{0.88}{
\begin{tabular}{r p{0.16\textwidth} p{0.16\textwidth} p{0.16\textwidth} p{0.16\textwidth} p{0.16\textwidth} }
	\multicolumn{1}{r}{} & \multicolumn{2}{c}{Primal nondegenerate ({\tspndeg})} & \multicolumn{2}{c}{PD degenerate ({\tspddeg})} \\
	\multicolumn{1}{r}{} & \multicolumn{1}{c}{Algorithm~\ref{alg:perAlg}} & Algorithm~\ref{alg:unperAlg} & \multicolumn{1}{c}{Algorithm~\ref{alg:perAlg}}	& \multicolumn{1}{c}{Algorithm~\ref{alg:unperAlg} } \\
	\cline{2-5}
	\multicolumn{1}{r|}{Avg simplex iterations}   	& \multicolumn{1}{c}{287}  & \multicolumn{1}{c|}{436}   &  \multicolumn{1}{c}{ 292}  & \multicolumn{1}{c|}{ 464} \\
	\multicolumn{1}{r|}{Avg {\sc ipm} iterations}  	& \multicolumn{1}{c}{10}  & \multicolumn{1}{c|}{10}   &  \multicolumn{1}{c}{10}  &  \multicolumn{1}{c|}{10 }  \\
	\multicolumn{1}{r|}{Avg $\mup^{k}$ and $\mu^{k}$ when crossover} & \multicolumn{1}{c}{$7.33\times 10^{-4}$}  & \multicolumn{1}{c|}{ $6.80\times 10^{-4}$}   &  \multicolumn{1}{c}{ $7.53\times 10^{-4}$}  &  \multicolumn{1}{c|}{$7.14\times 10^{-4}$ } \\
	\cline{2-5}
\end{tabular} }
\end{table}

We  also tracked the difference between the initial bases generated from Algorithms~\ref{alg:perAlg} and~\ref{alg:unperAlg}. We use relative difference\footnote{\label{fot-relative_basis__diff}%
The number of elements in either basis generated from Algorithms~\ref{alg:perAlg} or~\ref{alg:unperAlg} but not both divided by the cardinality of the union of two bases.
}
to measure the degree of difference between two bases. On average the relative difference is over $60\%$, and over $90\%$ of the test problems have greater than $50\%$ relative difference.
Thus our preliminary numerical experiments illustrate that using perturbations is likely to improve the efficiency when crossing over to simplex.

\paragraph{Netlib test problems ({\tsn}).}
The good prediction performance of the perturbed algorithm is not only obtained for randomly generated problems, but also for the subset of Netlib problems ({\tsn}). Here we add an additional termination criterion, namely we terminate both algorithms when $\mup^{k}$ and $\mu^{k}$ are less than $10^{-3}$ or when the relative residual~\eqref{eqv:rel_resid} is less than $10^{-6}$, whichever occurs first\footnote{\label{fot-large_b}%
This is because some problems have very large components in the right hand side $b$ with $\max(b) > 10^{3}$. For these problems, even when $\mup^{k} >10^{-3}$, the relative residual may already be less than $10^{-6}$ and this causes numerical problems when trying to decrease $\mup^{k}$ further. There are five problems of this kind, {\sc agg3}, {\sc forplan}, {\sc grow7}, {\sc israel} and {\sc share1b}, and we have marked those problems by $*$ in Table~\ref{tab-sumCrossover}. A possible remedy may be to consider using `scaled' perturbations in Algorithm~\ref{alg:perAlg}, namely, to set the perturbations to some percentage deviation for each component of the right-hand side $b$. 
}.

Figure~\ref{fig-crossover-netlib} presents the relative performance profile generated the same way as for the random tests (see~\eqref{eqv:rel_performance} and accompanying explanation). From this figure, we can see that for over half of the test problems, Algorithm~\ref{alg:perAlg} outperforms Algorithm~\ref{alg:unperAlg} by over $1.5$ times. Algorithm~\ref{alg:perAlg} `loses' for only 7 problems. 
\begin{figure}[htb]
\centering
\includegraphics[width=0.50\textwidth]{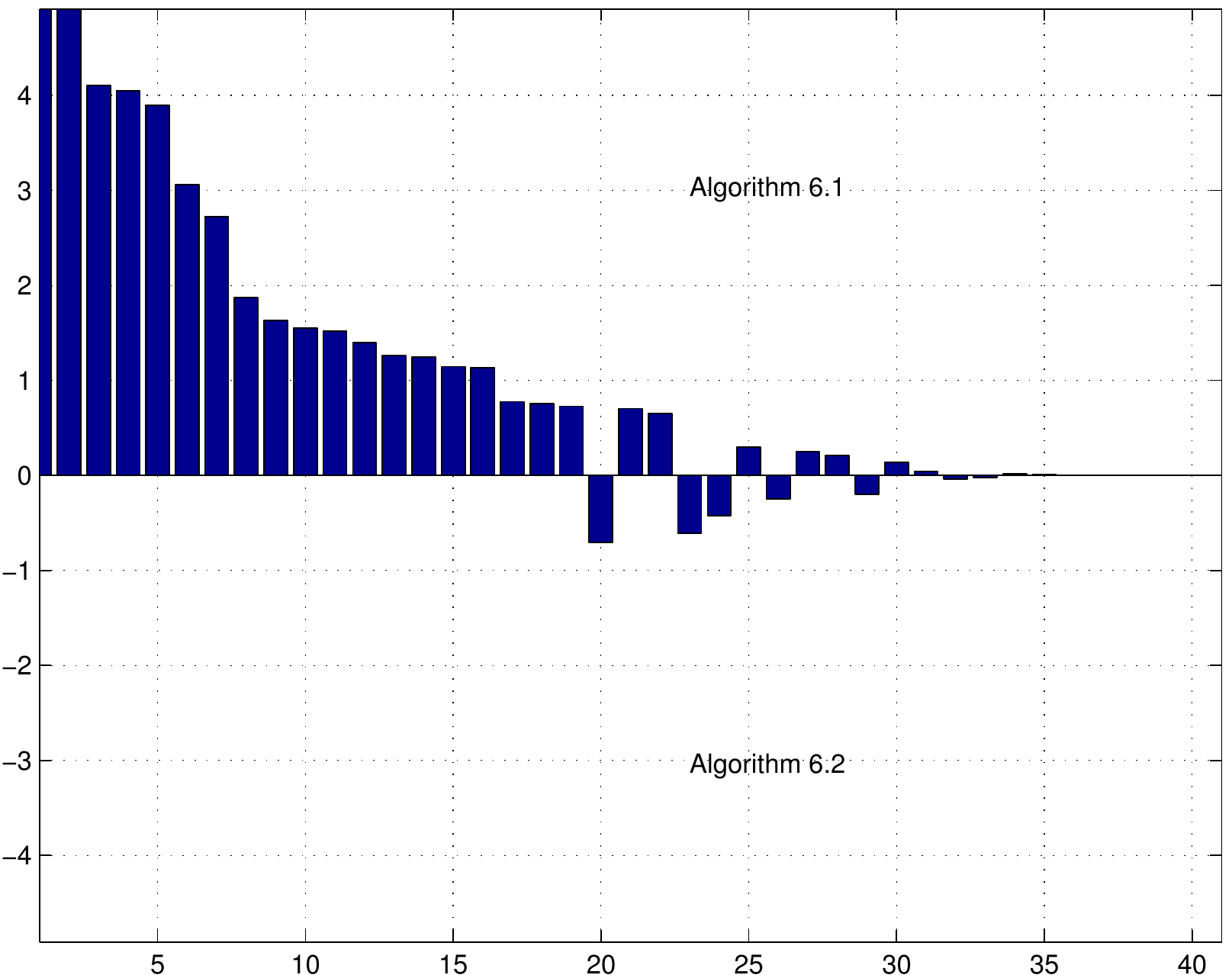}
\caption{Crossover to simplex for 37 Netlib problems}
\label{fig-crossover-netlib}
\end{figure}

We also summarise the results in Table~\ref{tab-crossover-netlib}. On average, we save about 38\% simplex iterations by applying perturbations.  The average numbers in the table exclude the data for {\sc ship08s}, since {\sc lp\_solve} fails to solve it when we do not apply perturbations.

\begin{table}[htb]
\centering
\caption{Crossover to simplex when $\mup^{k}< 10^{-3}$ for 37 Netlib problems (\tsn). }
\label{tab-crossover-netlib}
\begin{tabular}{r c c}
	\multicolumn{1}{r}{}  & \multicolumn{1}{c}{Algorithm~\ref{alg:perAlg}} & \multicolumn{1}{c}{ Algorithm~\ref{alg:unperAlg} } \\
	\cline{2-3}
	\multicolumn{1}{r|}{Avg simplex iterations}    	& \multicolumn{1}{c}{358}  & \multicolumn{1}{c|}{ 612 }    \\
	\multicolumn{1}{r|}{Avg {\sc ipm} iterations}  	& \multicolumn{1}{c}{22 }  & \multicolumn{1}{c|}{22 }    \\
	\cline{2-3}
\end{tabular}
\end{table} 
We do not give the average value of $\mup^{k}$ in the table as it is more involved than for random problems. In particular, for the problems with very large component in $b$ (problems marked by * in Table~\ref{tab-sumCrossover} ), the value of $\mup^{k}$ is greater than $10^{-3}$ for both Algorithms~\ref{alg:perAlg} and~\ref{alg:unperAlg}. There are 8 additional problems, including {\sc 25fv47}, {\sc bnl1}, {\sc brandy}, {\sc kb2}, {\sc scfxm2}, {\sc scrs8}, {\sc scatp1} and {\sc stair}, for which the value of $\mup^{k}$ is less than $10^{-3}$ only when we apply perturbations. This seems to imply that using perturbations can somehow accelerate the interior point method procedure or yield better conditioning. Except for these particular problems, the average value of $\mup^{k}$ is of order $10^{-4}$. For detailed data, see Table~\ref{tab-sumCrossover}. 

As for randomly generated problems, we also tracked and compared the differences between initial bases obtained from Algorithms~\ref{alg:perAlg} and~\ref{alg:unperAlg}. We use the same relative difference measure (see Footnote~\ref{fot-relative_basis__diff}). 
The average difference is about $40\%$, but there are 9 problems\footnote{%
{\sc afiro, agg3, grow7, isreal, sc50b, scfxm2, scfxm3, seba} and {\sc stocfor2}.
} with relative difference less than $10\%$. Algorithm~\ref{alg:perAlg} is no better than Algorithm~\ref{alg:unperAlg} for these problems. Generally, for small problems with small relative differences between bases, the simplex iterations are quite similar; for large problems, even small relative difference can yield quite different simplex iterations (such as for {\sc seba} and {\sc stocfor2}). The disappointing small relative difference of initial bases may be the result of inappropriate initial perturbations, improper shrinking speed of perturbations or ill-conditioning. 

%%%%%%%%%%% Section: Conclusion %%%%%%%%%%%%%%

\section{Conclusions and future work}
%\resetcounters

We have proposed the use of controlled perturbations for improving active set prediction capabilities of
 {\ipms} for {\lp}. The perturbations are chosen so as to slightly enlarge the feasible set
in the hope that the central path of the perturbed problems passes through or close to the original solution set when the perturbed barrier parameter is not too small.
 Our approach solves a (sequence of) perturbed problems
 using a standard primal-dual path-following method
 and predicts
using cut-off, the optimal active set of the original problem on the way. We have provided 
 theoretical and preliminary numerical evidence that this approach to active-set
prediction for {\ipms} looks promising in that the perturbed problems are not being solved to high accuracy before the original optimal active set can be accurately
predicted and that the perturbations help with the accuracy and speed
of the activity prediction for the original solution set. 
%Note that as we are applying a standard
%{\ipm} to a suite of perturbed problems, our approach maintains
%polynomial complexity provided the number of times we shrink the
%perturbations is (polynomially) finite \footnote{We can think of
 % Algorithm~\ref{alg:perAlg} as a standard primal-dual path-following
  %{\ipm} applied to solving (to some  accuracy) a sequence
  %of LP problems \eqref{eqv:perPD}, with different and shrinking vectors of perturbations on the
  %primal and dual variables. Thus if the number of iterations of Algorithm~\ref{alg:perAlg} in which
  %we shrink the perturbations $(\lambda,\phi)$ (which could be in each
  %iteration) is finite and polynomial in problem dimensions, then 
%Simple safeguards can
 % be included in Algorithm~\ref{alg:perAlg}  }.

There are several issues remaining for full validation of the proposed approach; such as the choice of the initial perturbations which we currently set to a fixed
small value that we then adjust, but that may be more suitably set to some problem-dependent value. At present, we have used 
cut-off to predict the original optimal active set when solving the perturbed problems~\eqref{eqv:perPD}; we plan to explore other suitable techniques for
the prediction such as the {\it identification function} proposed originally for nonlinear programming \cite{Facchinei}. Note that indicators \cite{bakry}
are not suitable for our purposes as they can only predict the perturbed optimal active-set when calculated in the context of an {\ipm} applied to~\eqref{eqv:perPD}.
Finally, a large-scale implementation and testing of the perturbed
approach and prediction is needed to complete our numerical
experiments for {\lp}. From a theoretical point of view, it would be
important to show polynomial complexity of a safeguarded (say, long-step \cite{wright}) variant of
Algorithm~\ref{alg:perAlg}; this seems achievable since one can think of
Algorithm~\ref{alg:perAlg} as a standard primal-dual path-following
  {\ipm} applied to solving (to some  accuracy) a sequence
  of {\lp} problems \eqref{eqv:perPD}, and so each  solve of a
  \eqref{eqv:perPD} 
  could be shown to have polynomial complexity.

There are several interesting/important areas
that may benefit from the application of the controlled perturbations approach for active-set prediction
for {\ipm}s. For example, it may prove useful for improving {\it warmstarting capabilities}~\cite{Engau2010,Skajaa2013} of {\sc
  ipm}s. Furthermore, it could be applied to the
{\it Homogeneous Self-Dual ({\hsd}) embedding model} \cite{yehsd1994,xu1996simplified} for an {\lp},  which is a very useful re-formulation
that allows
assessing whether  the given {\lp} problem has a solution, as
well as finding this solution,  depending on whether some auxiliary
variables are active or inactive at the solution of the {\hsd} problem.
Thus early and accurate activity prediction for {\hsd} models could save significant
computational effort especially if the original {\lp} problem is
infeasible/unbounded. Note that the {\hsd} model needs to solve a 
monotone {\lcp} problem and so one would need to first extend our
controlled perturbations approach to such problems, which seems
plausible. We have not explored the use of our active-set prediction
strategy in conjunction with {\it iterative linear solvers for} {\ipm}s
\cite{Gondzio25years}; predicting the active constraints sooner
may help alleviate the powerful effect that the ill-conditioning (of the {\ipm}
subproblem) has on the performance of iterative solvers for this subproblem
close to the optimal solution of the {\lp}.
Extending our activity prediction proposal
to the {\it convex quadratic programming} case is another potential future research direction.

%\paragraph*{\footnotesize Acknowledgements.} {\footnotesize
\begin{acknowledgements} 
%The second author of this paper was supported by the Principal's Career Development Scholarship from the University of Edinburgh. 
We are
  grateful to Nick Gould for useful discussions and insights. We also
  thank three anonymous referees for instructive comments that have
  improved the quality of the paper, and the Mathematical Institute, University of Oxford, for hosting the second author during
the completion of this work.
\end{acknowledgements} 

\bibliography{aiipm_bibtex}
\bibliographystyle{spmpsci} 

\appendix
\normalsize
\section{Proof of Lemma~\ref{lem-bound_x_s_strict_fea}}
\label{apd-proof_lemma_4_1}
%\resetcounters
%\renewcommand{\theequation}{A.\arabic{equation}}
%\renewcommand{\thelemma}{A.\arabic{lemma}}
%\renewcommand{\thetheorem}{A.\arabic{theorem}}

An error bound for an optimization problem bounds the distance from a given point to the solution set in terms of a residual function~\cite{pang}. 
In this section, we first formulate an {\lp} problem as a monotone Linear Complementarity Problem ({\lcp}) and then apply a global error bound for the monotone {\lcp} to the reformulated {\lp} problem in order to derive an error bound for the {\lp}, and so prove Lemma~\ref{lem-bound_x_s_strict_fea}. 

By setting $s = c-A^{\top}y$ and $y=y^{+}-y^{-}$, where $y^{+}=\max(y,0)$ and $y^{-}=-\min(y,0)$, the first order optimality conditions~\eqref{eqv:kkt_perturbed} with $\lambda = 0$ for~\eqref{eqv:originalPD} can be reformulated as

\begin{equation}
\label{eqv:kkt_modified}
\begin{array}{r}
\displaystyle  Ax-b \geq 0,\,\, -Ax+b \geq 0,\\
\displaystyle  c- A^{\top}y^{+}+A^{\top}y^{-} \geq 0, \\
\displaystyle  x^{T}(c- A^{\top}y^{+}+A^{\top}y^{-}) = 0,\\
\displaystyle  x \geq 0,\,\, y^{+} \geq 0,\,\, y^{-} \geq 0.\\
\end{array}
\end{equation}

Let
\begin{equation}
\label{eqv:LCP_Mzq}
 M=
 \begin{bmatrix}
 0& -A^{T} & A^{T}\\
 A& 0 & 0\\
 -A& 0 & 0
\end{bmatrix},
\quad
z=
\begin{bmatrix}
x\\
y^{+}\\
y^{-}
\end{bmatrix} 
\quad\mbox{and}\quad
q=
\begin{bmatrix}
 c\\
-b\\
 b
\end{bmatrix},
\end{equation}
where $A$, $b$ and $c$ are~\eqref{eqv:originalPD} problem data and $\pntp \in \Real^{n} \times \Real^{m} \times \Real^{n}$.
Then finding a solution of~\eqref{eqv:kkt_modified} is equivalent to solving the following problem,
\begin{equation}
\label{eqv:lcp}
\quad Mz+q \geq 0, \quad z \geq 0, \quad z^{T}(Mz+q)=0,\quad
\end{equation}
where $M$, $q$ and $z$ are defined in~\eqref{eqv:LCP_Mzq}, and $z$ is considered to be the vector of variables.

\begin{lemma}
\label{lem-sol}
\eqref{eqv:originalPD} is equivalent to the {\lcp} in~\eqref{eqv:lcp} with $M$ and $q$ defined in~\eqref{eqv:LCP_Mzq}, namely,
\begin{enumerate}
\item If $(x,y^{+},y^{-})$ is a solution of the {\lcp}~\eqref{eqv:lcp}, then $\pntp$ 
is a~\eqref{eqv:originalPD} solution, where $y=y^{+}-y^{-}$ and $s=c-A^{T}y$.
\item If $\pntp$ is a~\eqref{eqv:originalPD} solution, then $(x,y^{+},y^{-})$ is a solution of the {\lcp}~\eqref{eqv:lcp}.
\end{enumerate}
\end{lemma}

Next we show that our {\sc lp} problem can be viewed as a monotone {\sc lcp}~\cite{cottle}.
\begin{lemma}
\label{lem-M_psd}
The matrix $M$, defined in~\eqref{eqv:LCP_Mzq}, is positive semidefinite, and so~\eqref{eqv:lcp} is a monotone {\lcp}.
\end{lemma}
\begin{proof}
For all $v=(v_{1},v_{2},v_{3})$, where $v_{1} \in \Real^{n}, v_{2} \in \Real^{m}$ and $v_{3} \in \Real^{m}$,
$v^{T} M v= v^{T}_{2}Av_{1} - v^{T}_{3}Av_{1} - v_{1}^{T}A^{T}v_{2} + v^{T}_{1}A^{T}v_{3} = 0$, since $v^{T}_{2}Av_{1} = (v^{T}_{2}Av_{1})^{T}=v^{T}_{1}A^{T}v_{2}$ and $v^{T}_{3}Av_{1}=(v^{T}_{3}Av_{1})^{T}=v^{T}_{1}A^{T}v_{3}$. 
Thus $M$ is positive semidefinite.
\qed \end{proof}

A global error bound for a monotone {\sc lcp}~\cite{Mangasarian} is given next.

\begin{lemma}[{Mangasarian and Ren~\cite[Corollary 2.2]{Mangasarian}}]
\label{lem-eb}
Let $z$ be any point away from the solution set of the monotone {\lcp}~\eqref{eqv:lcp} and $z^{*}$ be the closest solution of~\eqref{eqv:lcp} 
to $z$ under the Euclidean norm $\|\cdot\|$.
Then $r(z)+w(z)$ is a global error bound for~\eqref{eqv:lcp}, namely,
\[ 
	\| z-z^{*}\| \leq \tau (r(z)+w(z)), 
\]
where $\tau$ is some problem-dependent constant, independent of $z$ and $z^{*}$, and
\begin{equation}
\label{eqv:r_w}
\displaystyle\begin{array}{c}
r(z)=\| z-(z-Mz-q)_{+}\|\\[1ex]
\text{and}\\[1ex]
w(z)=\left\| \left( -Mz-q,-z,z^{T}(Mz+q) \right)_{+}\right\|. 
\end{array}
\end{equation}
\end{lemma}

\begin{lemma}
\label{lem-error_bound_lp}
Given the monotone {\lcp}~\eqref{eqv:lcp} with $M$ and $q$ defined
in~\eqref{eqv:LCP_Mzq}, let $(x,y^{+},y^{-})$ be any point away from
the solution set of this problem and 
$(x^{*},(y^{*})^{+},$
$(y^{*})^{-})$ be the closest solution of this {\lcp} to $(x,y^{+},y^{-})$ under the Euclidean norm $\|\cdot\|$. Then we have
\[
	\|(x,y^{+},y^{-}) - (x^{*},(y^{*})^{+},(y^{*})^{-})\| \leq \tau (r(x,y^{+},y^{-})+w(x,y^{+},y^{-})), 
\]
where $\tau$ is some problem-dependent constant, independent of $(x,y^{+},y^{-})$ and of\\$(x^{*},(y^{*})^{+},(y^{*})^{-})$, 
\begin{equation}
\label{eqv:r_lcp}
	r(x,y^{+},y^{-}) = \left\| \left(\, \cmin{x}{c-A^{T}y},\,\, \cmin{y^{+}}{Ax-b},\,\, \cmin{y^{-}}{ b-Ax )} \, \right) \right\|,
\end{equation}
and
\begin{equation}
\label{eqv:w_lcp}
w(x,y^{+},y^{-}) = \| (-(c-A^{T}y),\,\,b-Ax,\,\,Ax-b,\,\,-x,\,\, -y^{+},\,\ -y^{-},\,\,c^{T}x - b^{T}y)_{+}\|,
\end{equation}
and where $y = y^{+} - y^{-}$.
\end{lemma}
\begin{proof}
Substituting~\eqref{eqv:LCP_Mzq} into~\eqref{eqv:r_w} and noting that $ u - (u-v)_{+} = \cmin{u}{v}$ for any $u$, $v$ vectors, we have
\begin{equation*}
{\footnotesize{\begin{array}{l}
r(x,y^{+},y^{-}) \\[1ex]
= \left\| \left( x - (x - (c - A^{T}(y^{+} - y^{-})))_{+},\,\, y^{+} - (y^{+} - (Ax - b))_{+}, \,\, y^{-} - (y^{-} - ( b-Ax ))_{+} \right) \right\| , 
\end{array}}}
\end{equation*}
and
\begin{equation*}
{\footnotesize{\begin{array}{l}
 w(x,y^{+},y^{-}) \\[1ex]
= \left\| \left(-(c-A^{T} ( y^{+} - y^{-} ) ),\,\,b-Ax,\,\,Ax-b,\,\,-x,\,\, -y^{+},\,\ -y^{-},\,\,c^{T}x - b^{T} ( y^{+} - y^{-} ) \right)_{+}\right\|.
\end{array}}}
\end{equation*}
Recalling $y = y^{+} - y^{-} $, \eqref{eqv:r_lcp} and~\eqref{eqv:w_lcp} follow directly from the above equations.
\qed \end{proof}

\begin{theorem}[Error bound for {\lp}]
\label{lem_errorbounds_lp}
Let $\pntp \in \Real^{n}\times\Real^{m}\times\Real^{n}$ where $s = c - A^{T}y $. Then there exist a~\eqref{eqv:originalPD} solution $(x^{*},y^{*},s^{*})$
%\footnote{Note that we may lose the property that this is the closest solution to the given point.} 
and problem-dependent constants $\tau_{p}$ and $\tau_{d}$, independent of $\pntp$ and $\optSol$, such that
\begin{equation*}
\|x-x^{*}\| \leq \tau_{p}\left(r(x,y) + w(x,y) \right)
\quad
\text{and}
\quad
\|s-s^{*}\| \leq \tau_{d} \left(r(x,y) + w(x,y) \right),
\end{equation*}
where
\begin{equation}
\label{r_lp}
  r(x,y) = \left\| \left(\cmin{x}{s},\,\, \cmin{y^{+}}{Ax-b},\,\, \cmin{y^{-}}{-Ax + b}\right) \right\|,
\end{equation}
and
\begin{equation}
\label{w_lp}
w(x,y) = \| (-s,\,\,b-Ax,\,\,Ax-b,\,\,-x,\,\,c^{T}x - b^{T}y)_{+}\|,
\end{equation}
and where $y^{+} = \cmax{y}{0}$ and $y^{-} = -\cmin{y}{0}$.
\end{theorem}

\begin{proof}
Consider the monotone {\lcp}~\eqref{eqv:lcp} with $M$ and $q$ defined in~\eqref{eqv:LCP_Mzq} and $z=(x,y^{+},y^{-})$. Let $z^{*}=(x^{*},(y^{*})^{+},(y^{*})^{-})$ be the closest solution to $z$ in the solution set of this {\lcp}. From Lemma~\ref{lem-sol}, $(x^{*},y^{*},s^{*})$ with $y^{*}=(y^{*})^{+}-(y^{*})^{-}$ and $s^{*}=c-A^{T}y^{*}$ is a~\eqref{eqv:originalPD} solution.
(Note that we may lose the property that this is the closest solution to the given point.) 
From $( y^{+} ,y^{-} ) \geq 0$, $s = c- A^{T}y$ and Lemma~\ref{lem-error_bound_lp}, we have
\[
	\|(x,y^{+},y^{-}) - (x^{*},(y^{*})^{+},(y^{*})^{-})\| \leq \tau (r(x,s)+w(x,s)),
\]
where $r(x,s)$ and $w(x,s)$ are defined in~\eqref{r_lp} and~\eqref{w_lp}, respectively. 
This and norm properties give
\[
	\max\left(\| x-x^{*} \|, \| y^{+}- (y^{*})^{+} \|,\| y^{-}- (y^{*})^{-} \| \right)\leq \tau (r(x,s)+w(x,s)),
\]
and so letting $\tau_{p} = \tau$, we deduce 
$
\| x-x^{*} \| \leq \tau_{p} (r(x,s)+w(x,s)).
$
Since $s^{*} = c - A^{T} y^{*}$, we also have
\[
\|s - s^{*}\| \leq \|A^{T}\|\|y - y^{*}\| \leq \|A^{T}\| (  \| y^{+}- (y^{*})^{+} \| + \| y^{-}- (y^{*})^{-} \| ) 
\leq \tau_{d} (r(x,s)+w(x,s)),
\]
where $\tau_{d}= 2\tau\|A\|$.
\qed \end{proof}

\paragraph{Proof of Lemma~\ref{lem-bound_x_s_strict_fea}.}
Since $\pntp \in \sfsp$,~\eqref{eqv:strictlyFeasibleSet_per} gives $Ax=b$ and $A^{T}y+s=c$. Then the result follows directly from Theorem~\ref{lem_errorbounds_lp}. 
\hfill $\Box$

\section{Proof of Lemma~\ref{lem-hat_epsilon_g_epsilon}}
\label{sec-compare_eps}
%\renewcommand{\theequation}{B.\arabic{equation}}
%\renewcommand{\thelemma}{B.\arabic{lemma}}
%\renewcommand{\thetheorem}{B.\arabic{theorem}}
%\resetcounters

Theorem~\ref{thm-preserving_actv} shows that we are able to preserve the optimal strict complementarity partition after perturbing the problems if the original~\eqref{eqv:originalPD} has a unique and nondegenerate solution. Actually, we can take a step further and show that then~\eqref{eqv:perPD} will also have a unique and nondegenerate~solution.

\begin{theorem}
\label{thm-perturbedProbHasUniqueSol}
Assume~\eqref{asm-full_row_rank} holds and the~\eqref{eqv:originalPD} problems have a unique and nondegenerate solution $\optSol$. Let $\Actv$ and $\Iactv$ denote the corresponding optimal active and inactive sets. Then there exists $\hat{\lambda}=\hat{\lambda}(A,b,c,x^{*},s^{*})>0$ such that the perturbed problems~\eqref{eqv:perPD} with $0 \leq \|\lambda\| < \hat{\lambda}$ have a unique and nondegenerate solution and the optimal active set is the same as that of the original~\eqref{eqv:originalPD} problems.
\end{theorem}

\begin{proof}
We consider the equivalent perturbed problem~\eqref{eqv:perPD_formPQ}. From Theorem~\ref{thm-preserving_actv}, we know there exists a $\hat{\lambda}(A,b,c,x^{*},s^{*}) > 0$ such that~\eqref{eqv:perPD_formPQ} with $0 \leq \|\lambda\| < \hat{\lambda}$ has a strictly complementary solution $ \optSolp$ with the same optimal active and inactive sets $\Actv$ and $\Iactv$, namely we have
\[
\hat{p}_{\Iactv} > 0, \quad \hat{p}_{\Actv} = 0, \quad \hat{q}_{\Actv} > 0, \quad \text{and} \quad \hat{q}_{\Iactv} = 0,
\]
and also
\begin{equation}
\label{eqv:Ap_I_b}
	A_{\Iactv}\hat{p}_{\Iactv} = b_{\lambda}.
\end{equation}
Next we are about to show that $ \optSolp$ is the unique solution of~\eqref{eqv:perPD_formPQ}. Assume there exists another solution $\bar{p} \neq \hat{p}$. 
Then $(\bar{p},\hat{y},\hat{q})$ satisfies the optimality conditions~\eqref{eqv:kkt_perturbed_form2}.
From the complementarity equations (the third term) in~\eqref{eqv:kkt_perturbed_form2} and $\hat{q}_{\Actv} > 0$ we have $\bar{p}_{\Actv} = 0 = \hat{p}_{\Actv}$. Then we have
$
A_{\Iactv}\bar{p}_{\Iactv} = b_{\lambda}.
$
It follows from this and~\eqref{eqv:Ap_I_b} that
$
A_{\Iactv} ( \bar{p}_{\Iactv} - \hat{p}_{\Iactv} ) = 0.
$
As the~\eqref{eqv:originalPD} solution is unique and nondegenerate, we must have $|\Iactv| = m$ and $rank(A_{\Iactv}) = m$, namely, $A_{\Iactv}$ is nonsingular, which implies $\hat{p}_{\Iactv} = \bar{p}_{\Iactv}$. Then~\eqref{eqv:perPD_formPQ} has a unique and nondegenerate primal solution, which also implies unique and nondegenerate dual solution.
\qed \end{proof}
To prove Lemma~\ref{lem-hat_epsilon_g_epsilon}, we also need the following series of useful lemmas.

\begin{lemma}[{Farkas' Lemma~\cite[Lemma 5.1]{Bazaraa}}]
\label{lem-FarkasLemma}
One and only one of the following two systems has a solution:
\begin{eqnarray*}
\mbox{System 1:} & Tw \geq 0 & \mbox{and} \quad b^{T}w<0, \\
\mbox{System 2:} & T^{T}y=b & \mbox{and} \quad y \geq 0,
\end{eqnarray*}
where $T \in \Real^{m \times n}$, $b \in \Real^{m}$, $w \in \Real^{n}$ and $y\in \Real^{m}$.
\end{lemma}

\begin{lemma}
\label{lem-predictSoonerLemma2}
Given $i \in \{1,\ldots,n\}$, the following system
\[
\begin{dcases}
   y+Ax \geq 0\\
   x-A^{T}y \geq 0 \quad \mbox{and} \quad x_{i} - A^{T}_{i}y > 0\\
   (x,y) \geq 0
\end{dcases}
\]
always has a solution, where $A \in \Real^{m \times n}$, $x \in \Real^{n}$, $y\in \Real^{m}$ and $A_{i}$ is the $i^{th}$ column of $A$.
\end{lemma}
\begin{proof}\
Without losing generality, we can choose $i=1$. Partition $x$ and $A$ as
$ x = \left[\, x_{1} \quad \bar{x}^{T} \,\right]^{T} $ and 
$
A =
\begin{bmatrix}
A_{1} & \bar{A}
\end{bmatrix},
$
where $ \bar{x} = \left[\, x_{2} \quad \dots \quad x_{n} \,\right]^{T} $ and 
$
\bar{A} =
\begin{bmatrix}
A_{2} & \dots & A_{n}
\end{bmatrix}.
$

We need to prove the following system has a solution
\begin{equation}
\label{eqv:applyFarkasLemma1}
\left\{
\begin{array}{ccccccc}
y                     & + & A_{1}x_{1} & + & \bar{A} \bar{x} & \geq  & 0  \\
-\bar{A}^{T} y &    &                   &  + & \bar{x}             & \geq & 0  \\
y                     &    &                   &     & 			 & \geq & 0  \\
		      &    & x_{1}	         &    &	      	   	 & \geq & 0  \\
		      &	    &		         &    & \bar{x}              & \geq & 0  \\
A_{1}^{T} y     &  - & x_{1}          &    &           		 &  <    & 0
\end{array} \right..
\end{equation}

From Lemma~\ref{lem-FarkasLemma}, we know~\eqref{eqv:applyFarkasLemma1} has a solution if and only if
\begin{equation}
\label{eqv:applyFarkasLemma2}
\left\{
\begin{array}{c}
\begin{bmatrix}
I_{m} & -\bar{A} & I_{m} & 0 & 0 \\
A^{T}_{1} & 0 & 0 & 1 & 0\\
\bar{A}^{T} & I_{n-1} & 0 & 0 & I_{n-1} \\
\end{bmatrix}
\begin{bmatrix}
u_{1} \\
u_{2} \\
u_{3} \\
u_{4} \\
u_{5}
\end{bmatrix}
=
\begin{bmatrix}
A_{1} \\
-1 \\
0
\end{bmatrix}    \\
(u_{1},u_{2},u_{3},u_{4},u_{5}) \geq 0
\end{array} \right.,
\end{equation}
has no solution, where $u_{1} \in \Real^{m}$, $u_{2} \in \Real^{n-1}$, $u_{3} \in \Real^{m}$, $u_{4} \in \Real$ and $u_{5} \in \Real^{n-1}$.

Assume~\eqref{eqv:applyFarkasLemma2} has a solution $(u_{1},u_{2},u_{3},u_{4},u_{5}) \geq 0$. Then we get
\begin{subequations}
\begin{align}
u_{1} - \bar{A}u_{2} - A_{1}  = -u_{3} \leq 0,  \label{eqv:ApplyFarkasLemma-buildContradiction1}   \\
A^{T}_{1}u_{1}  = -1 - u_{4} <0, \label{eqv:ApplyFarkasLemma-buildContradiction2}   \\
\bar{A}^{T} u_{1}  = -u_{2} - u_{5} \leq 0. \label{eqv:ApplyFarkasLemma-buildContradiction3}
\end{align}
\end{subequations}

Multiplying both sides of~\eqref{eqv:ApplyFarkasLemma-buildContradiction1} by $u_{1}^{T} \geq 0$, we have
\[
u^{T}_{1}u_{1} - (\bar{A}^{T}u_{1})^{T}u_{2} - A^{T}_{1}u_{1} = -u^{T}_{1}u_{3}.
\]

From~\eqref{eqv:ApplyFarkasLemma-buildContradiction2},~\eqref{eqv:ApplyFarkasLemma-buildContradiction3} and the nonnegativity of the variables, we know
\[
u^{T}_{1}u_{1} - (\bar{A}^{T}u_{1})^{T}u_{2} - \bar{A}^{T}_{1}u_{1} > 0
\quad\text{but}\quad
-u^{T}_{1}u_{3} \leq 0.
\]
Thus~\eqref{eqv:applyFarkasLemma2} has no solution, which implies~\eqref{eqv:applyFarkasLemma1} has a solution.
\qed \end{proof}

\begin{lemma}
\label{lem-predictSoonerLemma3}
The system
\begin{equation}
\label{eqv:predictSoonerLemma3}
\begin{dcases}
     y+Ax \geq 0\\
     x-A^{T}y > 0\\
     (x,y) \geq 0
\end{dcases}
\end{equation}
always has a solution, where $A \in \Real^{m \times n}$, $x \in \Real^{n}$ and $y\in \Real^{m}$.
\end{lemma}
\begin{proof}
From Lemma~\ref{lem-predictSoonerLemma2}, we know for any $i \in \{1,\ldots,n\}$, there exists $(x^{i}, y^{i})  \geq 0$ where $x^{i} \in \Real^{n}$ and $y^{i} \in \Real^{m}$, such that
\begin{equation}
\label{eqv:for_proof_tmp}
\begin{dcases}
     y^{i}+Ax^{i} \geq 0,\\
     x^{i}-A^{T}y^{i} \geq 0 \quad \mbox{and} \quad x^{i}_{i} - A^{T}_{i}y^{i} > 0.\\
\end{dcases}
\end{equation}
Set $\hat{x} = \sum_{i=1}^{n} x^{i} \geq 0$ and $\hat{y} = \sum_{i=1}^{n}y^{i} \geq 0$.
Then from~\eqref{eqv:for_proof_tmp}, we have
\[
\hat{y} + A\hat{x}  = \sum_{i=1}^{n}y^{i}  + A(\sum_{i}^{n}x^{i}) = \sum_{i=1}^{n}(y^{i}+Ax^{i}) \geq 0,
\]
and
\[
\hat{x} - A^{T}\hat{y} = \sum_{i=1}^{n}(x^{i} - A^{T}y^{i}) =
\begin{bmatrix}
(x^{1}_{1} - A_{1}^{T}y^{1}) + (x^{2}_{1} - A_{1}^{T}y^{2}) + \cdots + (x^{n}_{1} - A_{1}^{T}y^{n}) \\
\vdots \\
(x^{1}_{n} - A_{n}^{T}y^{1}) + (x^{2}_{n} - A_{n}^{T}y^{2}) + \cdots + (x^{n}_{n} - A_{n}^{T}y^{n})
\end{bmatrix} > 0.
\]
\qed \end{proof}

\begin{lemma}
\label{lem-predictSoonerLemma4}
The system
\begin{equation}
\label{eqv:predictSoonerLemma4}
\begin{dcases}
     y+Ax > 0\\
     x-A^{T}y \geq 0\\
     (x,y) \geq 0
\end{dcases}
\end{equation}
always has a solution, where $A \in \Real^{m \times n}$, $x \in \Real^{n}$ and $y\in \Real^{m}$.
\end{lemma}
\begin{proof}
Replace A in Lemma~\ref{lem-predictSoonerLemma3} by $-A^{T}$.
\qed \end{proof}

\begin{lemma}
\label{lem-predictSoonerLemma5}
The system
\[
\begin{dcases}
     y+Ax > 0\\
     x-A^{T}y > 0\\
     (x,y) \geq 0.
\end{dcases}
\]
always has a solution, where $A \in \Real^{m \times n}$,  $x \in \Real^{n}$ and $y\in \Real^{m}$.
\end{lemma}
\begin{proof}
From Lemmas~\ref{lem-predictSoonerLemma3} and~\ref{lem-predictSoonerLemma4}, we know there exist $(\hat{x}, \hat{y})$ and $(\tilde{x},\tilde{y})$ such that~\eqref{eqv:predictSoonerLemma3} and~\eqref{eqv:predictSoonerLemma4} hold respectively. 
Set $\bar{x} = \hat{x}+\tilde{x}$ and $\bar{y} = \hat{y}+\tilde{y}$ and deduce
\[
\bar{y}+A\bar{x}  = \underbrace{(\hat{y} + A \hat{x})}_{\geq 0}+\underbrace{(\tilde{y}+A\tilde{x})}_{>0} >0
\quad \text{and} \quad
\bar{x} - A^{T}\bar{y}  = \underbrace{(\hat{x} - A^{T} \hat{y})}_{>0}+\underbrace{(\tilde{x}-A^{T}\tilde{y})}_{\geq 0} >0.
\]
\qed \end{proof}

\paragraph{Proof of Lemma~\ref{lem-hat_epsilon_g_epsilon}.}
Assume $\Actv$ and $\Iactv$ are the optimal active and inactive sets at the unique solution $\optSol$ of~\eqref{eqv:originalPD}. Then from~\eqref{eqv:eps}, we have 
\begin{equation}
\label{eqv:simple_eps_A_b_c}
 \epsilon(A,b,c) = \min\left(\min_{i \in \Iactv}(x^{*}_{i}),\min_{i \in \Actv }(s^{*}_{i})\right).
 \end{equation}
From Theorem~\ref{thm-perturbedProbHasUniqueSol}, we know there exists a $\hat{\lambda} = \hat{\lambda} (A,b,c,x^{*},s^{*})$ such that~\eqref{eqv:perPD} with $0 \leq \|\lambda\| < \hat{\lambda} $ has a unique and nondegenerate solution and $\Actv$ and $\Iactv$ are the optimal active and inactive sets.
Since $\optSolp$ defined in~\eqref{eqv:uniqueSolpyq} is a solution of~\eqref{eqv:perPD_formPQ}, $\optSollambda = ( \hat{p} - \lambda, \hat{y}, \hat{q} - \lambda)$ is a solution of~\eqref{eqv:perPD} and also unique, with $\Actv$ and $\Iactv$ being the optimal active and inactive sets.
This and~\eqref{eqv:eps_hat} give
\begin{equation}
\label{eqv:simple_A_hat_b_hat_c}
\epsilon(A,b_{\lambda},c_{\lambda}) =  \min\left(\min_{i \in \Iactv}(\hat{x}_{i} + \lambda_{i}),\min_{i \in \Actv }(\hat{s}_{i} + \lambda_{i})\right).
\end{equation}
From~\eqref{eqv:uniqueSolpyq}, recalling that $\hat{p} = \hat{x} + \lambda$ and $\hat{q} = \hat{s} + \lambda$,  we have 
\[
	( \hat{x}_{\Iactv} +\lambda_{\Iactv} ) - x^{*}_{\Iactv} =  \lambda_{\Iactv} + A^{-1}_{\Iactv}A_{\Actv}\lambda_{\Actv}
	\quad\text{and}\quad
	( \hat{s}_{\Actv} + \lambda_{\Actv} )- s^{*}_{\Actv} = \lambda_{\Actv} - (A_{\Iactv}^{-1}A_{\Actv})^{T}\lambda_{\Iactv}.
\]
This,~\eqref{eqv:simple_eps_A_b_c} and~\eqref{eqv:simple_A_hat_b_hat_c} give us that $\epsilon(A,b_{\lambda},c_{\lambda}) > \epsilon(A,b,c)$, provided 
\begin{equation}
\label{eqv:epsilon}
\begin{dcases}
     \lambda_{\Iactv}+A_{\Iactv}^{-1}A_{\Actv}\lambda_{\Actv} > 0\\
    \lambda_{\Actv} - (A_{\Iactv}^{-1}A_{\Actv})^{T} \lambda_{\Iactv} > 0\\
     \lambda_{\Iactv}, \lambda_{\Actv} \geq 0
\end{dcases}.
\end{equation}
It remains to find a solution of \eqref{eqv:epsilon} whose norm is less than $\hat{\lambda}$. 
From Lemma~\ref{lem-predictSoonerLemma5}, we know \eqref{eqv:epsilon} always has a solution, say $\bar{\lambda}$. 
Since \eqref{eqv:epsilon} is homogeneous, $\frac{\hat{\lambda}}{2 \|\bar{\lambda}\|}\bar{\lambda}$ is also a solution, and $\| \frac{\hat{\lambda}}{2 \|\bar{\lambda}\|}\bar{\lambda} \| < \hat{\lambda}$. Without losing generality, we denote this solution as $\bar{\lambda}$. Furthermore, \eqref{eqv:epsilon} holds for all $\lambda$ with $0<\lambda = \alpha \bar{\lambda}<\bar{\lambda}$ where $\alpha \in (0,1)$. \hfill $\Box$

\section{An Active-set Prediction Procedure}
\label{apd:ActivesetPredictionProcedure}
%\resetcounters
%\renewcommand{\thealgorithm}{C.\arabic{algorithm}}
Note that in Procedure~\ref{alg:actvPrediction}, $\Actv^{k}$  is the predicted active set, $\Iactv^{k}$, the predicted inactive set and $\Nactv^{k} = \{ 1,2,\ldots,n  \} \backslash \left( {\Actv}^{k} \cup {\Iactv}^{k}\right)$, the set of all undetermined indices at the $k^{th}$ iteration.
\begin{algorithm}[htb]
\floatname{algorithm}{Procedure}
\caption{An Active-set Prediction Procedure}
\label{alg:actvPrediction}
\begin{algorithmic}
	\STATE{\textbf{Initialise}: $A^{0} = \Iactv^{0} = \emptyset$ and $\Nactv^{0} = \{  1,2,\dots,n \}$.}
	\STATE{\textbf{At $k^{th}$ iteration, $k > 1$},}
	\FOR{$i = 1,\dots,n$}
		\IF{$i \in \Nactv^{k}$}
			\IF{the threshold test~\eqref{eqv:actvPredctImplmtnCriteria} is satisfied for iterations $k-1$ and $k$}
				\STATE{$\Actv^{k} = \Actv^{k} \cup \{i\}$ and $\Nactv^{k} = \Nactv^{k}\backslash\{i\}$;}
			\ELSE
				\STATE{$\Iactv^{k} = \Iactv^{k} \cup \{i\}$ and $\Nactv^{k} = \Nactv^{k}\backslash\{i\}$.}
			\ENDIF
                 \ENDIF
		\IF{$i \in \Actv^{k}$ and the threshold test is not satisfied}
			\STATE{$\Actv^{k} = \Actv^{k} \backslash \{i\}$ and $\Nactv^{k} = \Nactv^{k}\cup\{i\}$;}
		\ENDIF
		\IF{$i \in \Iactv^{k}$ and the threshold test is satisfied}
			\STATE{$\Iactv^{k} = \Iactv^{k} \backslash \{i\}$ and $\Nactv^{k} = \Nactv^{k}\cup\{i\}$.}
		\ENDIF
	\ENDFOR
\end{algorithmic}
\end{algorithm}

\section{Results for crossover to simplex on selected Netlib problems}
%\resetcounters
%\renewcommand{\thetable}{D.\arabic{table}}

From the left to the right, we give 
the name of the test problems, number of equality constraints, 
number of variables, the value of duality gap $\mup^{K}$ when we terminate the (perturbed) Algorithm~\ref{alg:perAlg}, 
the value of duality gap $\mu^{K}$ when we terminate  the (unperturbed)  Algorithm~\ref{alg:unperAlg}, 
number of {\ipm} iterations, 
the relative difference (see Footnote~\ref{fot-relative_basis__diff} on Page~\pageref{fot-relative_basis__diff}) between two bases generated from Algorithms~\ref{alg:perAlg} and~\ref{alg:unperAlg}, 
simplex iterations for Algorithm~\ref{alg:perAlg} 
and the simplex iterations for  Algorithm~\ref{alg:unperAlg}. 
Since the algorithm without perturbations is terminated at the same {\ipm} iteration as Algorithm~\ref{alg:perAlg}, we show only the number of {\ipm} iterations for the latter. Problems on which Algorithm~\ref{alg:perAlg} loses are marked in bold font. `---' means the simplex solver fails for a particular test problem.

\begin{table}[htb]
\centering
\caption{Crossover to simplex test on a selection of Netlib problems}
\label{tab-sumCrossover} 
\resizebox{\textwidth}{!}{
\begin{tabular}{| l | l | l | l | l | l | l | l | l |}
\hline
\textbf{Probs} 
    & {\textbf{m}} %
    & {\textbf{n}} %
    & {\textbf{$\mu^{K}_{\lambda}$}} %
    & {\textbf{$\mu^{K}$}} %
    & {\textbf{IPM Iter}} %
    & {\textbf{Basis Diff}} %
    & {\textbf{splxIter Per}} %
    & {\textbf{splxIter Unp}} \\ % 
  
\hline \hline
    25FV47 &  798 & 1854 &  9.38e-04 &  1.34e-03 &        35 &      0.15 &      4193 &      6951 \\ 
  ADLITTLE &   55 &  137 &  3.79e-04 &  2.23e-04 &        16 &      0.45 &        18 &       119 \\ 
     AFIRO &   27 &   51 &  3.68e-04 &  6.84e-06 &        11 &      0.07 &         9 &         9 \\ 
     AGG3* &  516 &  758 &  9.05e-02 &  6.39e-02 &        25 &      0.07 &       112 &       123 \\ 
     BLEND &   74 &  114 &  6.55e-04 &  7.21e-04 &        10 &      0.37 &        35 &        59 \\ 
      BNL1 &  632 & 1576 &  5.41e-04 &  1.96e-02 &        28 &      0.31 &      1583 &      1632 \\ 
    BRANDY &  149 &  259 &  4.83e-04 &  1.09e-03 &        18 &      0.38 &        76 &       278 \\ 
    CZPROB &  737 & 3141 &  4.00e-04 &  1.67e-04 &        56 &      0.77 &       106 &      1822 \\ 
\textbf{E226} &  220 &  469 &  6.13e-04 &  6.98e-04 &        18 &      0.54 &       428 &       319 \\ 
     FIT1D & 1050 & 2075 &  4.81e-04 &  2.00e-04 &        22 &      0.39 &        53 &       787 \\ 
     FIT1P & 1026 & 2076 &  5.55e-04 &  4.04e-04 &        20 &      0.36 &       259 &       760 \\ 
  FORPLAN* &  157 &  485 &  4.67e-03 &  1.33e-02 &        29 &      0.45 &       119 &       341 \\ 
\textbf{GROW7*} &  420 &  581 &  4.56e-02 &  5.56e-02 &        15 &      0.06 &       226 &       190 \\ 
\textbf{ISRAEL*} &  174 &  316 &  5.39e-02 &  1.87e-02 &        32 &      0.01 &       164 &       143 \\ 
\textbf{KB2} &   52 &   77 &  3.83e-04 &  1.15e-02 &        21 &      0.24 &        44 &        27 \\ 
     SC50A &   49 &   77 &  1.64e-04 &  6.42e-05 &        10 &      0.12 &        22 &        27 \\ 
     SC50B &   48 &   76 &  5.37e-04 &  1.59e-04 &         8 &      0.00 &        37 &        37 \\ 
    SCAGR7 &  129 &  185 &  2.11e-04 &  2.80e-04 &        18 &      0.43 &        21 &        65 \\ 
    SCFXM1 &  322 &  592 &  6.19e-04 &  4.35e-04 &        24 &      0.38 &       188 &       413 \\ 
\textbf{SCFXM2} &  644 & 1184 &  5.48e-04 &  1.04e-03 &        27 &      0.02 &       690 &       672 \\ 
    SCFXM3 &  966 & 1776 &  8.73e-04 &  8.77e-04 &        28 &      0.01 &      1062 &      1074 \\ 
\textbf{SCRS8} &  485 & 1270 &  7.65e-04 &  1.42e-03 &        29 &      0.39 &       320 &       315 \\ 
     SCSD1 &   77 &  760 &  5.54e-04 &  5.54e-04 &         7 &      0.95 &       125 &       214 \\ 
     SCSD6 &  147 & 1350 &  5.86e-04 &  5.91e-04 &         8 &      0.94 &       346 &       411 \\ 
     SCSD8 &  397 & 2750 &  4.88e-04 &  5.10e-04 &        11 &      0.90 &       366 &       965 \\ 
    SCTAP1 &  300 &  660 &  5.31e-04 &  3.64e-03 &        19 &      0.31 &       114 &       179 \\ 
    SCTAP2 & 1090 & 2500 &  6.81e-04 &  2.30e-07 &        21 &      0.43 &       145 &       344 \\ 
    SCTAP3 & 1480 & 3340 &  7.81e-04 &  1.11e-07 &        22 &      0.48 &        54 &       451 \\ 
      SEBA & 1029 & 1550 &  5.92e-04 &  3.09e-04 &        23 &      0.04 &        43 &        70 \\ 
  SHARE1B* &  112 &  248 &  2.87e-03 &  8.18e-02 &        27 &      0.24 &       176 &       204 \\ 
   SHARE2B &   96 &  162 &  3.19e-04 &  3.90e-04 &        14 &      0.41 &        57 &       126 \\ 
   SHIP04L &  356 & 2162 &  6.55e-04 &  2.92e-04 &        27 &      0.61 &        13 &       215 \\ 
   SHIP08L &  688 & 4339 &  6.55e-04 &  5.41e-04 &        29 &      0.81 &       441 &      1056 \\ 
   SHIP08S &  416 & 2171 &  6.34e-04 &  3.56e-04 &        26 &      0.76 &        70 &       --- \\ 
   SHIP12S &  466 & 2293 &  2.12e-04 &  2.74e-05 &        33 &      0.71 &        18 &       541 \\ 
     STAIR &  362 &  544 &  6.56e-04 &  1.11e-02 &        16 &      0.29 &       292 &       294 \\ 
\textbf{STOCFOR2} & 2157 & 3045 &  5.63e-04 &  4.74e-05 &        39 &      0.08 &      1213 &       796 \\ 
\hline\hline
\end{tabular} }
\end{table}
\end{document}